\documentclass[final,3p,times]{elsarticle}

\usepackage{amssymb,amsmath,amsthm}
\usepackage{graphicx}
\usepackage{algorithm}
\usepackage{algpseudocode}

\usepackage{subcaption}
\usepackage{hyperref}
\usepackage{booktabs}

\newcommand{\eps}{\varepsilon}
\newcommand{\thetarec}{\theta_{\mathrm{rec}}}
\newcommand{\thetaadv}{\theta_{\mathrm{adv}}}

\newtheorem{proposition}{Proposition}
\newtheorem{theorem}{Theorem}
\newtheorem{lemma}{Lemma}
\newtheorem{remark}{Remark}

\usepackage{mathptmx}
\usepackage{pgfplots}
\pgfplotsset{compat=1.18}

\usepackage{tikz}
\usetikzlibrary{angles,calc}
\usetikzlibrary{decorations.pathmorphing}

\def\Xint#1{\mathchoice
{\XXint\displaystyle\textstyle{#1}}%
{\XXint\textstyle\scriptstyle{#1}}%
{\XXint\scriptstyle\scriptscriptstyle{#1}}%
{\XXint\scriptscriptstyle\scriptscriptstyle{#1}}%
\!\int}
\def\XXint#1#2#3{{\setbox0=\hbox{$#1{#2#3}{\int}$ }
\vcenter{\hbox{$#2#3$ }}\kern-.6\wd0}}
\def\dashint{\Xint-}

\makeatletter
\def\ps@pprintTitle{%
  \let\@oddhead\@empty
  \let\@evenhead\@empty
  \def\@oddfoot{}%
  \let\@evenfoot\@oddfoot}
\makeatother

\begin{document}

\begin{frontmatter}

\title{\texorpdfstring
  {A domain decomposition method for the directional contact \\ angle hysteresis interval on doubly periodic rough surfaces}
  {A domain decomposition method for the directional contact angle hysteresis interval on doubly periodic rough surfaces}}
\author[UU]{Zijie Lin}
\ead{zjlin@math.utah.edu}
\author[UU]{William M Feldman\corref{cor1}}
\ead{feldman@math.utah.edu}
\author[UU]{Braxton Osting}
\ead{osting@math.utah.edu}
\cortext[cor1]{Corresponding author.}

\affiliation[UU]{%
organization={Department of Mathematics, University of Utah},
addressline={155 S. 1400 E., Rm 233},
city={Salt Lake City},
postcode={84112},
state={UT},
country={USA}}

\begin{abstract}
We study wetting on a doubly periodic rough surface in three dimensions. Although the liquid–vapor interface meets the solid at the local Young's angle, microscale roughness can cause the macroscopic apparent angle to differ substantially from this value. We formulate the directional contact angle hysteresis (CAH) interval in terms of apparent angles associated with pinned microscopic configurations. To approximate its receding and advancing endpoints, we evolve capillary mean curvature flow (CMCF) toward extremal stationary states. Computing these states is difficult because the pinning that produces hysteresis is generated at the scale of the roughness, whereas the apparent angle is only meaningful at the macroscopic scale, so a single uniform grid must resolve both. We therefore introduce a two-scale alternating (TSA) method based on a Schwarz decomposition: a Merriman--Bence--Osher (MBO) diffusion-generated scheme resolves the contact-line near region, while a linearized minimal-surface problem updates the far region. For an idealized reference iteration, we prove decay of an approximate interfacial energy. Numerical experiments on a representative doubly periodic surface show a strongly anisotropic CAH interval whose width varies by more than a factor of three with contact-line orientation and changes sharply near the diagonal directions. Because a stationary droplet must meet the solid at an apparent angle inside this interval, the computed anisotropy constrains which macroscopic wetted regions the surface can support; it is consistent with a square-like stationary droplet whose sides align with the diagonal directions.
\end{abstract}

\begin{keyword}
contact angle hysteresis \sep
capillarity \sep
threshold dynamics \sep
MBO diffusion-generated motion \sep
Schwarz alternating method \sep
domain decomposition \sep
free boundary problem

% JCP does not use MSC codes; kept here in case they are wanted elsewhere.
% \MSC[2020]
% 35R35 \sep % Free boundary problems for PDEs
% 53E10 \sep % Flows related to mean curvature
% 65N55 \sep % Multigrid methods; domain decomposition for BVPs
% 76D45 \sep % Capillarity (surface tension)
% 49Q05     % Minimal surfaces and optimization
\end{keyword}

\end{frontmatter}

\section{Introduction} \label{sec:Intro}
When a liquid drop rests on a solid surface, the angle between the liquid interface and the solid is called the (static) \emph{contact angle}.
For an ideal flat solid surface, a stationary droplet meets the solid surface at Young's angle \cite{Young_1805}. This is the contact angle condition satisfied by the globally energy-minimizing configuration.
However, real solids can be rough and/or have heterogeneous chemical patterns, which affect wettability.
On such surfaces, the contact line may become pinned, allowing static configurations to exhibit a range of contact angles.
This phenomenon is known as \emph{contact angle hysteresis} (CAH) and is characterized by an interval of allowed apparent contact angles \([\,\thetarec,\thetaadv\,]\), whose endpoints are the receding and advancing contact angles, respectively.
As a fundamental property, CAH plays a key role in various wetting phenomena.
For example, it governs imbibition of droplets on microdecorated surfaces \cite{Bico_2001, Chu_2010, Courbin_2007}.
Control of wetting on such surfaces is in turn critical for applications such as
the design of ``self-cleaning'' surfaces \cite{barthlott1997purity},
inkjet printing \cite{calvert2001inkjet}, and
bio-microarrays \cite{Yoshino_2006}.
Furthermore, CAH affects droplet evaporation modes and is used to quantify how surface micropatterns influence evaporation dynamics \cite{Kashaninejad_2020}.

Physicists and engineers are interested in controlling droplet shapes via construction of microstructured and chemically patterned surfaces \cite{Raj_2014}.
Various polygonal droplet shapes have been realized experimentally.
The occurrence of controlled anisotropic shapes of macroscopic droplets on micropatterned surfaces has been linked with anisotropy in the CAH interval \cite{kusumaatmaja2008anisotropic}.
In other words, the hysteresis interval can depend on the macroscopic orientation of the contact line relative to the underlying micropattern.

Due to its practical importance, CAH has been the subject of extensive theoretical study.
Joanny and de Gennes introduced an early theoretical model for CAH based on localized surface heterogeneities \cite{joanny1984model}.
In this model, sufficiently strong surface heterogeneities induce pinning of the contact line and give rise to CAH.
This framework has since been extended to investigate contact-line dynamics in a variety of settings, providing further insight into the mechanisms associated with CAH \cite{joanny1990motion,raphael1989dynamics}.
From an engineering perspective, thermodynamic approaches have also been used to model contact-line distortion and thereby predict contact angle hysteresis on heterogeneous and superhydrophobic surfaces \cite{raj2012unified}.

In the last two decades, both globally stable states and CAH for liquid drops on rough surfaces have been studied in the mathematical literature, and several of these effects have been established rigorously.
Homogenization of the wetting energy on rough and chemically heterogeneous surfaces yields effective macroscopic energy densities and explains the emergence of Wenzel- and Cassie-type laws \cite{Alberti_2005,CaffarelliMellet,mellet2012capillary,Feldman_2018,Xu_2010}.
Several energetic and dissipation-based models have been proposed to explain CAH and its associated stick–slip behavior \cite{desimone2007new,Caffarelli_2007}.
In partially linearized models where the free surface is graphical, the set of pinned angles and its anisotropy have been studied in detail in \cite{feldman2021limit}.
 Moreover, jump discontinuities of this interval as a function of the contact-line direction have been linked to faceting of the wetted set \cite{feldman2019free}, which is the mechanism underlying the anisotropic droplet shapes discussed above.
However, a general and rigorous mathematical characterization of CAH remains a challenge.

CAH has also been studied numerically, but resolving it on a rough surface places competing demands on the discretization.
In a numerical study using a phase-field model, stick–slip behavior and CAH were observed in three-dimensional droplets spreading on physically flat, chemically patterned surfaces \cite{Zhong_2016}.
From a computational perspective, the CAH interval can be estimated by identifying the numerical solutions corresponding to the receding and advancing configurations.
Under a volume constraint, these configurations are expected to be surfaces of constant mean curvature meeting the solid at Young's angle along the contact line.
After rescaling to the microscale near the contact line, the pressure term in the curvature equation drops out and the interface becomes a minimal surface, see the discussion below in Section~\ref{s:TheoHystInt}.
Thus, a natural mathematical way to find stationary solutions is the energy-decreasing flow of the capillary functional, where the interface follows the mean curvature flow away from the contact line while the Young's angle condition is enforced at the contact line.

An efficient computational method for simulating mean curvature flow is the Merriman--Bence--Osher (MBO) diffusion-generated method, which iteratively performs diffusion and thresholding steps \cite{MBO1993,Merriman_1994}.
An energetic formulation of this method, which generalizes well to multiphase problems, was developed in \cite{Esedoglu_2014}.
A variation of this scheme was developed by Xu et al. for the capillary problem, where the equilibrium interface satisfies the Young's angle condition near the contact line \cite{Xu_2017}, and was subsequently improved \cite{Wang_2019}.
These methods are easy to implement and computationally efficient, especially when the fast Fourier transform (FFT) can be used in the diffusion step.
However, they are formulated on a single uniform grid, and that poses a significant challenge for the problem at hand.
The pinning that produces hysteresis is generated at the scale $O(\eps)$ of the roughness, while the apparent contact angle that defines the interval is only meaningful at the $O(1)$ scale of the far field, a separation of scales that also appears in reduced models of dynamic CAH on rough surfaces \cite{Xiao_2023}. 
A uniform discretization fine enough to resolve the contact line must therefore be carried across the entire domain and quickly becomes prohibitive as $\eps \to 0$.
Related work on dynamic CAH has used multiscale expansions and averaging to derive effective boundary conditions on two-dimensional chemically inhomogeneous surfaces \cite{Zhang_2022}.
Extending such descriptions to general three-dimensional problems with rough or chemically inhomogeneous surfaces remains challenging.
Thus, a robust computational scheme for determining the CAH interval is still lacking.
Resolving both scales at once is the essential computational difficulty, and it is what motivates the domain decomposition developed in this work.

This paper makes three contributions.
First, we formulate the directional CAH interval for a doubly periodic rough surface and characterize its endpoints as the apparent angles of extremal stationary states of a capillary mean curvature flow.
Second, we introduce the two-scale alternating (TSA) method, a Schwarz decomposition that couples an MBO diffusion-generated scheme near the contact line to a linearized minimal surface in the far region, and we prove an energy-decay property for a reference form of the iteration.
Third, we compute the directional CAH interval for a doubly periodic rough surface in three dimensions and find a pronounced anisotropy, including a sharp change in the interval at certain lattice-aligned orientations.

We next formulate the capillary model and define the CAH interval considered in this work.

\subsection{Surface energy and capillary problems}
Consider a wetting problem in a domain $\Omega \subset \mathbb R^n$, $n=2,3$.
We denote the
liquid domain by $L$,
vapor domain by $V$, and
solid domain by $S$
to obtain the disjoint partition
$\Omega= L\cup V\cup S$.
Accordingly, the liquid–vapor, solid–liquid, and solid–vapor interfaces are denoted
$\Sigma_{LV} :=\partial L\cap \partial V$,
$\Sigma_{SL} :=\partial S \cap \partial L$, and
$\Sigma_{SV} := \partial S \cap \partial V$,
respectively.
The contact line is written as
$\Gamma_{SLV} := \partial S \cap \partial L \cap \partial V$.
Let $\gamma_{SL}(\mathbf{x})$,
$\gamma_{SV}(\mathbf{x})$, and
$\gamma_{LV}$
denote the solid–liquid, solid–vapor, and liquid–vapor energy densities.
The equilibrium configurations of the system are the critical points of the total interfacial energy,
\begin{equation}
\label{equ:interface_energy}
\mathcal{E}(L,V; \Omega)
 \, = \,
 \gamma_{LV} |\Sigma_{LV}|
 \ + \
 \int_{\Sigma_{SL}} \gamma_{SL}(\mathbf{x} ) \,d\mathbf{x}
 \ + \
 \int_{\Sigma_{SV}} \gamma_{SV}(\mathbf{x} )
 \, \,d\mathbf{x},
\end{equation}
subject to volume constraints and/or other side conditions.
Volume constraints on the liquid or vapor region are physically natural, and Dirichlet constraints on the domain boundary $\partial \Omega$ are often mathematically useful.

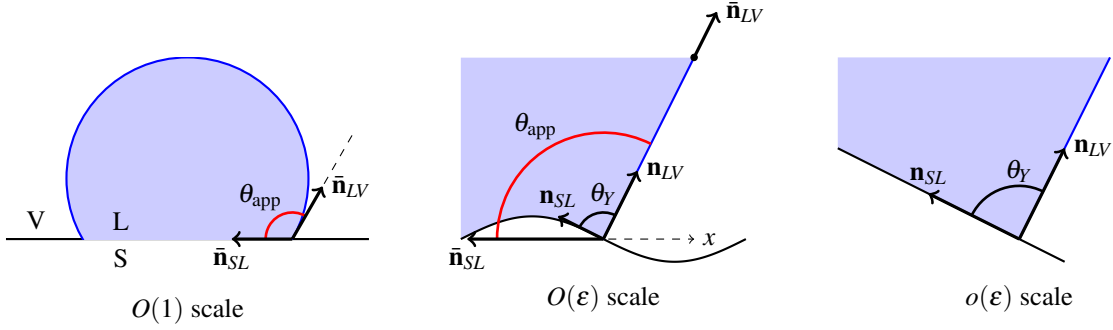
\begin{figure}[t]
\centering
\hspace{-10mm}
\begin{tikzpicture}[baseline]
%=== 1. O(1) scale droplet ===%
\begin{scope}[xshift=0cm, scale=0.8]
\draw[thick] (-3,0) -- (3,0);

% fill the droplet: circle center (0,1), radius 2, clipped to y>=0
\begin{scope}
\clip (-3,0) rectangle (3,3);          % keep only y>=0
\fill[blue!20] (0,1) circle[radius=2]; % fill circle
\draw[blue, thick] (0,1) circle[radius=2]; % outline
\end{scope}

\coordinate (C) at ({sqrt(3)},0);
\coordinate (H) at ($(C)+(-1,0)$);
\coordinate (T) at ($(C)+(1, {sqrt(3)})$);

% angle marker
\draw pic[draw=red,
line width=1.0pt, 
angle radius=10pt,
angle eccentricity=1.2]
{angle = T--C--H};
\node at (1.2, 0.7) {$\theta_{\rm app}$};
\node[below] at (0,-0.8) {$O(1)$ scale};
\draw[->,very thick] (3^0.5,0) -- (3^0.5+1/2,3^0.5/2);
\draw[->,very thick] (3^0.5,0) -- (3^0.5-1,0);
\node[right] at (3^0.5+1/2,3^0.5/2) {$\bar{\mathbf{n}}_{LV}$};
\node[below] at (3^0.5-1,0) {$\bar{\mathbf{n}}_{SL}$};
\draw[-, dashed] (3^0.5,0) -- (3^0.5+1,3^0.5);

\node[above] at (-1.1,0) {L};
\node[below] at (-1.1,0) {S};
\node[above] at (-2.5,0) {V};

\end{scope}

%=== 2. O(ε) scale straight interface + roughness ===%
\hspace{5mm} 
\begin{scope}[xshift=5cm, scale=0.6]
% Fill region: above substrate and left of y = 2x
\begin{scope}
\clip (-4, -4) rectangle (2, 4); % bounding box
\clip (-4, 0) -- (0,0) -- (2,4) -- (-4,4) -- cycle; % left of the interface
\fill[blue!20]
plot[domain=-pi:1.5*pi, samples=300] (\x, {-0.5*sin(deg(\x))}) -- (1.5*pi, 4) -- (-pi, 4) -- cycle;
\end{scope}

% Substrate: y = -sin(x)
\draw[thick, domain={-pi}:{pi}, samples=300]
plot (\x, {-0.5*sin(deg(\x))});

% Interface: y = 2x
\draw[blue, thick] (0,0) -- (2,4);

% Contact angle
\coordinate (O) at (0,0);
\coordinate (H) at (-1,0);
\coordinate (T) at (1,2);
\coordinate (C) at (-2,1);
\draw pic[draw=red,
line width=1.0pt, 
angle radius=40pt,
angle eccentricity=1.2]
{angle = T--O--H};
\node at (-1.5, 2.5) {$\theta_{\rm app}$};

\draw pic[draw=black,
line width=1.0pt, 
angle radius=10pt,
angle eccentricity=1.2]
{angle = T--O--C};
\node at (0.0, 1.0) {$\theta_{Y}$};
% x-axis
\draw[->, dashed] (-2,0) -- (2,0) node[right] {\(x\)};
\node[below] at (0,-0.8) {$O(\eps)$ scale};
\draw[->,very thick] (2,4) -- (2.5,5);
\filldraw (2,4) circle (2pt);
\node[right] at (2.5,5) {$\bar{\mathbf{n}}_{LV}$};
\draw[->,very thick] (0,0) -- (-3,0);
\node[below] at (-3,0) {$\bar{\mathbf{n}}_{SL}$};
\draw[->,very thick] (0,0) -- (-1,0.5);
\node[above] at (-1,0.5) {${\mathbf{n}}_{SL}$};
\draw[->,very thick] (0,0) -- (0.75,1.5);
\node[right] at (0.75,1.5) {${\mathbf{n}}_{LV}$};
\end{scope}

%=== 3. o(ε) scale Young's angle ===%
\hspace{5mm} 
\begin{scope}[xshift=10cm, scale=0.6]
\begin{scope}
\fill[blue!20]
(-4,2) -- (0,0) -- (2,4)-- (-4,4) -- cycle; % area above y = -x
\end{scope}

% Solid plane: y = -x
\draw[thick] (-4,2) -- (1,-1/2);

% Interface: y = 2x
\draw[blue, thick] (0,0) -- (2,4);

% Contact angle θ_Y
\coordinate (O) at (0,0);
\coordinate (S) at (-2,1);  % along y = -1/2*x
\coordinate (I) at (1,2);   % along y = 2x
\draw pic[draw=black,
line width=1.0pt,   
angle radius=20pt,
angle eccentricity=1.2]
{angle = I--O--S};

\node at (0.0, 1.5) {$\theta_Y$};

\node[below] at (0,-0.8) {$o(\eps)$ scale};
\draw[->,very thick] (0,0) -- (-2,1);
\node[above] at (-2,1) {${\mathbf{n}}_{SL}$};
\draw[->,very thick] (0,0) -- (1,2);
\node[right] at (1,2) {${\mathbf{n}}_{LV}$};
\end{scope}
\end{tikzpicture}
\vspace{-12mm}
\caption{{\bf (left)} At $O(1)$ scale, a droplet on a rough surface has constant curvature and makes an apparent angle, $\theta_{\textrm{app}}$, with the horizontal plane.
{\bf (center)} Zooming in to the $O(\eps)$ scale, we see the roughness of the surface.
Away from the solid region, the droplet has zero curvature and takes the apparent angle $\theta_{\textrm{app}}$ with the horizontal plane.
Near the solid region, the droplet makes Young's angle $\theta_Y$ with the rough surface.
In this 2D illustration, the normal and macroscopic normal, ${\mathbf{n}}_{LV}$ and $\bar{\mathbf{n}}_{LV}$, happen to be the same but could differ in the 3D case; see Figure~\ref{fig:3D_apparent}.
{\bf (right)} Zooming in to the $o(\eps)$ scale, the contact line is straight and we obtain the classical depiction for Young's angle.
The goal of this paper is to develop a computational method that solves the $O(\eps)$ scale problem.}
\label{fig:Dirichlet_condition_wetting}
\end{figure}

\paragraph{Volume constrained minimizers} A classical calculus of variations computation shows that the Euler–Lagrange equations for minimization of \eqref{equ:interface_energy} under a volume constraint $|L| = \textup{Vol}$ are the free boundary capillary problem,
\begin{equation*}
\begin{cases}
\begin{aligned}
&\kappa_{LV} = p &&\text{on } \Sigma_{LV} \\
&\theta_{SLV}(\mathbf{x}) = \theta_Y(\mathbf{x}) && \mathbf{x} \in \Gamma_{SLV}.
\end{aligned}
\end{cases}
\end{equation*}
Here, $\kappa_{LV}$ is the mean curvature of the surface $\Sigma_{LV}$ (where $\kappa_{LV} < 0$ when $L$ is convex) and
$p$ is the pressure, which is, mathematically speaking, a Lagrange multiplier for the volume constraint.
The second equation is referred to as the \emph{Young's angle condition}.
Here, $\theta_{SLV}(\mathbf{x}) := \angle(\mathbf{n}_{SL}(\mathbf{x}),\mathbf{n}_{LV}(\mathbf{x}))$
is the angle made between
$\mathbf{n}_{SL}(\mathbf{x})$, the normal vector to $\Gamma_{SLV}$ which is tangent to $\partial S$ at $\mathbf{x}$ and points inward toward the wetted region $\Sigma_{SL}$, and $\mathbf{n}_{LV}(\mathbf{x})$, the normal vector to $\Gamma_{SLV}$ which is tangent to $\Sigma_{LV}$ at $\mathbf{x}$ and points inward to that surface.
Young's angle $\theta_{Y}$, defined by the relation
$\cos(\theta_{Y}(\mathbf{x}))=\frac{\gamma_{SV}(\mathbf{x})-\gamma_{SL}(\mathbf{x})}{\gamma_{LV}}$,
arises as the Euler–Lagrange condition for the minimization of the energy \eqref{equ:interface_energy}.

\paragraph{Volume constrained minimizers with directional hysteresis}
In practice, however, the measured (large-scale) angle may differ from Young's angle due to surface roughness at much smaller scales.
It is therefore important to distinguish between the macroscopic \emph{apparent contact angle} and the microscopic contact angle.
The apparent contact angle $\theta_{\rm app} (\mathbf{x}) := \angle(\bar{\mathbf{n}}_{LV}(\mathbf{x}),\bar{\mathbf{n}}_{SL}(\mathbf{x}))$ is the angle between the macroscopic normals \(\bar{\mathbf{n}}_{LV}\) and $\bar{\mathbf{n}}_{SL}$.
These vectors are the normals to the macroscopic (i.e. roughness-averaged) contact line $\Gamma_{SLV}$,
 tangent to the macroscopic $\partial L$ and $\partial S$ respectively at $\mathbf{x}$; see Figure~\ref{fig:Dirichlet_condition_wetting}~(left).
Due to the roughness, the contact line is pinned at the microscopic scale, where the surface \(\partial S\) cannot be treated as flat, as shown in Figure~\ref{fig:Dirichlet_condition_wetting}~(center).
Consequently, even with the Young's angle condition locally satisfied, the apparent angle \(\theta_{\rm app}\) often deviates from Young's angle; see Figure~\ref{fig:Dirichlet_condition_wetting}~(center).
Indeed, if we further zoom into the \(o(\eps)\) scale, we obtain the classical depiction for Young's angle; see Figure~\ref{fig:Dirichlet_condition_wetting}~(right).

The contact angle hysteresis (CAH) interval is the collection of all apparent contact angles achieved at large scale by stationary microstates.
The CAH interval may depend on the direction of the (apparent) inward normal to the contact line.
In other words, anisotropy naturally emerges at the macroscale.
In what follows, we restrict our attention to \(\mathbb{Z}^2\)-periodic rough surfaces, for which the directional CAH is independent of the macroscopic position \(\mathbf{x}\) and depends only on the direction \(\mathbf{k}\).
In this setting, we formulate a relaxed free boundary problem with anisotropic hysteresis, where the classical contact-angle condition is replaced by the interval constraint,
\begin{equation}\label{equ:vol_minimizer_relaxed}
\begin{cases}
\begin{aligned}
&\kappa_{LV} = p &&\text{on } \Sigma_{LV} \\
&\theta_{\rm app} (\mathbf{x}) \in [\thetarec(\bar{\mathbf{n}}_{SL}(\mathbf{x})),\thetaadv(\bar{\mathbf{n}}_{SL}(\mathbf{x}))] && \mathbf{x} \in\Gamma_{SLV}.
\end{aligned}
\end{cases}
\end{equation}
The angles $\thetarec(\bar{\mathbf{n}}_{SL}(\mathbf{x}))$ and $\thetaadv(\bar{\mathbf{n}}_{SL}(\mathbf{x}))$ are referred to as the receding and advancing angles, respectively.
With a slight abuse of notation, we write the CAH interval as \([\thetarec(\bar{\mathbf{n}}_{SL}(\mathbf{x})),\thetaadv(\bar{\mathbf{n}}_{SL}(\mathbf{x}))]\); a more rigorous definition
is provided in the next section.

\subsection{The contact angle hysteresis (CAH) interval} \label{s:TheoHystInt}
For a rough surface with $O(\eps)$-scale roughness, rescaling near the contact line leads to a near-region problem;
see Figure~\ref{fig:Dirichlet_condition_wetting}~(center).
A leading-order asymptotic analysis in the rescaled coordinate \( \frac{\mathbf{x}}{\eps} \) shows that the liquid–vapor interface curvature \(\kappa_{LV}\) reduces to zero and the Young's angle condition holds at the contact line,
\begin{equation}
\label{prob:infty_channel}
\begin{cases}
\begin{aligned}
& \kappa_{LV} = 0, &&\text{on } \Sigma_{LV}, \\
&\theta_{SLV} (\mathbf{x}) = \theta_Y(\mathbf{x}), && \mathbf{x} \in \Gamma_{SLV}.
\end{aligned}
\end{cases}
\end{equation}
Note that this system is only valid for the region near the contact line. Additional boundary conditions must be imposed to ensure regularity and consistency with the far-region.

\paragraph{Far-field condition minimizers}
 For a $\mathbb{Z}^2$-periodic function $\psi$, we consider a solid domain below a rough surface given by $S:=\{\mathbf{x} = (x,y,z)\in \mathbb{R}^3 \colon z\leq \psi(x,y)\}$.
 Here and throughout, we identify $\mathbf{x} \leftrightarrow (x,y,z)$.
We define the domain
 \[
 \Pi := \{ (x,y,z)\in \mathbb{R}^3 \colon z> \psi(x,y)\}.
 \]
 We assume that the liquid–vapor interface behaves like a flat plane away from the contact line; see \ref{sec:Justification_linearization} for discussion and justification.
 For simplicity, we define the two-dimensional vector \(\mathbf{k}(\mathbf{x}):=-\bar{\mathbf{n}}_{SL}(\mathbf{x})|_{\mathbb{R}^2}\) as the apparent outward unit normal at point \(\mathbf{x}\) along the contact line, pointing from the wetted region toward the non-wetted region in the macroscopic sense.
 Here \(\mathbf{k}(\mathbf{x})\) is a unit vector: \(\bar{\mathbf{n}}_{SL}\) is tangent to the macroscopic solid surface, which for \(S=\{z\leq\psi(x,y)\}\) with \(\psi\) periodic is the plane \(\{z=0\}\), so \(\bar{\mathbf{n}}_{SL}\) is \(xy\)-planar.
 We refer to $\theta$ as a \emph{pinned angle} with respect to the \emph{apparent contact line normal} $\mathbf{k}$ if there exists a solution for \eqref{prob:infty_channel} such that the liquid region \(L\) satisfies
\begin{equation} \label{equ:flat_like}
     \left\{ \mathbf{x} = (x,y,z) \in \Pi \colon  \mathbf{k}  \cdot \left(\begin{smallmatrix} x \\ y \end{smallmatrix}\right)
      + z \cot \theta < -C_0 \right\}
     \,\, \subset \,\, L
     \,\, \subset \,\,
     \left\{ \mathbf{x} = (x,y,z) \in \Pi \colon
      \mathbf{k}  \cdot \left(\begin{smallmatrix} x \\ y \end{smallmatrix}\right)
      + z \cot \theta < C_0 \right \},
\end{equation}
for some constant \(C_0>0\), i.e.,
 \(\Sigma_{LV}\) is sandwiched between two parallel planes.
The contact angle pair $(\mathbf{k},\theta)$ is then called \emph{pinned}.
In the graphical setting of \cite{feldman2021limit}, the set of pinned angles at a given apparent normal \(\mathbf{k}\) is shown to form a closed interval.
In the full capillary setting considered here, where \(\Sigma_{LV}\) need not be graphical, we conjecture that the set of pinned angles
\[
  \Theta_p(\mathbf{k}) \;:=\; \{\theta \colon (\mathbf{k},\theta)\text{ is pinned}\}
\]
also forms a closed interval, which we identify as the CAH interval.
Furthermore, when $\psi$ is smooth, we conjecture that $\Theta_p(\mathbf{k}) \subset (0,\pi)$, i.e. the advancing and receding angles do not degenerate to total wetting or total dewetting.
Accordingly, we define
\(
  \thetarec(\mathbf{k}) \;:=\; \min\Theta_p(\mathbf{k})
  \) and
\(
  \thetaadv(\mathbf{k}) \;:=\; \max\Theta_p(\mathbf{k})
\)
so that $\Theta_p(\mathbf{k}) = \left[\thetarec(\mathbf{k}), \, \thetaadv(\mathbf{k}) \right]$.

\begin{figure}[t!]
\centering
\begin{minipage}[c]{0.40\textwidth}\centering
    \includegraphics[width=\linewidth, trim={146 311 300 282}, clip]{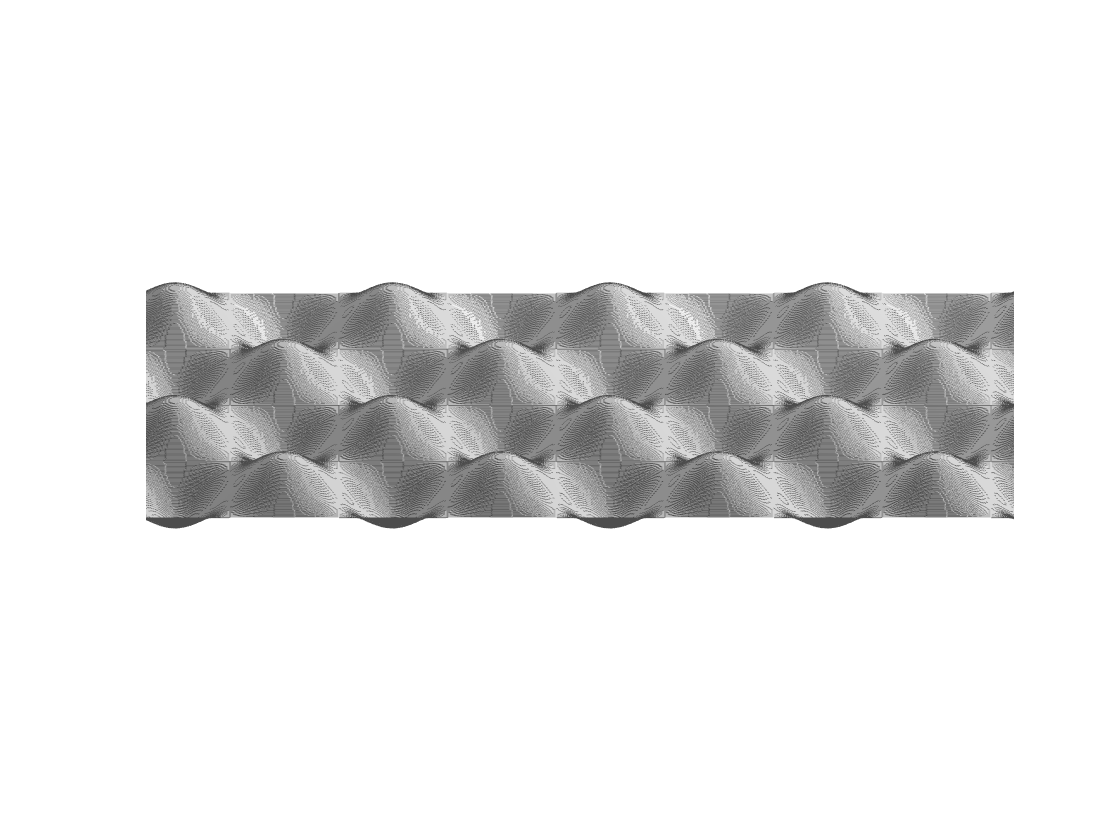}
\end{minipage}
\begin{minipage}[c]{0.40\textwidth}\centering
    \includegraphics[width=\linewidth, trim={85 135 15 80}, clip]{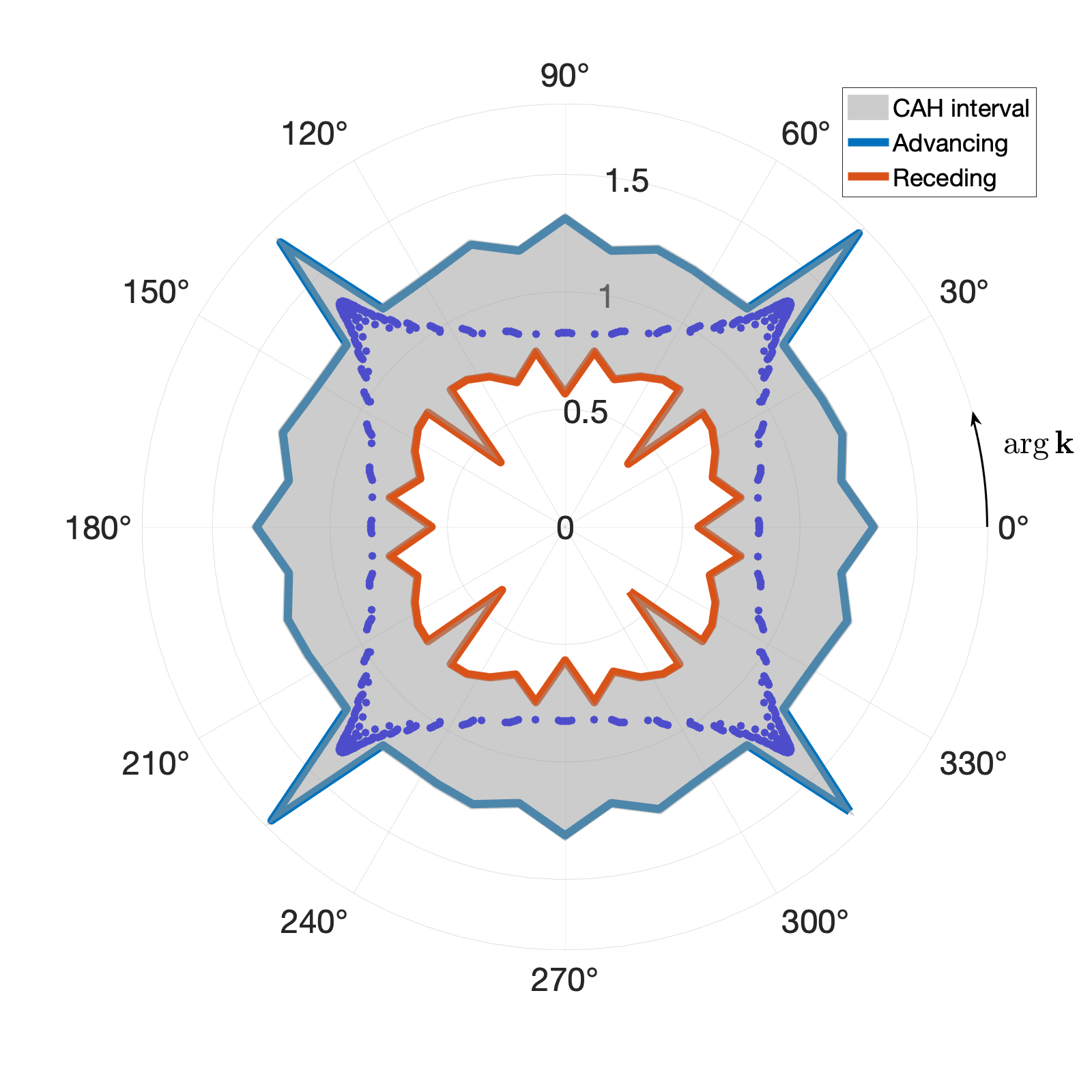}
\end{minipage} \\[3mm]
\begin{minipage}[c]{0.40\textwidth}\centering
    \includegraphics[width=\linewidth, trim={70 46 70 32}, clip]{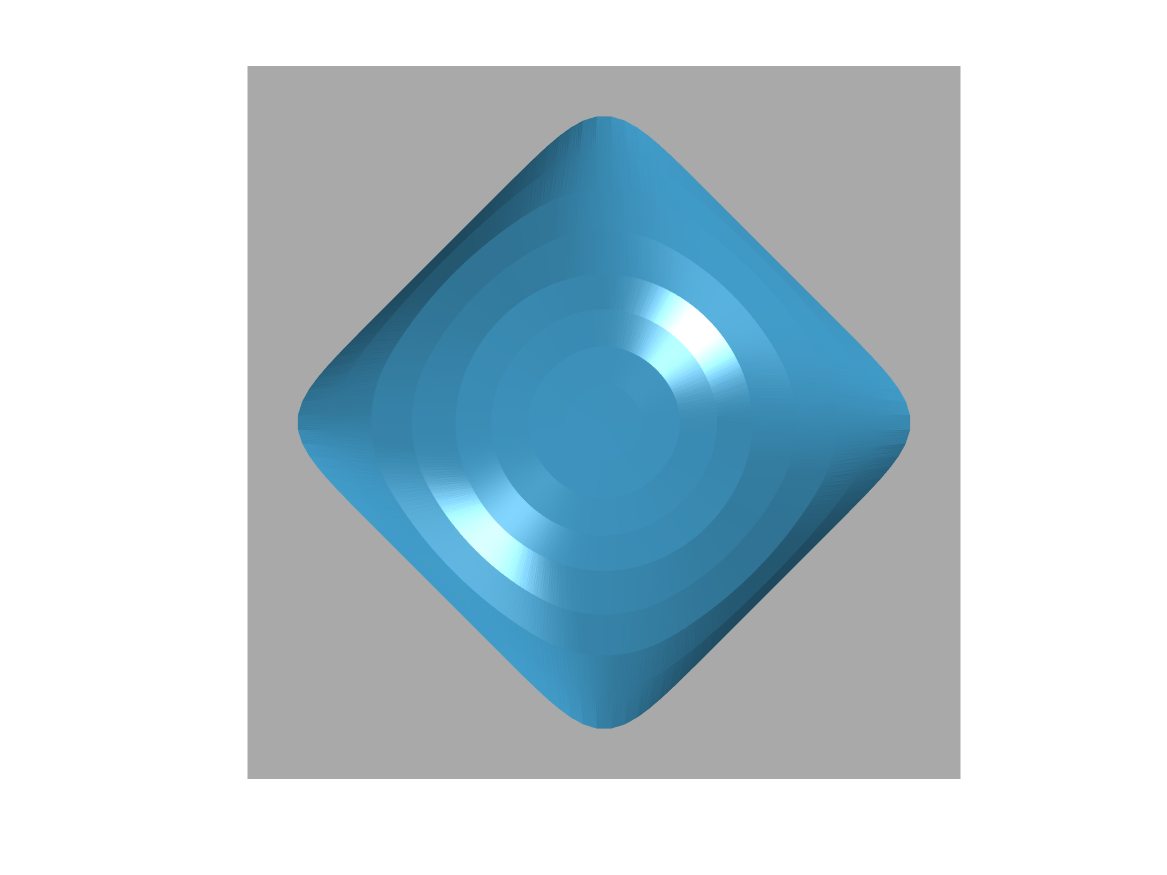}
\end{minipage}
\begin{minipage}[c]{0.40\textwidth}\centering
    \includegraphics[width=\linewidth, trim={150 107 130 127}, clip]{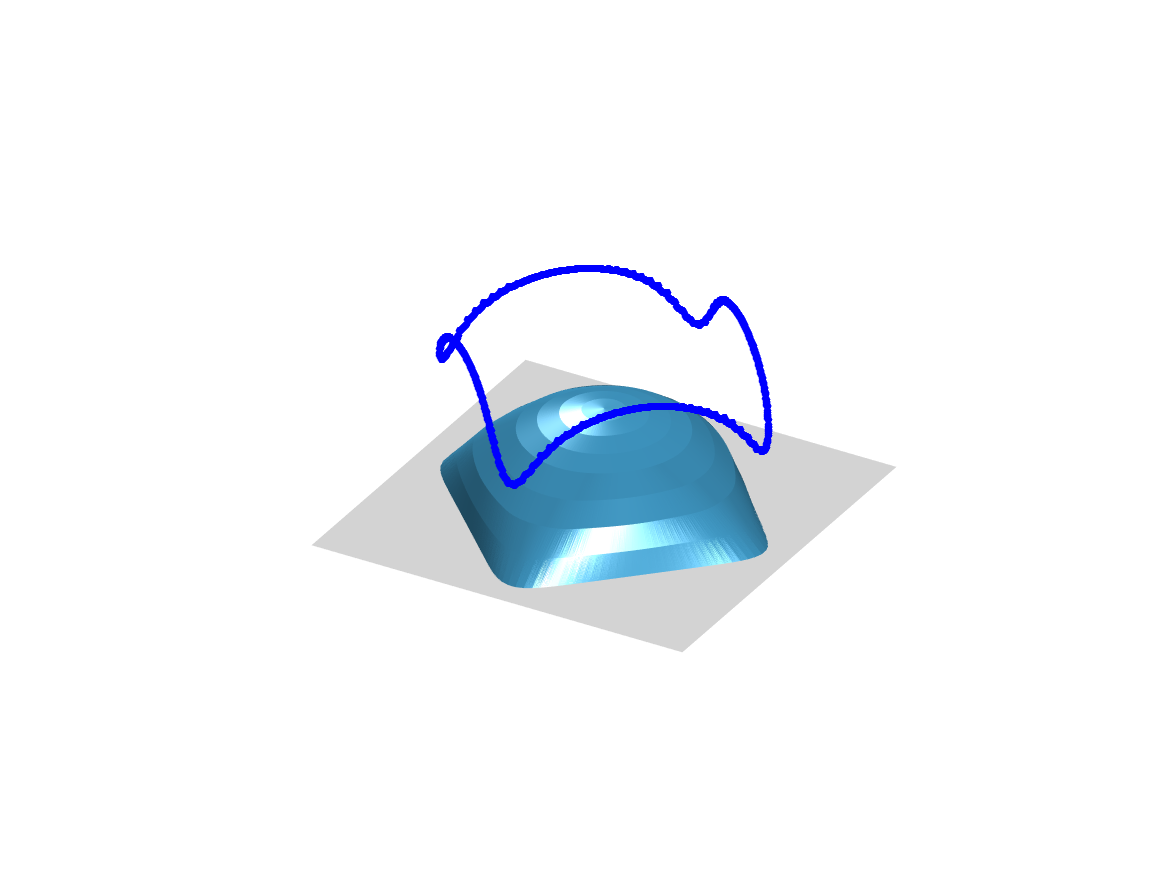}
\end{minipage}
\caption{{\bf (top left)} Zoom-in of a doubly periodic rough surface
\(z= \psi(x,y) = \eps\sin \frac{x}{\eps} \sin \frac{y}{\eps}\).
{\bf (top right)} The computed CAH interval (gray region) and the simulated droplet contact angles (blue dots) plotted as a function of the contact-line orientation $\arg \mathbf{k}$ for this surface.
{\bf (bottom left)} View of the droplet (light blue) from the top, with the solid surface shown in gray.
{\bf (bottom right)} Contact angle of the droplet (deep blue) plotted along the contact line, with its value represented by the $z$-coordinate.}
\label{fig:droplet_sideview}
\end{figure}

\paragraph{Illustrative example of an anisotropy effect}
To illustrate the CAH interval and its importance in applications, we present a computational example that uses the methods developed in this paper to study a droplet on an example micropatterned surface.
Given a particular periodic micropattern, our method computes the anisotropic CAH interval, $[\thetarec(\mathbf{k}),\thetaadv(\mathbf{k})]$, and thereby predicts whether that rough surface supports a stationary droplet with a desired wetted region.
We start with a simple micropattern with height profile
$z = \psi(x,y) = \eps \sin \frac{x}{\eps} \sin\frac{y}{\eps}$;
see Figure~\ref{fig:droplet_sideview}(top left).
Then, using the computational method developed in this work, we compute the CAH interval as a function of the contact-line orientation, $\mathbf{k}$.
For each orientation angle, $\arg \mathbf{k} = \arctan \frac{k_2}{k_1}$, we plot the interval of pinned angles in gray; see Figure~\ref{fig:droplet_sideview}~(top right).
Notice that the CAH interval is significantly larger at orientation angle $\tfrac{\pi}{4}$ (up to symmetry) than at any other angle.
Using the Surface Evolver software \cite{brakke1994surface}, we identify a particular square-like wetted region supporting a solution of the equations of capillarity~\eqref{equ:vol_minimizer_relaxed} with contact angle consistent with the anisotropic CAH; see Figure~\ref{fig:droplet_sideview}~(bottom).
In the bottom left panel we show a top-down view of the droplet displaying the square-like wetted set.
The square-like shape is oriented consistently with the diagonal micropattern, with side directions at angle \(\tfrac{\pi}{4}\).
In the bottom right panel, we superimpose the contact angle graphed in dark blue on the vertical axis over the contact line.
We also show the contact angle achieved by this droplet, as a function of the contact-line orientation, superimposed in dark blue on the CAH interval plot in the upper right panel.
Because the contact angles, plotted in blue, are inside the CAH interval, this droplet is stationary on this micropatterned surface.

\paragraph{Capillary mean curvature flow (CMCF)} Computationally, our goal will be to approximate solutions of \eqref{prob:infty_channel}
that additionally satisfy the far-field condition \eqref{equ:flat_like} for a given apparent contact angle, $\theta$. This is a challenging computational problem as it is posed on an infinite domain and involves a microscopic Neumann-like boundary condition on $\Gamma_{SLV}$. To facilitate this, we introduce the \emph{capillary mean curvature flow (CMCF)}, \begin{equation}
\label{e:CMCF}
\begin{cases}
V_n = \kappa, & \text{on } \Sigma_{LV},\\
\theta_{SLV}(\mathbf{x}) = \theta_Y(\mathbf{x}), &\mathbf{x}\in \Gamma_{SLV},
\end{cases}
\end{equation}
where $V_n$ denotes the normal velocity of the interface.
Away from the contact line, the liquid–vapor interface $\Sigma_{LV}$ evolves according to mean curvature flow, while at the contact line, it is constrained to meet the solid surface at Young's angle, $\theta_Y(\mathbf{x})$.
In this paper, we focus on hysteresis induced by surface roughness rather than chemical inhomogeneity so, in what follows, we assume a constant Young's angle $\theta_Y(\mathbf{x}) = \theta_Y$.
In \eqref{e:CMCF}, an additional boundary condition should be specified as $z \to \infty$ to find solutions that satisfy the far-field condition \eqref{equ:flat_like}; we will discuss this further in Section~\ref{s:2methods3d}.
Note that CMCF \eqref{e:CMCF} does not necessarily match the physics of the fluid dynamics; rather, the stationary states of CMCF are the same as the stationary states of the full fluid dynamical model, i.e., solutions to \eqref{prob:infty_channel}.

\subsection{Overview of results}\label{sec:main}
In Section~\ref{s:HysteresisInterval}, we propose two computational approaches to augment the CMCF \eqref{e:CMCF} with a far-field condition replacing \eqref{equ:flat_like} and to approximate the CAH interval.
The first approach is posed on a half-space domain and imposes a Neumann-type condition at infinity, which is updated according to a bisection strategy to approximate the CAH interval.
The second approach introduces a fictitious boundary to form a finite height channel and prescribes Dirichlet boundary conditions on the fictitious boundary.
By initializing the liquid–vapor interface $\Sigma_{LV}$ to prescribed lower and upper angle barrier configurations, the method converges to maximal and minimal stationary states.
These states are expected to attain the smallest and largest apparent contact angles, providing approximations for the receding and advancing angles, respectively.
Because the finite-channel method is less computationally expensive than the bisection method in our experiments, we use it throughout the remainder of the paper.

In Section~\ref{s:alg}, we introduce a novel computational method for simulating CMCF \eqref{e:CMCF} in a finite height channel where the solid surface has microscale roughness.
To handle the multiscale nature of the problem, we adapt the Schwarz alternating method to our setting \cite{lions1988schwarz,Toselli_2005}.
This \emph{two-scale alternating (TSA) method} decomposes the computational domain into a near-region close to the solid surface with microscale structure and a macroscale far-region; see Algorithm~\ref{alg:TSA}.
In the near-region, the CMCF satisfying the Young's angle condition at the contact line \eqref{e:CMCF} is evolved using an MBO diffusion-generated method.
In the far-region, the curvature of the minimal surface is well-approximated by its linearization and we explicitly solve the resulting equations via a Fourier series solution in terms of the boundary data from the diffusion-generated solution; see Theorem~\ref{thm:aprox_channel_min_upper_interface}.
In Theorem~\ref{t:TSA_energy}, we prove that every iteration of a reference reformulation of the TSA method decreases an approximation to the total interfacial energy \eqref{equ:interface_energy}.

In Section~\ref{s:Imple}, we describe our numerical implementation for computing the directional CAH on the doubly periodic surface,
\(z = \psi(x,y) = \eps \sin \frac{x}{\eps} \sin\frac{y}{\eps}\).
The computational domain is first rotated so that the propagation direction aligns with the $x$-direction.
The rotation is chosen to be rational in order to preserve the $\mathbb{Z}^2$-periodicity.
Consequently, the convolution in the MBO diffusion-generated method can be efficiently computed using the FFT.
To update the minimal surface in the far-region, Dirichlet boundary data is sampled and quadrature is performed to evaluate the Fourier series solution, again via FFT. The minimal surface in the far-region need only be reconstructed in the overlap region.
The two steps are iterated until a two-level stopping criterion is satisfied in the near-region.

In Section~\ref{s:NumExp}, we present the results of numerical experiments that approximate the CAH interval and demonstrate its anisotropic dependence on contact-line orientation.
For example, for the doubly periodic rough surface \(z=\eps\sin \frac{x}{\eps} \sin \frac{y}{\eps}\), illustrated in Figure~\ref{fig:droplet_sideview}~(top left), we observe a sharp change in the CAH interval at the diagonal direction \(\tfrac{\pi}{4}\); see Figure~\ref{fig:droplet_sideview}~(top right).
With a slight abuse of terminology, we refer to this sharp directional change as a discontinuity.
The shapes of the advancing and receding interfaces are also studied.
For example, a pair of advancing and receding interfaces is shown in Figure~\ref{fig:3D_apparent}.
For rational directions near the discontinuity, we observe overlapping patterns of the contact lines, suggesting a ``typewriter''-like transition of the advancing front.
Note that Figure~\ref{fig:3D_apparent} also serves as a 3D illustration of Figure~\ref{fig:Dirichlet_condition_wetting}~(center).

\begin{figure} [t]
    \centering
    \begin{subfigure}{0.4\textwidth}
        \centering
        \includegraphics[width=\linewidth]{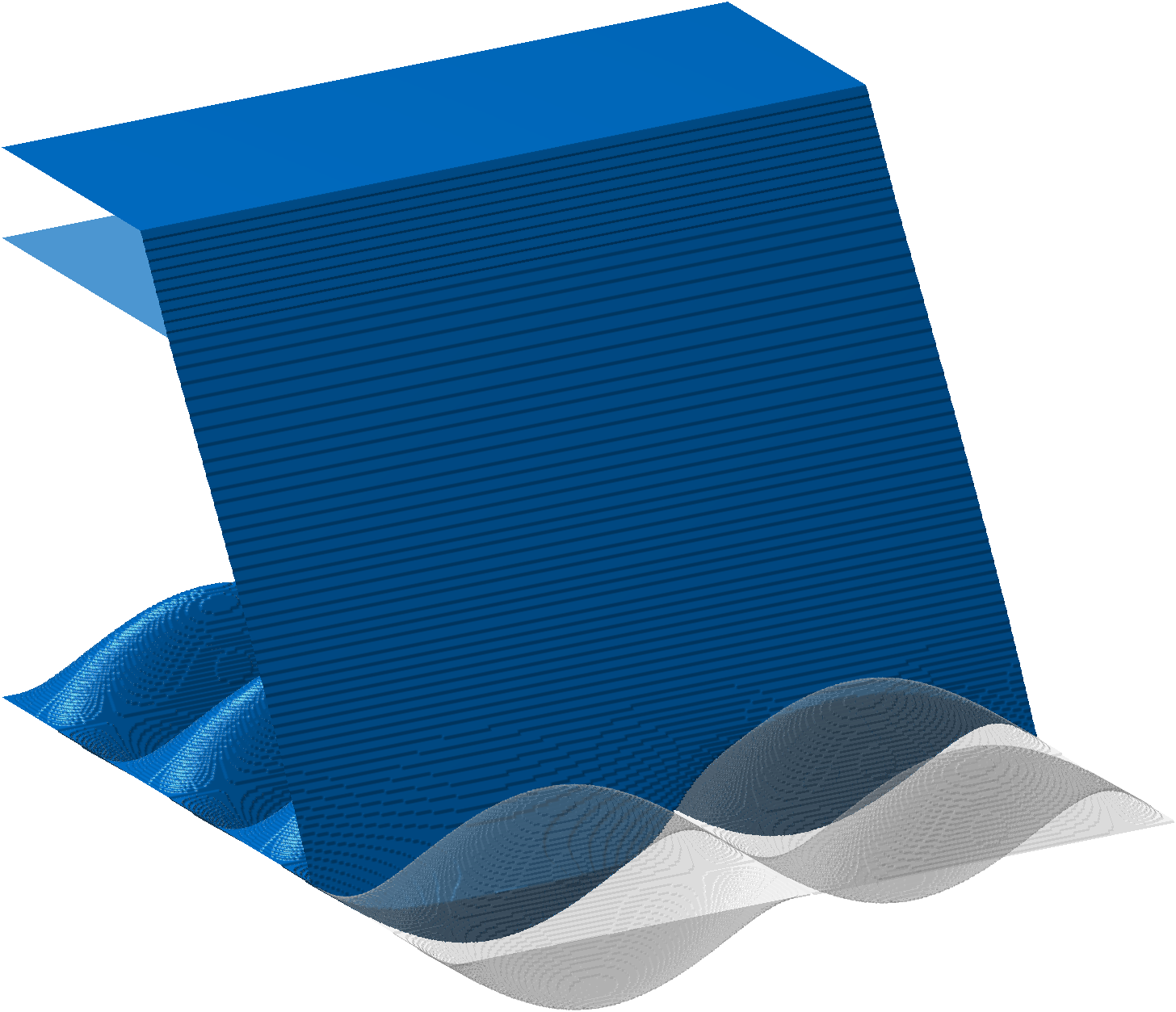}
        \caption{Receding angle}
    \end{subfigure}
    \hspace{1cm}
    \begin{subfigure}{0.4\textwidth}
        \centering
        \includegraphics[width=\linewidth]{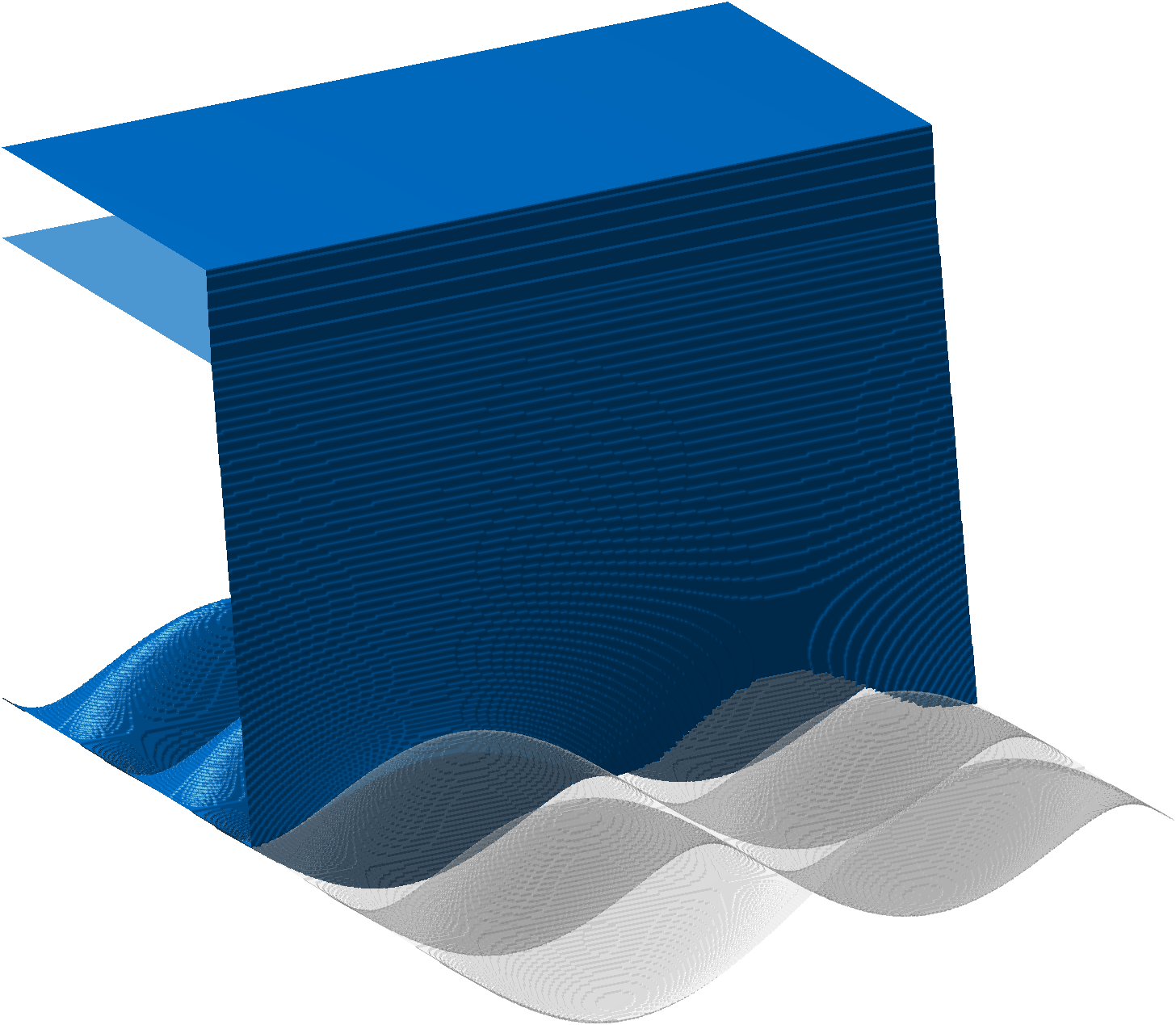}
        \caption{Advancing angle}
    \end{subfigure} \\
    \vspace{10mm}
    \begin{subfigure}{0.4\textwidth}
        \centering
        \includegraphics[width=\linewidth, trim={170 230 120 200}, clip]{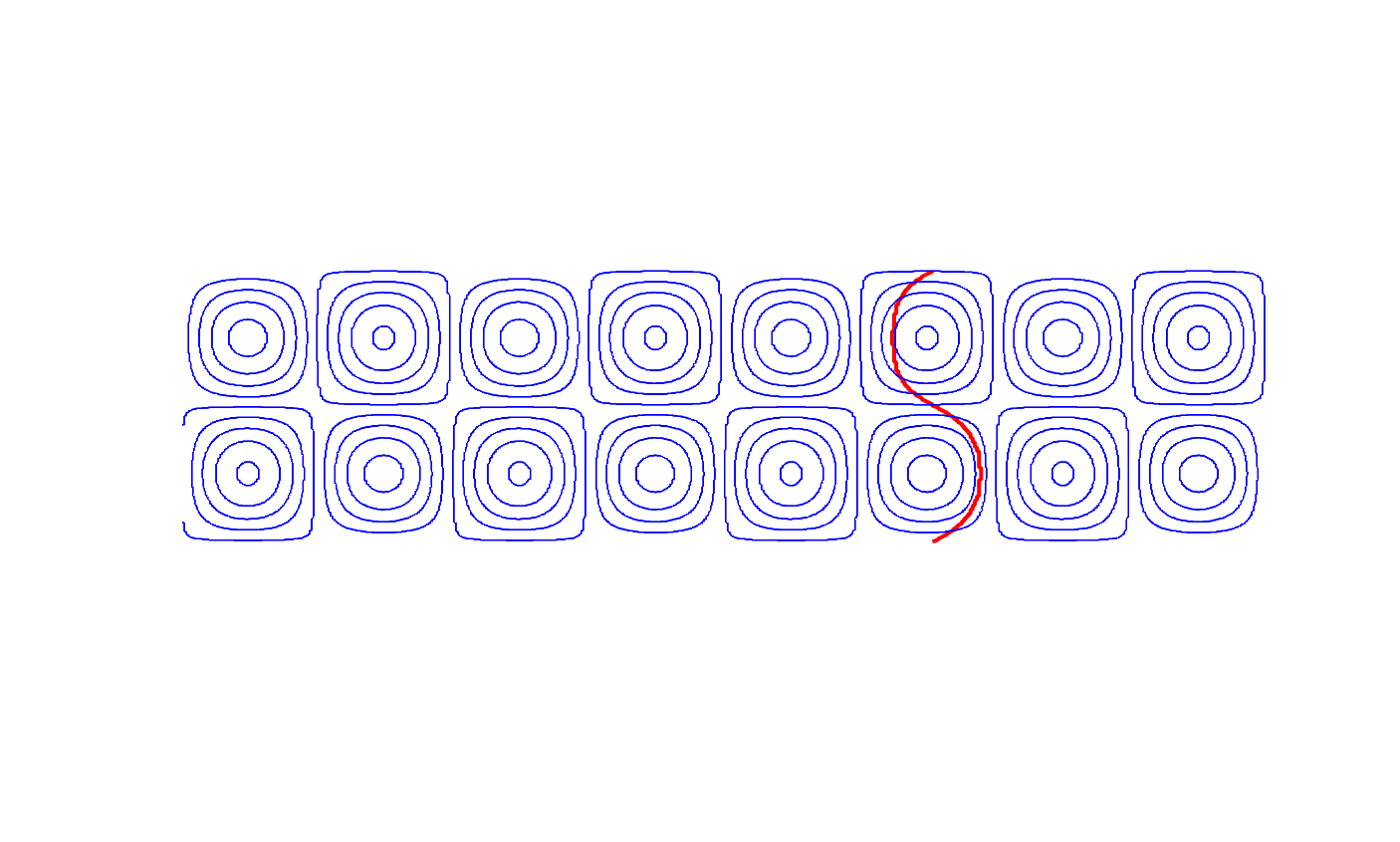}
        \caption{Receding contact line}
    \end{subfigure}
    \hspace{1cm}
    \begin{subfigure}{0.4\textwidth}
        \centering
        \includegraphics[width=\linewidth, trim={170 230 120 200}, clip]{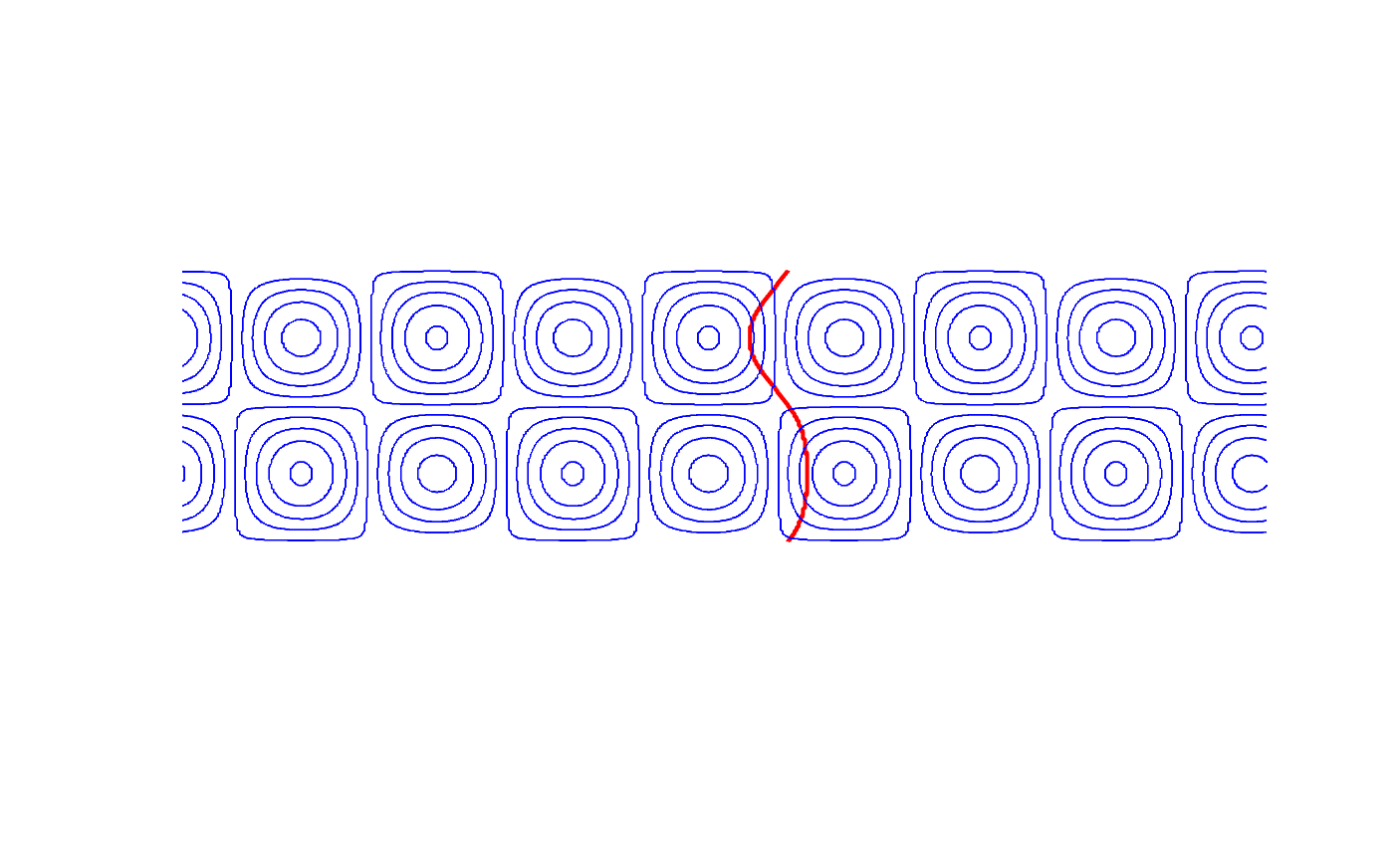}
        \caption{Advancing contact line}
    \end{subfigure}
    \caption{Interfaces and contact lines of the stationary state at direction \(\arg \mathbf{k} = 0\). (\textbf{a})--(\textbf{b}) The liquid–vapor interface is decomposed into a near-region (solid blue) and a far-region (transparent blue).
    The solid surface (gray) is defined by \(z = \psi(x,y)=\eps \sin \frac{x}{\eps} \sin \frac{y}{\eps}\). (\textbf{c})--(\textbf{d}) Contour plots of the contact lines (red) together with the rough surface (blue; maxima/minima labeled with \(+/-\)).
    The wetted region is located on the left.} \label{fig:3D_apparent}
\end{figure}

We discuss the results of the paper and future directions in Section~\ref{s:Disc}.

\section{Computation of the contact angle hysteresis (CAH) interval} \label{s:HysteresisInterval}
The computation of the CAH interval involves finding stationary states of the capillary mean curvature flow (CMCF) \eqref{e:CMCF}
with the far-field boundary condition \eqref{equ:flat_like}
on a rough surface.
Since this time-evolution problem is posed on an unbounded domain, with the far-field slope prescribed only in the limit $z\to\infty$, we construct and compare two approximation procedures.
To introduce the intuition, we start with a toy model.

\subsection{A toy model for the 2D hysteresis}
To illustrate the pinning mechanism and motivate the two approximation strategies, we first study a one-dimensional toy model.
This is a model of a particle on the line under the effects of an oscillatory potential gradient $f$ and a forcing $\alpha$:
\begin{align} \label{prob:toy_channel}
    \begin{cases}
        \dot{x} = -f(\frac{x}{\eps}
        )+\alpha, \\
        x(0) = x_0.
    \end{cases}
\end{align}
Here, $f$ is a continuous $1$-periodic function with range $f([0,1)) = [a,b]$.
The function $f$ is not given in explicit form, so it cannot be directly evaluated.
Our analysis therefore relies only on its periodicity and range.
The parameter $\eps\ll 1$ characterizes the spatial scale of oscillation.

In our context, this toy model can be viewed as a reduced description of contact-line pinning on a chemically patterned surface: \(x\) is a scalar coordinate measuring the displacement of the contact line along its direction of propagation, \([a,b]\) corresponds to the hysteresis interval, and $\alpha$ plays the role of the apparent contact angle.
The identification of a forcing with an angle is natural here: the apparent contact angle sets the net driving force on the contact line, so prescribing $\alpha$ is the one-dimensional analogue of prescribing the slope of the interface in the far field.
As such, if there is a fixed point in the system with parameter $\alpha$, we refer to such an $\alpha$ as a \emph{pinned force} and the interval of all such parameters as the pinning interval.
As a simplified model, it captures the essential mechanism of pinning and hysteresis without attempting to represent all microscopic geometry details present in higher dimensions.
Similar simplified modeling strategies are also employed, for example, in \cite{Mielke_2011}.

In this model, a stationary state exists precisely when $\alpha \in f([0,1)) = [a,b]$, so the pinning interval is $[a,b]$.
In general, as in the capillarity problem we are studying, it is not so easy to determine the stationary solutions.
We present two, more robust, approaches to determine the hysteresis interval in this toy model.

\subsubsection{Approach I: Bisection Method} \label{sec:toy_bisection}
A straightforward approach is based on bisection.
To iteratively approximate the maximal pinned force $b$,  we begin with an initial interval \([\alpha^{(0)}_{\mathrm{lower}}, \alpha^{(0)}_{\mathrm{upper}}] \ni b \).
At iteration $k$, we compute the midpoint
$ \alpha^{(k)}_{\mathrm{mid}} = \frac{\alpha^{(k)}_{\mathrm{lower}} + \alpha^{(k)}_{\mathrm{upper}}}{2}$ and
\begin{itemize}
\item if  \(-f(\frac{x}{\eps})+ \alpha^{(k)}_{\mathrm{mid}} > 0\) for all \(x\), we set $\alpha^{(k+1)}_{\mathrm{upper}} = \alpha^{(k)}_{\mathrm{mid}}$ and
$\alpha^{(k+1)}_{\mathrm{lower}} = \alpha^{(k)}_{\mathrm{lower}}$
so that the bracketing interval at the $k+1$ iteration is \([\alpha^{(k)}_{\mathrm{lower}},\alpha^{(k)}_{\mathrm{mid}}]\)
\item otherwise, we
set $\alpha^{(k+1)}_{\mathrm{lower}} = \alpha^{(k)}_{\mathrm{mid}}$ and
$\alpha^{(k+1)}_{\mathrm{upper}} = \alpha^{(k)}_{\mathrm{upper}}$
so that the  bracketing interval at the $k+1$ iteration is
\([\alpha^{(k)}_{\mathrm{mid}}, \alpha^{(k)}_{\mathrm{upper}}]\).
\end{itemize}
The sequence \(\{ \alpha^{(k)}_{\mathrm{mid}} \}\) converges linearly to the maximal pinned force \(b\).
Similarly, we can generate a sequence converging to the minimal pinned force \(a\).
The hysteresis interval $[a,b]$ is thus approximated by the bisection method with error \(O(\eps)\) through \(O\left(\log \eps^{-1} \right)\) iterations.

As \(f\) cannot be evaluated explicitly, we iterate the bisection by evolving the system \eqref{prob:toy_channel}.
If \(x\) is observed to traverse a displacement of several periods, we conclude that \(-f(\frac{x}{\eps})+ \alpha^{(k)}_{\mathrm{mid}} > 0\) for all \(x\).
However, the time to evolve the system increases as $\alpha^{(k)}_{\mathrm{mid}}$ approaches $b$.
To get past a single near-stationary region, where $|\alpha - f(\frac{x}{\eps})| = O(\delta)$ for $1>\delta>\eps$, the traversal time is expected to be \(O(\delta^{1/2})\), assuming that \(f\) behaves quadratically near its local extrema.

\subsubsection{Approach II: Forced Method} \label{sec:toy_finite_channel}

In this approach, we consider a modified system in which the constant parameter \(\alpha\) is replaced by a Lipschitz continuous, strictly decreasing function \(\alpha(x)\).
Consider the initial position $x(0) = x_0$ chosen such that \(\alpha(x_0)> b\), which ensures $\dot{x}>0$ for all $t>0$.
Along the trajectory, the parameter $\alpha(x(t))$ evolves according to
\[
\frac{d \alpha }{d t} = \frac{d \alpha}{d x} \frac{dx}{dt} <0,
\]
where the inequality follows from the monotonicity of \(\alpha(x)\).
For \(\eps\ll 1\), \(\alpha(x)\) remains nearly constant over each period of length \(\eps\).
We claim that when the system reaches a steady state \(x^\star\), i.e.  $f\left(\frac{x^\star}{\eps}\right) = \alpha(x^\star)$,
\begin{equation}\label{e.alpha-est}
    |\alpha(x^\star) - b| = O(\eps).
\end{equation}
See Figure~\ref{fig:forced_proof} for an illustration of the convergence, as \(\eps \to 0\), of the stationary state \(x^\star\) generated via the forced method towards  a maximum \(x_{\max}\), satisfying $f(\frac{x_{\max}}{\varepsilon}) = b$.

Similarly, the minimal value \(a\) can be approximated by  initializing \(x(0) = x_0\) such that \(\alpha(x_0)<a\) and applying the same procedure.
Combining the results, the hysteresis interval is determined with error \(O(\eps)\).

\begin{figure}
\centering
\begin{tikzpicture}[scale=.8]
\begin{axis}[
title={$\eps = 1$},
xlabel={$x$},
legend pos=north east,
width=7cm,
height=7cm
]
\addplot[domain=-4:4, thick, green] {1} node[pos=0.8, right] {};

\addlegendentry{$b$}

\addplot[domain=-4:4, thick, blue] {2/3 - x/3} node[pos=0.8, right] {};
\addlegendentry{$\alpha(x)$}

\addplot[domain=-4:4, thick, red] {sin(deg(x))};
\addlegendentry{$f(x)$}

\addplot[only marks, mark=*, color=black] coordinates {(-1,1)};
\node[anchor=south] at (axis cs:-1,1) {\scriptsize $\left(x_b,f(\frac{x_b}{\eps})=b\right)$};

\addplot[only marks, mark=*, color=black] coordinates {(0.5,0.4794)};
\node[anchor=west] at (axis cs:0.5,0.4794) {\scriptsize $\left(x^*,f(\frac{x^*}{\eps})\right)$};

\end{axis}
\end{tikzpicture}
\hspace{1cm}
\begin{tikzpicture}[scale=.8]
\begin{axis}[
title={$\eps = 0.1$},
xlabel={$x$},
legend pos=north east,  
width=7cm,
height=7cm,
samples=200,  % Increases resolution of the plot
domain=-4:4
]
% Plot alpha_max
\addplot[thick, green] {1};  
\addlegendentry{$b$}

% Plot alpha(x)
\addplot[thick, blue] {2/3 - x/3};  
\addlegendentry{$\alpha(x)$}

% Plot f(x/epsilon)
\addplot[thick, red] {sin(deg(x/0.1))};  
\addlegendentry{$f(\frac{x}{0.1})$}

\addplot[only marks, mark=*, color=black] coordinates {(-1,1)};
\node[anchor=south] at (axis cs:-1,1) {\scriptsize $\left(x_b,f(\frac{x_b}{\eps})=b\right)$};

\addplot[only marks, mark=*, color=black] coordinates {(-0.5281, 0.8472)};
\node[anchor=west] at (axis cs:-0.5281, 0.8472) {\scriptsize $\left(x^*,f(\frac{x^*}{\eps})\right)$};

\end{axis}
\end{tikzpicture}
\caption{An illustration of the forced method for computing the pinning interval for \eqref{prob:toy_channel}.
As  $\eps$ tends to zero, the stationary state \(x^\star\) generated via the forced method converges to an actual maximum \(x_{\max}\), satisfying $f(\frac{x_{\max}}{\varepsilon}) = b$.
See Section~\ref{sec:toy_finite_channel}.}
\label{fig:forced_proof}
\end{figure}
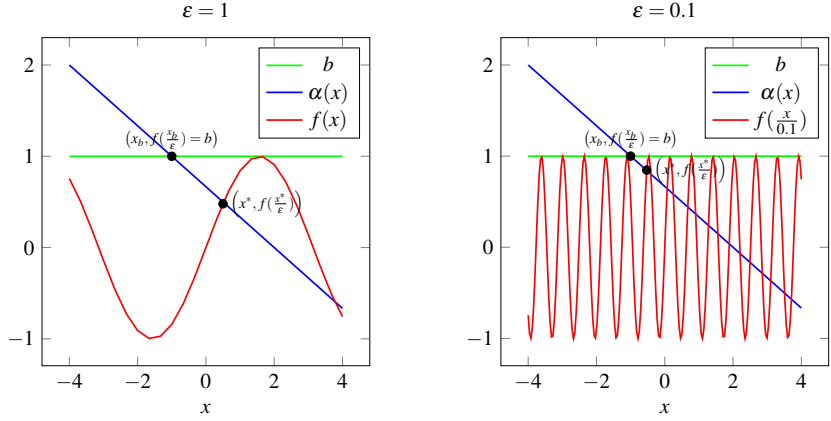

\begin{proof}[Proof of \eqref{e.alpha-est}]
    Let \(x_b\) denote the solution to \(\alpha(x)=b\).
    As $\eps$ is the period of \(f\left(\frac{x}{\eps}\right)\), there exists $x_{\max}\in[x_b,x_b+\eps]$ such that $f(\frac{x_{\max}}{\eps})=b$.
    Since $f\leq b=\alpha(x_b)$, we have $\alpha(x)-f(\frac{x}{\eps})\geq 0$ at $x=x_b$.
    Since $\alpha$ is decreasing and $x_{\max}\geq x_b$, we have $\alpha(x_{\max})\leq b = f(\frac{x_{\max}}{\eps})$, so $\alpha(x)-f(\frac{x}{\eps})\leq 0$ at $x=x_{\max}$.
    By the Intermediate Value Theorem, there exists $x^\star\in[x_b,x_{\max}]\subset [x_b,x_b+\eps]$ such that \(\alpha(x^\star)-f(\frac{x^\star}{\eps})= 0\).
    Choose $x^\star$ to be the smallest such root.
    Then $\dot{x}=\alpha(x)-f(\frac{x}{\eps})> 0$ over $(-\infty,x^\star)$.
    For any $x_0 \leq x^\star$, the limit of $x(t)$ is therefore $x^\star$.
    Therefore, $$|\alpha(x^\star) - b| \leq \mathrm{Lip}(\alpha)\,|x_{\max}-x_b|\leq \mathrm{Lip}(\alpha)\,\eps =O(\eps),$$ where $\mathrm{Lip}(\alpha)$ denotes the Lipschitz constant of $\alpha$.
\end{proof}

As in the bisection method, additional time is required to traverse the near-stationary regions of \eqref{prob:toy_channel}.
Each period cell, where $|\alpha - b| =O(\delta)$, costs $O( \delta^{-1/2})$ to traverse.
As there are $ O(\frac{\delta}{\eps})$ such period cells, the total time complexity to reach the $O(\eps)$ error is $O(\eps^{-1})$ summing over dyadic scales $1 \geq \delta \geq \eps$.

\subsection{Methods to approximate the CAH interval} \label{s:2methods3d}
We aim to augment CMCF \eqref{e:CMCF}  with a far-field condition that approximates \eqref{equ:flat_like}.
Each of the two approaches to the toy model has a counterpart for the capillary problem, obtained by supplying the far-field condition in the corresponding way.
The bisection method of Section~\ref{sec:toy_bisection} becomes a bisection over prescribed far-field angles, imposed as a Neumann-type condition at $z=\infty$ on the half-space; the forced method of Section~\ref{sec:toy_finite_channel} becomes CMCF in a channel of finite height, in which a Dirichlet condition on a fictitious upper boundary supplies the decreasing forcing.
We describe the two in turn.

\subsubsection{Approach I: Bisection method in a half-space domain}\label{sec:bisection}
For a $\mathbb{Z}^2$-periodic function $\psi(x,y)$, we consider a
rough solid surface  given by
$S:=\{(x,y,z)\in \mathbb{R}^3 \colon  z\leq \psi(x,y)\}$. The half-space domain is given by:
\[
 \Pi :=\{ (x,y,z)\in \mathbb{R}^3 \colon z>\psi(x,y) \}.
\]
We consider CMCF \eqref{e:CMCF} on the half-space domain with a Neumann-like boundary condition at $z = \infty$,
\begin{align} \label{prob:bisection}
    \begin{cases}
        V_n = \kappa, & \text{on } \Sigma_{LV},\\
        \theta_{SLV}(\mathbf{x}) = \theta_Y, &\mathbf{x}\in \Gamma_{SLV}, \\
        \mathbf{N}_{LV}\to(\sin(\theta)\mathbf{k},\cos(\theta)), &z \to \infty,
\end{cases}
\end{align}
where $\mathbf{N}_{LV}$ denotes the outward unit normal vector of $\Sigma_{LV}$ pointing from the liquid region toward the vapor region, and $\theta$ is a prescribed contact angle.
The objective is to determine the maximal and minimal solutions in terms of the apparent contact angle.

The CMCF \eqref{prob:bisection} may (and does) admit multiple stationary states.
To approximate the CAH interval, we seek the maximal and minimal stationary states.
These states can be approximated by evolving the CMCF~\eqref{prob:bisection} from planar initial data with different apparent contact angles.
We introduce the planar liquid region associated with an apparent contact angle $\theta$:
\begin{equation}\label{equ:L_theta_infty}
L_{\theta} := \{ (x,y,z)\in \Pi\setminus S\colon  \left(\begin{smallmatrix} x \\ y \end{smallmatrix}\right)\cdot\mathbf{k} \sin \theta + z \cos \theta < 0 \}.
\end{equation}
Note that $L_{\theta}$ in \eqref{equ:L_theta_infty} is a half-space intersecting $\Pi$ and is generically not the liquid region corresponding to a stationary solution of \eqref{prob:bisection} since it does not satisfy the Young's angle condition.

We define the lower and upper angle barrier configurations \( L_{\theta_-}, L_{\theta_+}\), where $\theta_- < \theta_+$ are chosen so that the apparent contact angle of every stationary solution \(L\) of CMCF~\eqref{prob:bisection} lies in \((\theta_-,\theta_+)\).
We apply a bisection algorithm starting with lower and upper bounds $\theta_-$ and $\theta_+$.
We evolve the flow from the initial data \( L_{\theta_-}, L_{\theta_+}\).
Starting from \(L_{\theta_-}\), the interface recedes, without converging to a stationary state.
Conversely, starting from \(L_{\theta_+}\), the interface advances.
Then we pick $\theta \in (\theta_-,\theta_+)$, for example $\theta = \tfrac{\theta_-+\theta_+}{2}$, and we proceed with a bisection algorithm as in Section~\ref{sec:toy_bisection} depending on whether the flow \eqref{prob:bisection} from initial data $L_\theta$ is observed to advance, recede, or remain pinned.

\subsubsection{Approach II: CMCF in a finite height channel}

Next we consider CMCF \eqref{e:CMCF} in a \emph{finite channel domain} with height $R$,
 \[
 \Pi_{R}:=\{ (x,y,z)\in \mathbb{R}^3\colon  \psi(x,y)<z<R \}.
 \]
The height $R$ is an artificial parameter of the method.
Because the interface relaxes exponentially to a plane away from the rough surface, as discussed in~\ref{sec:Justification_linearization}, the computed apparent contact angle is insensitive to $R$ once $R$ is large compared with the period of the roughness.
We augment CMCF \eqref{e:CMCF} with a Dirichlet boundary condition at the (fictitious) boundary $z=R$, to obtain
\begin{align} \label{prob:channel}
\begin{cases}
V_n = \kappa, & \text{on } \Sigma_{LV},\\
\theta_{SLV}(\mathbf{x}) = \theta_Y, &\mathbf{x}\in \Gamma_{SLV},\\
L \cap \{z=R\} =  \left\{(x,y,R) \colon
\left(\begin{smallmatrix} x \\ y \end{smallmatrix}\right) \cdot \mathbf{k}  \leq 0 \right\}.
\end{cases}
\end{align}
The Dirichlet boundary condition in the last line of \eqref{prob:channel} aligns the outer normal of the liquid–vapor interface at \(z=R\) with the prescribed apparent normal \(\mathbf{k}\).

As in Section~\ref{sec:bisection}, we seek the maximal and minimal stationary states.
Due to the Dirichlet condition, these states can be obtained by evolving the CMCF~\eqref{prob:channel} from the adapted lower and upper angle barrier configurations.
For the channel setting, we consider the planar liquid region
\begin{equation*}
L_{\theta} := \{ (x,y,z)\in \Pi_R\setminus S\colon  \left(\begin{smallmatrix} x \\ y \end{smallmatrix}\right)\cdot\mathbf{k} \sin \theta + (z-R) \cos \theta < 0 \}.
\end{equation*}
With a slight abuse of notation, the corresponding lower and upper angle barrier configurations in the channel domain are denoted by
\begin{equation}\label{equ:angle_barriers}
\hbox{lower and upper angle barrier configurations} \,\ L_{\theta_-}, \, L_{\theta_+},
\end{equation}
where $\theta_- < \theta_+$ are chosen such that $L_{\theta_-}\supset L$ and $L_{\theta_+}\subset L$ for any stationary solution \(L\) of  CMCF~\eqref{prob:channel}.
Starting from the lower angle barrier configuration \(L_{\theta_-}\), the interface is observed to recede and converge to the maximal stationary state.
The maximal stationary state attains the smallest apparent contact angle among all the stationary states, and this creates an approximation of $\thetarec$, see Figure~\ref{fig:lower_upper_angle_barrier}~(left).
Conversely, starting from the upper angle barrier configuration \(L_{\theta_+}\), the interface is observed to advance and converge to the minimal stationary state, which attains the largest apparent contact angle and thus provides an approximation of $\thetaadv$, see Figure~\ref{fig:lower_upper_angle_barrier}~(right).

Existence of lower and upper angle barrier configurations can be established in the chemically patterned case by choosing $\theta_+ > \max_{\mathbf{x}} \theta_{Y}(\mathbf{x})$ and $\theta_- < \min_{\mathbf{x}} \theta_Y(\mathbf{x})$.
However, for general rough surfaces, there is no theoretical guarantee that the angle barriers will produce an advancing / receding solution.
In practice, we simulate CMCF~\eqref{prob:channel} using lower and upper angle barriers chosen sufficiently far from the observed pinning interval.

The initial interfaces are empirically observed to evolve monotonically across several periods of the roughness before converging to the corresponding stationary states.
The same empirical observation also holds in the previous bisection setting.

\begin{figure}[t]
\centering
\begin{tikzpicture}[x={(1cm,0cm)}, y={(0.4cm,0.6cm)}, z={(0cm,1cm)}]

% Left
\begin{scope}[xshift=-4cm]
% rough surfavce
\coordinate (A) at (-3,-1,0);
\coordinate (B) at (3,-1,0);
\coordinate (C) at (-3,1,0);
\coordinate (D) at (3,1,0);
% interfacve
\coordinate (E) at (-2,-1,2);
\coordinate (F) at (-2,1,2);
\coordinate (G) at (2,-1,0);
\coordinate (H) at (2,1,0);
% receding endpoint
\coordinate (I) at (0,-1,0);
\coordinate (J) at (0,1,0);
% upper floor
\coordinate (K) at (-3,-1,2);
\coordinate (L) at (3,-1,2);
\coordinate (M) at (-3,1,2);
\coordinate (N) at (3,1,2);

% Blue box
\draw[black,thick] (A)--(B) (C)--(D);
\draw[black,thick] (K)--(L) (M)--(N);
\draw[fill=red!20,draw=red] (A) -- (C) -- (J)--(I)--cycle;
\draw[blue,thick] (E)--(F) (E)--(G) (F)--(H);
\draw[blue,thick,decorate,decoration={snake,amplitude=0.3mm,segment length=3mm}] (G) --(H);
\node at (-3,0,2) {L};
\node at (1,0,2) {V};
\node at (0,0,-1) {S};
\node at (-2,1.5,2) {$z=R$};
\draw[->,thick,black] (2.5,0,0) -- (1.5,0,0); 
\end{scope}
% Right
\begin{scope}[xshift=4cm]
% rough surfavce
\coordinate (A) at (-3,-1,0);
\coordinate (B) at (3,-1,0);
\coordinate (C) at (-3,1,0);
\coordinate (D) at (3,1,0);
% interfacve
\coordinate (E) at (2,-1,2);
\coordinate (F) at (2,1,2);
\coordinate (G) at (-2,-1,0);
\coordinate (H) at (-2,1,0);
% receding endpoint
\coordinate (I) at (0,-1,0);
\coordinate (J) at (0,1,0);
% upper floor
\coordinate (K) at (-3,-1,2);
\coordinate (L) at (3,-1,2);
\coordinate (M) at (-3,1,2);
\coordinate (N) at (3,1,2);

\draw[black,thick] (A)--(B) (C)--(D);
\draw[black,thick] (K)--(L) (M)--(N);
\draw[fill=red!20,draw=red] (B) -- (D) -- (J)--(I)--cycle;
\draw[blue,thick] (E)--(F) (E)--(G) (F)--(H);
\draw[blue,thick,decorate,decoration={snake,amplitude=0.3mm,segment length=3mm}] (G) --(H);
\node at (-1,0,2) {L};
\node at (3,0,2) {V};
\node at (0,0,-1) {S};
\node at (2,1.5,2) {$z=R$};
\draw[->,thick,black] (-2.5,0,0) -- (-1.5,0,0); 
\end{scope}
\end{tikzpicture}
\caption{ CMCF in a finite height channel~\eqref{prob:channel} is initialized from lower (\textbf{left}) and upper (\textbf{right}) angle barrier configurations; the hysteresis interval is shown as a pinning region on the rough surface, marked in red.}
\label{fig:lower_upper_angle_barrier}
\end{figure}
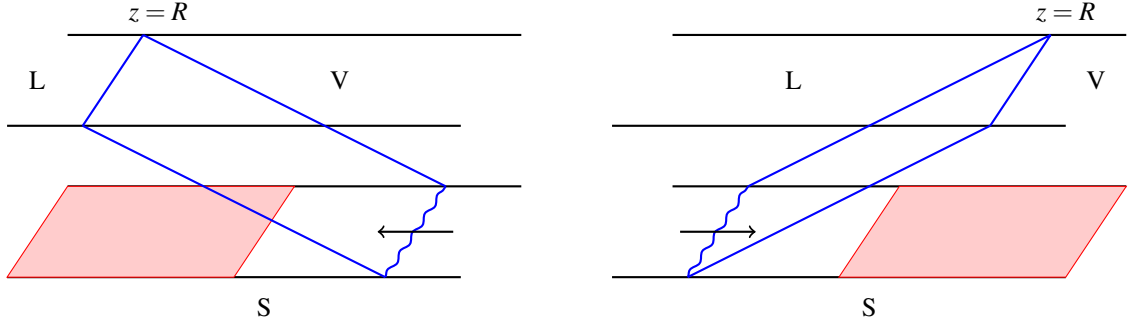

\begin{remark} \label{rem:cost}
In our observations, the bisection method~\eqref{prob:bisection} is computationally more expensive than the CMCF in a finite height channel \eqref{prob:channel}; therefore, in what follows, we restrict our attention to CMCF  in a finite height channel \eqref{prob:channel}.
\end{remark}

\section{A two-scale alternating (TSA) method for capillary mean curvature flow (CMCF) \texorpdfstring{\eqref{prob:channel}}{}} \label{s:alg}
Computational simulation of the CMCF in a finite height channel~\eqref{prob:channel} is challenging due to the multiple spatial scales that appear in this problem.
Naively using a uniform fine discretization throughout the domain is prohibitively computationally expensive and inefficient.
To resolve this issue, we introduce a two-scale alternating (TSA) method  based on the Schwarz alternating approach.

The classical Schwarz alternating method \cite{Schwarz+1869+105+120} was designed to solve  boundary value problems (e.g., the Laplace equation) on  domains that can be decomposed into simpler pieces.
The classical Schwarz method decomposes the domain into overlapping subdomains and alternately solves the governing problem on each subdomain using boundary data from the most recent iterate.
The domain $\Omega$ is decomposed into two overlapping subdomains $\Omega_1$ and $\Omega_2$.
First, one solves the PDE on $\Omega_1$ using boundary data from an initial guess.
Then, the PDE is solved on $\Omega_2$ using the updated boundary data from $\Omega_1$.
The process is iterated until the two subdomain solutions converge in the transition region, $\Omega_1 \cap \Omega_2$.

Here, we adapt the Schwarz alternating idea to simulate the CMCF~\eqref{prob:channel}, which allows us to efficiently handle the different spatial scales.
Specifically, we decompose the computational domain into two overlapping subdomains.
The microscopic \emph{near-region}, \(C_{\mathrm{near}}\), contains the contact line, where a relatively fine discretization is used to capture the interfacial dynamics.
The macroscopic \emph{far-region}, \(C_{\mathrm{far}}\), is away from the solid boundary, where the geometry is relatively simple.
The liquid–vapor interface is updated alternately between the two subdomains.
In the near-region, a single iteration of a variant of the MBO diffusion-generated method is performed.
In the far-region, an approximate minimal surface is constructed by solving a linearized problem.
We refer to this multiscale iterative procedure as the \emph{two-scale alternating (TSA) method}.
This TSA method is illustrated in Figure~\ref{fig:zoomin_channel} and described in more detail in the following sections.

\begin{figure}[t]
\centering
\begin{subfigure}{0.4\textwidth}
\centering
\begin{tikzpicture}[x={(1cm,0cm)}, y={(0.0cm,0.0cm)}, z={(0cm,1.0cm)}]
\begin{scope}[xshift=0cm]
% rough surfavce
\coordinate (A) at (-2,-1,0);
\coordinate (B) at (3,-1,0);
\coordinate (C) at (-2,1,0);
\coordinate (D) at (3,1,0);
% interfacve
\coordinate (E) at (-1,-1,2);
\coordinate (F) at (-1,1,2);
\coordinate (G) at (1,-1,0);
\coordinate (H) at (1,1,0);
% upper transition trace
\coordinate (I) at (0.5,-1,0.5);
\coordinate (J) at (0.5,1,0.5);
% upper floor
\coordinate (K) at (-2,-1,2);
\coordinate (L) at (3,-1,2);
\coordinate (M) at (-2,1,2);
\coordinate (N) at (3,1,2);

\coordinate (O) at (-1,1,0.5);
\draw pic[draw=black,
line width=1.0pt, 
angle radius=5pt,
angle eccentricity=1.2]
{angle = O--I--G};
\node at (0.3,-1,0.2) {$\theta_{i}$};

\draw[black,thick] (A)--(B) (C)--(D);
\draw[black,thick] (K)--(L) (M)--(N);
\draw[blue,thick] (E)--(F) (E)--(G) (F)--(H);
\draw[blue,thick,decorate,decoration={snake,amplitude=0.7mm,segment length=3mm}] (G) --(H);
\draw[blue,thick,dashed,decorate,decoration={snake,amplitude=0.5mm,segment length=3mm}] (I) --(J);
\node at ($(H)+(-2.5,0,0.25)$) {$z=r$};
\node at ($(H)+(0,0,-0.25)$) {$C_{\mathrm{near}}$};
\node at ($(H)+(-0.6,0,1.0)$) {$C_{\mathrm{far}}$};
\node at (-2,0,1) {L};
\node at (2,0,1) {V};
\node at (-1,0,-1) {S};
\node at ($(F)+(0,0.5,.3)$) {$z=R$};
\draw[red, thick] ($ (G)+(-1,0,-0.5) $) -- ($ (G)+(1,0,-0.5) $) --
($ (G)+(1,0,0.5) $) -- ($ (G)+(-1,0,0.5) $) -- cycle;
\draw[red, thick] ($ (H)+(-1,0,-0.5) $) -- ($ (H)+(1,0,-0.5) $) --
($ (H)+(1,0,0.5) $) -- ($ (H)+(-1,0,0.5) $) -- cycle;
\draw[red, thick] ($ (H)+(-1,0,-0.5) $) -- ($ (H)+(1,0,-0.5) $) --
($ (G)+(1,0,-0.5) $) -- ($ (G)+(-1,0,-0.5) $) -- cycle;
\draw[red, thick] ($ (H)+(-1,0,0.5) $) -- ($ (H)+(1,0,0.5) $) --
($ (G)+(1,0,0.5) $) -- ($ (G)+(-1,0,0.5) $) -- cycle;
\draw[green, thick] ($ (H)+(-2,0,0.25) $) -- ($ (H)+(2,0,0.25) $) --($ (G)+(2,0,0.25) $) -- ($ (G)+(-2,0,0.25) $)--cycle;
\end{scope}
\end{tikzpicture}
\end{subfigure}
\hspace{1cm}
\begin{subfigure}{0.4\textwidth}
\centering
\begin{tikzpicture}[x={(1cm,0cm)}, y={(0.4cm,0.6cm)}, z={(0cm,1cm)}]
\begin{scope}[xshift=0cm]
% rough surfavce
\coordinate (A) at (-2,-1,0);
\coordinate (B) at (3,-1,0);
\coordinate (C) at (-2,1,0);
\coordinate (D) at (3,1,0);
% interfacve
\coordinate (E) at (-1,-1,2);
\coordinate (F) at (-1,1,2);
\coordinate (G) at (1,-1,0);
\coordinate (H) at (1,1,0);
% upper transition trace
\coordinate (I) at (0.5,-1,0.5);
\coordinate (J) at (0.5,1,0.5);
% upper floor
\coordinate (K) at (-2,-1,2);
\coordinate (L) at (3,-1,2);
\coordinate (M) at (-2,1,2);
\coordinate (N) at (3,1,2);

\draw[black,thick] (A)--(B) (C)--(D);
\draw[black,thick] (K)--(L) (M)--(N);
\draw[blue,thick] (E)--(F) (E)--(G) (F)--(H);
\draw[blue,thick,decorate,decoration={snake,amplitude=0.7mm,segment length=3mm}] (G) --(H);
%\draw[blue,thick,dashed,decorate,decoration={snake,amplitude=0.5mm,segment length=3mm}] (I) --(J);
\node at (-2,0,2) {L};
\node at (1,0,2) {V};
\node at (-1,0,-1) {S};
%\node at ($(F)+(0,0.5,0)$) {$z=R$};
\draw[red, thick] ($ (G)+(-1,0,-0.5) $) -- ($ (G)+(1,0,-0.5) $) --
($ (G)+(1,0,0.5) $) -- ($ (G)+(-1,0,0.5) $) -- cycle;
\draw[red, thick] ($ (H)+(-1,0,-0.5) $) -- ($ (H)+(1,0,-0.5) $) --
($ (H)+(1,0,0.5) $) -- ($ (H)+(-1,0,0.5) $) -- cycle;
\draw[red, thick] ($ (H)+(-1,0,-0.5) $) -- ($ (H)+(1,0,-0.5) $) --
($ (G)+(1,0,-0.5) $) -- ($ (G)+(-1,0,-0.5) $) -- cycle;
\draw[red, thick] ($ (H)+(-1,0,0.5) $) -- ($ (H)+(1,0,0.5) $) --
($ (G)+(1,0,0.5) $) -- ($ (G)+(-1,0,0.5) $) -- cycle;
\draw[green, thick] ($ (H)+(-2,0,0.25) $) -- ($ (H)+(2,0,0.25) $) --($ (G)+(2,0,0.25) $) -- ($ (G)+(-2,0,0.25) $)--cycle;
\end{scope}
\end{tikzpicture}
\end{subfigure}

\vspace{1cm}
\begin{subfigure}{0.4\textwidth}
\centering
\begin{tikzpicture}[x={(1cm,0cm)}, y={(0.4cm,0.6cm)}, z={(0cm,1cm)}]
\begin{scope}[xshift=0cm]
% near-region
\coordinate (A) at (-3,-1,-1);
\coordinate (B) at (2,-1,-1);
\coordinate (C) at (-3,-1,1);
\coordinate (D) at (2,-1,1);
\coordinate (E) at (-3,1,-1);
\coordinate (F) at (2,1,-1);
\coordinate (G) at (-3,1,1);
\coordinate (H) at (2,1,1);
% interface 
\coordinate (I) at (1,-1,-0.5);
\coordinate (J) at (1,1,-0.5);
\coordinate (K) at (-2,-1,1);
\coordinate (L) at (-2,1,1);
% cuting trance of interface
\coordinate (M) at (-1,-1,0.5);
\coordinate (N) at (-1,1,0.5);

\draw[red, thick] (A) -- (B) --
(D) -- (C) -- cycle;
\draw[red, thick] (E) -- (F) --
(H) -- (G) -- cycle;
\draw[red, thick] (A) -- (B) --
(F) -- (E) -- cycle;
\draw[red, thick] (C) -- (D) --
(H) -- (G) -- cycle;
\draw[green, thick, dashed] ($(C)-(0.5,0,0.5)$) -- ($(D)-(-0.5,0,0.5)$) --
($(H)-(-0.5,0,0.5)$) -- ($(G)-(0.5,0,0.5)$) -- cycle;
\draw[black, thick] (-3,-1,-0.5) -- (2,-1,-0.5)--(2,1,-0.5)--(-3,1,-0.5)--cycle;
%\draw[black, thick] (-3,-1,0.5) -- (2,-1,0.5)--(2,1,0.5)--(-3,1,0.5)--cycle;
\draw[blue, thick] (I) -- (K) (J)--(L);
%\draw[blue,thick,decorate,decoration={snake,amplitude=0.1mm,segment length=3mm}] (M) --(N);
\draw[blue,thick,decorate,decoration={snake,amplitude=0.7mm,segment length=3mm}] (I) --(J);
\draw[blue,thick,dashed,decorate,decoration={snake,amplitude=0.5mm,segment length=3mm}] (K) --(L);
\node at (-2.5,0.3,1.2) {L};
\node at (0,0.3,1.2) {V};
\node at (-2,0,-1.3) {S};
\end{scope}
\end{tikzpicture}
\end{subfigure}
\hspace{1cm}
\begin{subfigure}{0.4\textwidth}
\centering
\begin{tikzpicture}[x={(1cm,0cm)}, y={(0.4cm,0.6cm)}, z={(0cm,1cm)}]
\begin{scope}[xshift=-4cm]
% rough surfavce
\coordinate (A) at (-2,-1,0);
\coordinate (B) at (3,-1,0);
\coordinate (C) at (-2,1,0);
\coordinate (D) at (3,1,0);
% interfacve
\coordinate (E) at (-1,-1,2);
\coordinate (F) at (-1,1,2);
\coordinate (G) at (1,-1,0);
\coordinate (H) at (1,1,0);
% upper transition trace
\coordinate (I) at (0.5,-1,0.5);
\coordinate (J) at (0.5,1,0.5);
% upper floor
\coordinate (K) at (-2,-1,2);
\coordinate (L) at (3,-1,2);
\coordinate (M) at (-2,1,2);
\coordinate (N) at (3,1,2);

% Blue box
\draw[black,thick] (A)--(B) (C)--(D);
\draw[black,thick] (K)--(L) (M)--(N);
\draw[blue,thick] (E)--(F) (E)--(G) (F)--(H);
%\draw[blue,thick,decorate,decoration={snake,amplitude=0.7mm,segment length=3mm}] (G) --(H);
\draw[blue,thick,decorate,decoration={snake,amplitude=0.6mm,segment length=3mm}] ($(G)+(-0.25,0,0.25)$) --($(H)+(-0.25,0,0.25)$);
\draw[blue,thick,dashed,decorate,decoration={snake,amplitude=0.5mm,segment length=3mm}] ($(G)+(-0.5,0,0.5)$) --($(H)+(-0.5,0,0.5)$);
\draw[blue,thick,dashed,decorate,decoration={snake,amplitude=0.3mm,segment length=3mm}] ($(G)+(-0.75,0,0.75)$) --($(H)+(-0.75,0,0.75)$);
\draw[blue,thick,dashed,decorate,decoration={snake,amplitude=0.1mm,segment length=3mm}] ($(G)+(-1,0,1)$) --($(H)+(-1,0,1)$);
\draw[blue,thick,dashed,decorate,decoration={snake,amplitude=0.05mm,segment length=3mm}] ($(G)+(-1.5,0,1.5)$) --($(H)+(-1.5,0,1.5)$);
%\draw[blue,thick] (I) --(J);
\node at (-2,0,2) {L};
\node at (1,0,2) {V};
\node at (-1,0,-1) {S};
%\node at ($(F)+(0,0.5,0)$) {$z=R$};
\draw[red, thick,dashed] ($ (G)+(-1,0,-0.5) $) -- ($ (G)+(1,0,-0.5) $) -- ($ (G)+(1,0,0.5) $) -- ($ (G)+(-1,0,0.5) $) -- cycle;
\draw[red, thick,dashed] ($ (H)+(-1,0,-0.5) $) -- ($ (H)+(1,0,-0.5) $) -- ($ (H)+(1,0,0.5) $) -- ($ (H)+(-1,0,0.5) $) -- cycle;
\draw[red, thick,dashed] ($ (H)+(-1,0,-0.5) $) -- ($ (H)+(1,0,-0.5) $) -- ($ (G)+(1,0,-0.5) $) -- ($ (G)+(-1,0,-0.5) $) -- cycle;
\draw[red, thick,dashed] ($ (H)+(-1,0,0.5) $) -- ($ (H)+(1,0,0.5) $) -- ($ (G)+(1,0,0.5) $) -- ($ (G)+(-1,0,0.5) $) -- cycle;
\draw[green, thick] ($ (H)+(-2,0,0.25) $) -- ($ (H)+(2,0,0.25) $) --($ (G)+(2,0,0.25) $) -- ($ (G)+(-2,0,0.25) $)--cycle;
\end{scope}
\end{tikzpicture}
\end{subfigure}
\caption{Diagram for the two-scale alternating (TSA) method for capillary mean curvature flow (CMCF) \eqref{prob:channel}.
(\textbf{top left}) Side view of the channel.
The liquid–vapor interface $\Sigma_{LV}$ is shown in blue, the red box indicates the near-region, and the green line marks the lower boundary of the far-region.
(\textbf{top right}) Angled side view of the channel.
(\textbf{bottom left}) In the near-region, the liquid–vapor interface $\Sigma_{LV}$ evolves using an MBO diffusion-generated method according to \eqref{prob:near_channel}.
(\textbf{bottom right}) In the far-region, the liquid–vapor interface $\Sigma_{LV}$ is constructed using lower boundary data from the previous iteration and the explicit solution to the linearized system given in Theorem~\ref{thm:aprox_channel_min_upper_interface}.}
\label{fig:zoomin_channel}
\end{figure}
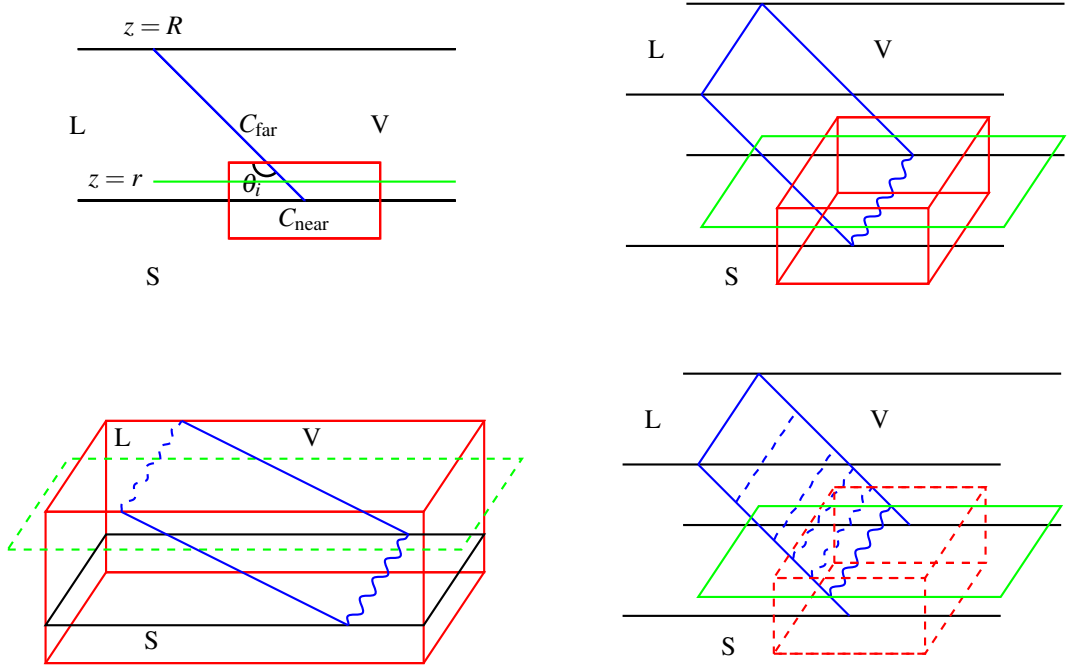

The near-region \(C_{\mathrm{near}}\) is defined, for \(0<r\ll R\) and \(c>1\), by
\[
 C_{\mathrm{near}}:=\{ (x,y,z)\in \mathbb{R}^3\colon  \psi(x,y)<z<cr \},
 \]
and is illustrated by the red box in Figure~\ref{fig:zoomin_channel}. \(c>1\) sets the height of the near-region in units of \(r\), and is fixed to \(c=\frac{6}{5}\) in our implementation.
The far-region \(C_{\mathrm{far}}\) is defined by
\[
C_{\mathrm{far}} :=\{ (x,y,z)\in \mathbb{R}^3\colon  r<z<R \},
\]
which overlaps with $C_{\mathrm{near}}$ and, in Figure~\ref{fig:zoomin_channel}, is illustrated as lying above the green plane.

\subsection{Near-region: MBO diffusion-generated method}\label{sec:wettingMBO}

In the near-region, $C_{\mathrm{near}}$, the interface evolves at a much smaller length scale compared to the far-region and develops complex geometry near the contact line.
Therefore, we will use a fine discretization to capture the local interfacial dynamics.

In this region, we consider the CMCF~\eqref{prob:channel} with the upper Dirichlet condition replaced by a Neumann-like boundary condition.
To implement the boundary condition, we introduce an \emph{``imaginary''} solid region
\[
\mathcal{I} :=\{ (x,y,z)\in \mathbb{R}^3\colon  z\geq cr \},
\]
which is placed strictly above and adjacent to the near-region.
The corresponding imaginary contact angle \(\theta_i\) is estimated from the previous iteration so that it matches the slope of the interface over the far-region \(C_{\mathrm{far}}\).
In practice, we consider the CMCF~\eqref{prob:channel} with \(\mathbf{k}=(1,0)\), where the interface propagates in the \(x\)-direction.
The traces of the interface on the upper boundary of \(C_{\mathrm{far}}\) and the lower boundary of \(C_{\mathrm{far}}\) are extracted on two fixed \(z\)-planes.
The average \(x\)-coordinate of each trace is then computed, and the resulting two representative points are used to estimate the interface slope, thereby determining the imaginary contact angle \(\theta_i\); see Figure~\ref{fig:zoomin_channel} (\textbf{top left}).

Formally, the dynamics is now given by CMCF with a fictitious contact angle condition on the upper boundary,
\begin{align}
\label{prob:near_channel}
    \begin{cases}
        V_n = \kappa, & \text{on } \Sigma_{LV},\\
\theta_{SLV}(\mathbf{x}) = \theta_Y, &\mathbf{x}\in \Gamma_{SLV},\\
        \theta_{\mathcal{I}LV}(\mathbf{x}) = \theta_i,  &\mathbf{x}\in \Gamma_{\mathcal{I}LV}.
\end{cases}
\end{align}

To simulate \eqref{prob:near_channel}, we apply an MBO diffusion-generated method, which is an efficient computational method for evolving interfaces by mean curvature flow; see \cite{MBO1993,Merriman_1994} and further developments in \cite{Esedoglu_2014,Wang_2019,Xu_2017}.
In particular, we adopt a variant developed in \cite{Wang_2019} without a volume constraint.
The method is described as follows.

\paragraph{MBO diffusion-generated method for  \eqref{prob:near_channel}}
Given Young's angle $\theta_Y$ and $\theta_i$, define three time-step lengths,
\(\tau_1>0\),
\begin{equation} \label{equ:tau23}
    \tau_2=\left(\frac{\pi\cos\theta_Y}{\pi-2\theta_Y}\right)^2 \tau_1 \qquad \text{and} \qquad \tau_3=\left(\frac{\pi\cos\theta_i}{\pi-2\theta_i}\right)^2 \tau_1.
\end{equation}
Consider the computational domain \(\Omega_{\mathrm{near}} := C_{\mathrm{near}}\cup S\cup \mathcal{I}\).
Starting from initial liquid and vapor regions, $L^0,V^0\subset C_{\mathrm{near}}$, the iterative scheme is as follows.
For each iteration $k$, set
\begin{equation}\label{equ:phi_in_MBO}
    \begin{aligned}
        \phi
=
\frac{1}{\sqrt{\tau_1}}
\,G_{\tau_1}\ast\bigl(\mathbf{1}_{V^k}-\mathbf{1}_{L^k}\bigr)
-
\frac{1}{\sqrt{\tau_2}}
\,G_{\tau_2}\ast\bigl(\cos\theta_Y\,\mathbf{1}_{S}\bigr)
-
\frac{1}{\sqrt{\tau_3}}
\,G_{\tau_3}\ast\bigl(\cos\theta_i\,\mathbf{1}_{
\mathcal{I}}\bigr),
    \end{aligned}
\end{equation}
where $\,G_{\tau}(\mathbf{x})=\frac{1}{(4\pi \tau)^{n/2}}\exp\left(-\frac{|\mathbf{x}|^2}{4\tau}\right)$ denotes the heat kernel with parameter $\tau$ and
$\mathbf{1} _{D}(\mathbf{x}) = \begin{cases} 1 & \mathbf{x}  \in D \\ 0 & \mathbf{x} \notin D \end{cases}$ is the indicator function of $D$.
The liquid and vapor regions are then updated by thresholding:
\begin{equation} \label{iter:threshold}
L^{k+1}
=
\{\mathbf{x}\in C_{\mathrm{near}} \colon \phi(\mathbf{x})<0\}
\qquad \textrm{and} \qquad
V^{k+1}
=
\{\mathbf{x}\in C_{\mathrm{near}}\colon \phi(\mathbf{x})\ge 0\}.
\end{equation}
Here, since we have no volume constraint, the threshold is chosen to be zero; see \ref{sec:Derivation_MBO} for a derivation.
The algorithm uses a two-level stopping criterion, described in Section~\ref{s:Imple}.
\vspace{2mm}

\paragraph{Analysis of the MBO diffusion-generated method}
This variant of the MBO diffusion-generated method can be  derived from iterative minimization of
the energy \(\mathcal{L}\),
where \(\mathcal{L}\) is a linearized approximation of the total interfacial energy \eqref{equ:interface_energy}.
This derivation is summarized as follows.

When $\tau\ll1$, the area of $\Sigma_{LV}$ can be approximated by
$
|\Sigma_{LV}| \approx \frac{\sqrt{\pi}}{\sqrt{\tau}} \int_{\Omega_{\mathrm{near}}} \mathbf{1}_{L} \,G_{\tau} \ast \mathbf{1}_{V} \, \,d\mathbf{x}$. Similarly, the interface energy~\eqref{equ:interface_energy} can be approximated by
\begin{equation}\label{equ:approx_interface_energy}
\begin{aligned}
&\mathcal{E}^{\tau} (\mathbf{1}_{L},\mathbf{1}_{V};\Omega_{\mathrm{near}}) \ = \
\frac{\sqrt{\pi}\gamma_{LV}}{\sqrt{\tau_1}} \int_{\Omega_{\mathrm{near}}} \mathbf{1}_{L} \,G_{\tau_1} \ast \mathbf{1}_{V} \,d\mathbf{x}
\ + \
\sum_{j=2,3}\frac{\sqrt{\pi}\gamma_{S_j L}}{\sqrt{\tau_j}} \int_{\Omega_{\mathrm{near}}} \mathbf{1}_{L} \,G_{\tau_j} \ast \mathbf{1}_{S_j} \,d\mathbf{x}
\ + \
\sum_{j=2,3}\frac{\sqrt{\pi}\gamma_{S_jV}}{\sqrt{\tau_j}} \int_{\Omega_{\mathrm{near}}} \mathbf{1}_{V} \,G_{\tau_j} \ast \mathbf{1}_{S_j} \,d\mathbf{x},
\end{aligned}
\end{equation}
where \(S_2=S\), \(S_3=\mathcal{I}\).
Small $\tau_1>0$ is used in convolution to approximate $\Sigma_{LV}$.
Similarly, $\tau_2$ is used to approximate $\Sigma_{SL}$ and $\Sigma_{SV}$ and $\tau_3$ is used to approximate $\Sigma_{\mathcal{I}L}$ and $\Sigma_{\mathcal{I}V}$.
Here the superscript \(\tau\) indicates the dependence of the energy, $\mathcal{E}^{\tau}$, on \(\tau_1\), \(\tau_2\) and \(\tau_3\).
The specific choice of \(\tau_2\) and \(\tau_3\) ensures that the correct Young's angle is obtained near the contact line in the MBO method; see~\cite[Section~3]{Wang_2019}.

Define the admissible set $\mathcal{B} =\left\{(u_1,u_2)\in BV(C_{\mathrm{near}}) \colon u_i(\mathbf{x}) \in \{0,1\} \text{ and } u_1(\mathbf{x})+u_2(\mathbf{x})=1,\text{ a.e. } \mathbf{x} \in C_{\mathrm{near}} \right\}$.
The interfacial energy~\eqref{equ:interface_energy} is approximated by $\mathcal{E}^{\tau} (u_1,u_2; \Omega_{\mathrm{near}})$ for ${(u_1,u_2)\in \mathcal{B}}$.
Further, we consider the convex hull $\mathcal{K} =\{(u_1,u_2)\in BV(C_{\mathrm{near}}) : u_i(\mathbf{x}) \in [0,1], u_1(\mathbf{x})+u_2(\mathbf{x})=1 \text{ a.e. } \mathbf{x} \in C_{\mathrm{near}}\}$ and consider the relaxed problem
\begin{align} \label{prob:relaxed_wetting}
\min_{(u_1,u_2)\in \mathcal{K}}\mathcal{E}^{\tau} (u_1,u_2; \Omega_{\mathrm{near}}).
\end{align}

\begin{lemma}\label{lem:char_min} We have
$$
\min_{(u_1,u_2)\in \mathcal{K}} \mathcal{E}^{\tau} (u_1,u_2; \Omega_{\mathrm{near}}) = \min_{(u_1,u_2)\in \mathcal{B}} \mathcal{E}^{\tau} (u_1,u_2; \Omega_{\mathrm{near}}).
$$
\end{lemma}

A proof of Lemma~\ref{lem:char_min} under the volume constraint is given in \cite[Lemma~2.1]{Wang_2019} and is easily adapted to our setting.
For $(u_1,u_2)\in\mathcal{K}$, the expansion of the energy $\mathcal{E}^{\tau}$ in \eqref{equ:approx_interface_energy} about \((u_1^k,u_2^k)\in \mathcal{K}\) is given by
\begin{align}
\label{equ:expansion_interface_energy}
\mathcal{E}^{\tau} (u_1,u_2;\Omega_{\mathrm{near}})
\ =  \
\mathcal{E}^{\tau} (u^k_1,u^k_2;\Omega_{\mathrm{near}})
\ + \
{\mathcal{L}}(u_1-u^k_1,u_2-u^k_2, u^k_1,u^k_2;\Omega_{\mathrm{near}})
\ + \
H(u_1-u^k_1,u_2-u^k_2;\Omega_{\mathrm{near}})
\end{align}
where
{\small
\begin{equation}\label{equ:linearied_interface_energy}
 \begin{aligned}
&{\mathcal{L}}(v_1,v_2, u^k_1,u^k_2; \Omega_{\mathrm{near}}) :=\sqrt{\pi}\int_{\Omega_{\mathrm{near}}} v_1\left(\frac{\gamma_{LV}}{\sqrt{\tau_1}} \,G_{\tau_1} \ast u^k_2 +\sum_{j=2,3}\frac{\gamma_{S_j L}}{\sqrt{\tau_j}}\,G_{\tau_j}\ast \mathbf{1}_{S_j}\right) \, \,d\mathbf{x}
\, + \,
\sqrt{\pi}\int_{\Omega_{\mathrm{near}}} v_2 \left(\frac{\gamma_{LV}}{\sqrt{\tau_1}}\,G_{\tau_1} \ast u^k_1 +\sum_{j=2,3}\frac{\gamma_{S_jV}}{\sqrt{\tau_j}} \,G_{\tau_j} \ast\mathbf{1}_{S_j}\right) \, \,d\mathbf{x},
\end{aligned}
\end{equation}}
and
\begin{align*}
    H(v_1,v_2; \Omega_{\mathrm{near}})&:= \frac{-\sqrt{\pi}\gamma_{LV}}{2\sqrt{\tau_1}} \int_{\Omega_{\mathrm{near}}} \left( v_1 \,G_{\tau_1}\ast v_1 + v_2 \,G_{\tau_1}\ast v_2 \right) \,d\mathbf{x}
    \, \, \leq  \, \, 0.
\end{align*}
In \cite{Wang_2019}, it was shown that the minimizer of \eqref{equ:linearied_interface_energy} under a volume constraint decreases the approximate energy \eqref{equ:approx_interface_energy} monotonically.
By a similar approach, we are able to prove the same energy-decreasing property, namely:

\begin{theorem} \label{thm:stablity_wetMBO}
Denote by $(u^k_1,u^k_2) = (\mathbf{1}_{L^k},\mathbf{1}_{V^k})$,
$k=0,1,2,\ldots$
the indicator functions obtained from the MBO diffusion-generated method for the CMCF~\eqref{prob:near_channel}. We have the energy-decreasing property
    $$\mathcal{E}^{\tau}(u^{k+1}_1,u^{k+1}_2; \Omega_{\mathrm{near}})\leq \mathcal{E}^{\tau}(u^k_1,u^k_2; \Omega_{\mathrm{near}}),$$ for $\tau_1>0$.
\end{theorem}
A proof of Theorem~\ref{thm:stablity_wetMBO} is provided in \ref{proof:stability_analysis}; here we sketch the idea.
Lemma~\ref{lem:iter_sol} in \ref{sec:Derivation_MBO}  gives that
$\mathcal{L}(u^{k}_1,u^{k}_2,u^{k}_1,u^{k}_2; \Omega_{\mathrm{near}})\geq \mathcal{L}(u^{k+1}_1,u^{k+1}_2,u^{k}_1,u^{k}_2; \Omega_{\mathrm{near}})$.
This gives a lower bound for the energy difference \(\mathcal{E}^{\tau}(u^k_1,u^k_2; \Omega_{\mathrm{near}})-\mathcal{E}^{\tau}(u^{k+1}_1,u^{k+1}_2; \Omega_{\mathrm{near}})\), which is nonnegative under the constraint $u^{k+1}_1+u^{k+1}_2=u^{k}_1+u^{k}_2=\mathbf{1}_{C_{\mathrm{near}}}$.

With the energy formulation \(\mathcal{E}^\tau\), we are now ready to introduce Algorithm~\ref{alg:TSA}.
In consideration of the interface in the global domain \(\bar{\Omega} =C_{\mathrm{near}}\cup C_{\mathrm{far}} \cup S\), we define relaxed admissible set \(\bar{\mathcal{K}}\) by
\[
\bar{\mathcal{K}} =\{(u_1,u_2)\in BV(C_{\mathrm{near}}\cup C_{\mathrm{far}}) : u_i(\mathbf{x}) \in [0,1], u_1(\mathbf{x})+u_2(\mathbf{x})=1 \text{ a.e. } \mathbf{x} \in C_{\mathrm{near}}\cup C_{\mathrm{far}}\}.
\]
For $(v_1,v_2)\in\bar{\mathcal{K}}$, we abbreviate the two constraint sets that appear in the alternating iteration by
\begin{align*}
\bar{\mathcal{K}}_{\mathrm{near}}(v_1,v_2)
&:=\{(u_1,u_2)\in \bar{\mathcal{K}} : (u_1,u_2)=(v_1,v_2) \text{ on } C_{\mathrm{far}}\setminus C_{\mathrm{near}}\}, \\
\bar{\mathcal{K}}_{\mathrm{far}}(v_1,v_2)
&:=\{(u_1,u_2)\in \bar{\mathcal{K}} : (u_1,u_2)=(v_1,v_2) \text{ on } C_{\mathrm{near}}\setminus C_{\mathrm{far}}\}.
\end{align*}
Also, we emphasize that the energy \(\mathcal{E}\) in Step~\ref{iter:step2_TSA} refers to the global interfacial energy~\eqref{equ:interface_energy} over \(\bar{\Omega}\).
Step~\ref{iter:step2_TSA} is implemented by constructing an approximate minimal surface, as described in the next section.

\begin{algorithm}[t]
\caption{Two-scale alternating (TSA) method.}
\label{alg:TSA}
\begin{algorithmic}[1]
\Require Young's angle $\theta_Y$; base time step $\tau_1>0$, with $\tau_2$ as in
  \eqref{equ:tau23}; tolerance $\mathrm{tol}>0$; initial configuration
  $(u_1^0,u_2^0)\in\bar{\mathcal{K}}$.
\For{$n=0,1,2,\ldots$}
  \State Update the imaginary angle $\theta_i^{n}$ from the traces of
    $(u_1^{2n},u_2^{2n})$, as in Section~\ref{sec:wettingMBO}, and set $\tau_3$
    by \eqref{equ:tau23}.\label{iter:step0_TSA}
  \State Near-region update (MBO):\label{iter:step1_TSA}
  \Statex \qquad $\displaystyle
    (u_1^{2n+1},u_2^{2n+1}) =
    \arg\min_{(u_1,u_2)\in\bar{\mathcal{K}}_{\mathrm{near}}(u_1^{2n},u_2^{2n})}
    \mathcal{L}\bigl(u_1-u_1^{2n},u_2-u_2^{2n},u_1^{2n},u_2^{2n};\Omega_{\mathrm{near}}\bigr)$
  \State Far-region update (minimal surface):\label{iter:step2_TSA}
  \Statex \qquad $\displaystyle
    (u_1^{2n+2},u_2^{2n+2}) =
    \arg\min_{(u_1,u_2)\in\bar{\mathcal{K}}_{\mathrm{far}}(u_1^{2n+1},u_2^{2n+1})}
    \mathcal{E}(u_1,u_2;\bar{\Omega})$
\State \algorithmicif\ $\bigl\|u_1^{2n+2}-u_1^{2n}\bigr\|_{L^1(C_{\mathrm{near}})}<\mathrm{tol}$ \algorithmicthen\ \textbf{break}
\EndFor
\Ensure Approximate stationary configuration $(u_1^{2n+2},u_2^{2n+2})$ for the
  time steps $(\tau_1,\tau_2,\tau_3)$.
\end{algorithmic}
\end{algorithm}

In the numerical implementation, this inner iteration is repeated within an outer time-step refinement procedure, as described in Section~\ref{sec:TSA_overview}.

Although $\theta_i^{n}$ does not appear explicitly in $\mathcal{L}$, it enters
Step~\ref{iter:step1_TSA} through the time step $\tau_3$ set in
Step~\ref{iter:step0_TSA}, which weights the imaginary solid $\mathcal{I}$ in
\eqref{equ:linearied_interface_energy}.
In Step~\ref{iter:step1_TSA}, all arguments of \(\mathcal{L}\) are understood to be restricted to \(C_{\mathrm{near}}\).

\subsection{Far-region: minimal surface approximation} \label{sec:condition_over_far}

In the far-region, \(C_{\mathrm{far}}\), the interface geometry is relatively smooth.
Therefore, instead of evolving according to CMCF~\eqref{prob:channel}, we use the corresponding stationary state to update the interface in the far-region at each iteration.

The stationary far-region interface is a minimal surface, with a Dirichlet boundary condition on the lower boundary of \(\partial C_{\mathrm{far}}\) provided by the updated near-region interface \(\Sigma_{LV}^k\) from the previous iteration.
For notational simplicity, we consider the stationary solution when
\( \mathbf{k}= \left(\begin{smallmatrix} 1  \\ 0 \end{smallmatrix} \right) \). Formally,
\begin{align*}
    \begin{cases}
        \kappa=0, & \text{on } \Sigma_{LV},\\[1mm]
        \Sigma_{LV} = \Sigma_{LV}^{k}, & \text{on } \partial C_{\mathrm{far}},\\
        L \cap \{z=R\} =  \left\{(x,y,R) \colon
x \leq 0 \right\}.
    \end{cases}
\end{align*}
This construction provides a smooth macroscopic interface that connects naturally to the finely resolved near-region interface. This is still a nonlinear PDE. Because the far-region interface remains close to a plane, we linearize the minimal-surface equation about its mean slope.

We first introduce the graphical representation. If a surface area minimizer in $\{z \geq 0\}$ is sandwiched between two parallel planes, as in \eqref{equ:flat_like},  then for $C$ depending only on the flatness hypothesis, it is a smooth subgraph in $\{z \geq C\}$. Further if the surface is periodic in the $y$-variable, then it asymptotes smoothly, and with exponential rate, to a plane with angle \(\theta\) as $z \to \infty$. See Proposition~\ref{prop:regularity-of-planelikes} in~\ref{sec:Justification_linearization} for details.
Motivated by this, we assume that the liquid–vapor interface $\Sigma_{LV}$ is given by the graph $ x(y,z)$
away from the contact line, defined over the domain
$\{(y,z)\colon  y \in [-l_y, l_y],\ z \in [r, R]\}$
or
$\{(y,z)\colon y \in [-l_y, l_y],\ z \in [r, \infty)\}$.
In these coordinates, the \emph{graphical minimal surface equation} is then
\begin{equation} \label{equ:min_upper_interface}
\left\{
\begin{aligned}
&\left(1 + x_{y}^{2} \right)\,x_{zz} - 2x_{y}x_{z}x_{yz} + \left(1 + x_{z}^{2}\right)\,x_{yy} = 0, \\
&x(y,r) = f(y),
&&x(y,R) = 0,\\
&x(-l_y,z) = x(l_y,z),
&&x_y(-l_y,z) = x_y(l_y,z),
 \end{aligned}
 \right.
\end{equation}
where \(f=\Sigma_{LV}^{k}|_{\partial C_{\mathrm{far}}}\). There are two scales in system~\eqref{equ:min_upper_interface}, a large \(R\gg 1\) and a relatively small amplitude \(\eta\), defined in Theorem~\ref{thm:aprox_channel_min_upper_interface}. Specifically, we may assume that \(x(y,z)\) differs only slightly from a flat plane:
\[
\|x(y,z)-m(z-R)\|_{L^\infty}\leq C\eta
\]
with some constants \(m\) and \(C,
\eta>0\). Therefore, we may consider the following expansion and apply the corresponding asymptotic analysis.

\begin{theorem}
\label{thm:aprox_channel_min_upper_interface}
Consider the expansion of the liquid–vapor interface \(x(y,z)\):
\[x(y,z) =m(z-R)+ u(y,z)\eta+o(\eta),\]
where
\begin{equation*}
    m= \frac{1}{(r-R)} \dashint_{-l_y}^{l_y} f(y)\,dy \quad \text{ and }\quad  \eta = \|f-m(r-R)\|_{\infty}.
\end{equation*}
Under this expansion, system~\eqref{equ:min_upper_interface} is approximated at leading order \(O(\eta)\) by
\begin{equation}
\label{equ:aprox_channel_min_upper_interface}
\left\{
\begin{aligned}
&u_{zz}+\left(1+m^2\right)u_{yy}=0, \\
&u(y,r) = \frac{1}{\eta} \left(f(y)-m(r-R) \right),
&& u(y,R) = 0, \\
&u(-l_y,z) = u(l_y,z),
&& u_y(-l_y,z) = u_y(l_y,z).
\end{aligned}
\right.
\end{equation}
This linear problem admits the following Fourier series solution

\begin{align*}
    u(y,z)&=\frac{1}{\eta}\sum_{n=1}^{\infty} \left[ A_n\cos\left(\pi n\frac{y+l_y}{l_y}\right)
    +B_n\sin\left( \pi n\frac{y+l_y}{l_y} \right) \right]\frac{1}{\sinh\left(\frac{2\pi nh}{l} \right)}\sinh\left(\pi n \sqrt{1+m^2}\frac{R-z}{l_y}\right),
\end{align*}
where \(h=R-r\), \(l=\frac{2 l_y}{\sqrt{1+m^2}}\), \(g(y)=f(\sqrt{1+m^2}y-l_y)-m(r-R)\) and the Fourier coefficients are given by
\begin{align*}
A_n=\frac{2}{l}
\int_0^l g(y)\cos\left(\frac{2\pi n y}{l}\right) \, dy
\qquad \textrm{and} \qquad
B_n=\frac{2}{l}
\int_0^l g(y)\sin\left(\frac{2\pi n y}{l} \right)\, dy.
\end{align*}
Finally, the liquid–vapor interface is given by
\begin{align}\nonumber
    x(y,z) &=m(z-R)+\eta u(y,z)\\
    \label{equ:far_region_sol}
    &=m(z-R)+\sum_{n=1}^{\infty} \left[A_n\cos\left(\pi n\frac{y+l_y}{l_y} \right)
    +B_n\sin \left(\pi n\frac{y+l_y}{l_y}\right) \right] \frac{1}{\sinh\left(\frac{2\pi nh}{l} \right)}
    \sinh\left(\pi n \sqrt{1+m^2}\,\frac{R-z}{l_y}\right).
\end{align}
\end{theorem}

\begin{proof}
Following the ansatz that $x(y,z)=m(z-R)+\eta u(y,z)+o(\eta)$, we rewrite $$(1 + x_{y}^{2})x_{zz} - 2x_{y}x_{z}x_{yz} + (1 + x_{z}^{2})x_{yy} = 0$$ as
$$
\left(1+\eta^2 u_y^2\right)\eta u_{zz}
\ - \
2\eta^2 u_y \left(m+\eta u_z\right) u_{yz}
\ + \
\left(1+(m+\eta u_{z})^2\right) \eta u_{yy}
\ = \ o(\eta).
$$
Matching the $O(\eta)$ term yields $u_{zz}+(1+m^2)u_{yy}=0.$ Therefore, system~\eqref{equ:min_upper_interface} is approximated by the linear system~\eqref{equ:aprox_channel_min_upper_interface} after rewriting all boundary conditions accordingly.
Applying the periodic boundary conditions and expanding the lower-boundary data in Fourier modes yields the stated solution.
\end{proof}

\begin{remark} \label{rem:exp_decay}
In practice, \(x(y,z)\) is approximated by truncating the series at
\(n=N_{\mathrm{F}}\) for some \(N_{\mathrm{F}}\in\mathbb{N}\).
After combining the \(\sinh\) factor in
\eqref{equ:far_region_sol} with the denominator in the Fourier
coefficients, the \(n\)th Fourier mode is exponentially damped at the rate
\(
\exp\left(
-2\pi n\frac{z-r}{l}
\right)
\)
for each fixed \(z>r\). The truncation parameter \(N_{\mathrm{F}}\) is
chosen accordingly.
\end{remark}

\subsection{Energy-slope relation of the macroscopic interface}\label{sec:energy_to_slope}
In the channel setting~\eqref{prob:channel}, the liquid–vapor interface is expected to remain close to a planar interface, up to an \(O(e^{-cz})\) correction induced by the surface roughness; see Figure~\ref{fig:energy_slope}.
This perturbation decays exponentially in the \(z\)-direction; see \ref{sec:Justification_linearization} for justification in the elliptic setting.
We expect similar behavior in this parabolic setting when the initial data are suitably chosen.
Accordingly, the total interfacial energy \(\mathcal{E}\) may be decomposed into two parts: the leading-order \(O(R)\) contribution
\(F\) from the planar interface and an \(O(1)\) correction induced by the microscopic roughness.
In this subsection, following this energy asymptotic computation, we provide a heuristic interpretation of the relation between the interfacial energy and the macroscopic slope.

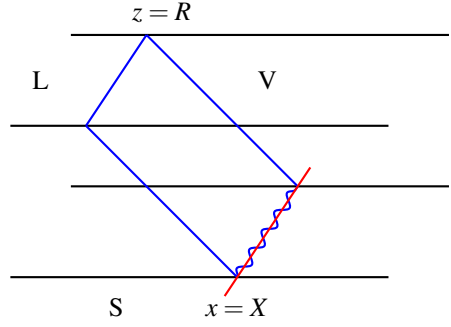
\begin{figure}[t]
    \centering
    \begin{tikzpicture}[x={(1cm,0cm)}, y={(0.4cm,0.6cm)}, z={(0cm,1cm)}]
\begin{scope}[xshift=0cm]
% rough surfavce
\coordinate (A) at (-2,-1,0);
\coordinate (B) at (3,-1,0);
\coordinate (C) at (-2,1,0);
\coordinate (D) at (3,1,0);
% interfacve
\coordinate (E) at (-1,-1,2);
\coordinate (F) at (-1,1,2);
\coordinate (G) at (1,-1,0);
\coordinate (H) at (1,1,0);
% upper transition trace
\coordinate (I) at (0.5,-1,0.5);
\coordinate (J) at (0.5,1,0.5);
% upper floor
\coordinate (K) at (-2,-1,2);
\coordinate (L) at (3,-1,2);
\coordinate (M) at (-2,1,2);
\coordinate (N) at (3,1,2);

\draw[black,thick] (A)--(B) (C)--(D);
\draw[black,thick] (K)--(L) (M)--(N);
\draw[blue,thick] (E)--(F) (E)--(G) (F)--(H);
\draw[blue,thick,decorate,decoration={snake,amplitude=0.7mm,segment length=3mm}] (G) --(H);
%\draw[blue,thick,dashed,decorate,decoration={snake,amplitude=0.5mm,segment length=3mm}] (I) --(J);
\draw[red,thick,shorten >=-3mm,shorten <=-3mm] (H)--(G);
\node at (-2,0,2) {L};
\node at (1,0,2) {V};
\node at (-1,0,-1) {S};
\node at ($(F)+(0,0.5,0)$) {$z=R$};
\node at (1,-1,-0.4) {$x=X$};
\end{scope}
\end{tikzpicture}
    \caption{Dependence of interfacial energy on macroscopic slope}
    \label{fig:energy_slope}
\end{figure}

Consider the channel setting shown in Figure~\ref{fig:energy_slope}, where the \(y\)-periodic cell has unit length for simplicity.
We denote \(X\) as the averaged \(x\)-position of the contact line and define the macroscopic slope
\(
m=-\frac{X}{R},
\)
similarly to the definition in Section~\ref{sec:condition_over_far}.
Up to a constant difference, the leading-order energy \(F\) can be written as
\[
F = \gamma_{LV}\sqrt{R^2+X^2} + (\bar{\gamma}_{SL}-\bar{\gamma}_{SV})X,
\]
where \(\bar{\gamma}_{SL}\) and \(\bar{\gamma}_{SV}\) denote effective energy densities at the macroscopic scale.
These quantities generally depend on the microscopic roughness and wetting configuration, see \cite{Alberti_2005,Feldman_2018} for more details.
Their detailed characterization is not needed for the following argument, and we treat them as constants for simplicity.
We denote \(\bar{\theta}_Y\) as the effective Young's angle satisfying \(\cos(\bar{\theta}_Y)=\frac{\bar{\gamma}_{SV}-\bar{\gamma}_{SL}}{\gamma_{LV}}\), which is conjecturally always inside the pinning interval.
Using the definition of \(m\), the dependence of \(F\) on \(m\) can be written as
\[
F(m)=\gamma_{LV}\Bigl(\sqrt{1+m^2}+\cos(\bar{\theta}_Y)m\Bigr)R.
\]

We note that \(F(m)\) has a unique minimizer at \(m=-\cot\bar{\theta}_Y\), and is strictly decreasing on \((-\infty,-\cot\bar{\theta}_Y]\) and strictly increasing on \([-\cot\bar{\theta}_Y,\infty)\).
Therefore, if the macroscopic slope \(m\) remains on one of the two monotone branches, minimizing the leading-order \(O(R)\) contribution to the interfacial energy drives \(m\) monotonically along that branch until the contact line pins.
The value at which \(m\) halts then determines the apparent contact angle of the maximal or minimal stationary state.
Consequently, convergence of \(m\) in both the advancing and receding processes implies convergence of the corresponding apparent contact angles and hence of the hysteresis interval.
The remaining \(O(1)\) contribution induced by the microscopic roughness may cause small oscillations around this overall trend.

This observation motivates us to study the energy structure underlying the TSA method.
Since the macroscopic slope \(m\) is the quantity that determines the apparent contact angle and hence the hysteresis interval, it is desirable to formulate the alternating procedure with respect to a consistent energy functional.
In the following, we introduce a reference formulation of the TSA method based on a single interfacial energy.

\subsection{Reference formulation of the two-scale alternating (TSA) method} \label{sec:energy_formulation}

In the TSA method (Algorithm~\ref{alg:TSA}), the iteration alternates between minimizing two different energies: the localized energy \(\mathcal{E}^\tau\), associated with interfaces in \(\Omega_{\mathrm{near}}\), and the global interfacial energy \(\mathcal{E}\), associated with interfaces in \(\bar{\Omega}=C_{\mathrm{near}}\cup C_{\mathrm{far}}\cup S\).
This distinction is natural from the implementation point of view, as the two energies correspond to the respective updates performed in the near- and far-regions.
However, because the two updates minimize different energies, an energy principle for the complete iteration is not immediate.

In this subsection, we link the two energies via an intermediate approximation. We therefore introduce a reformulation of the TSA iteration that admits an energy analysis.

\paragraph{Heuristic model in 2D}
We first consider a heuristic model: a two-dimensional CMCF in a finite channel of height \(R\gg1\), where the liquid–vapor interface \(\Sigma_{LV}\) is simply a line with interfacial energy \(\mathcal{E}\) defined in \eqref{equ:interface_energy}; see Figure~\ref{fig:heuristic} (\textbf{left}).

The interface evolves in the entire channel, under the assumption that \(\Sigma_{LV}\) remains nearly straight above a positive height \(r\sim 1\), as motivated by the asymptotic analysis
of \ref{sec:Justification_linearization}.
Under this assumption, we introduce an imaginary solid region \(\mathcal{I}\) above \(\{z=r\}\) with constant energy densities \(\gamma_{\mathcal{I}L},\gamma_{\mathcal{I}V}\) such that \(\cos(\theta_i) =\frac{-x}{\sqrt{R^2+x^2}}= \frac{\gamma_{\mathcal{I}V}-\gamma_{\mathcal{I}L}}{ \gamma_{LV}}\); see Figure~\ref{fig:heuristic} (\textbf{right}).
We consider the induced localized interfacial energy
\[
\bar{\mathcal{E}}
 \, = \,
 \gamma_{LV} |\Sigma_{LV}\cap {\{z\leq r\}}|
 \ + \
 \gamma_{\mathcal{I}L} |\Sigma_{\mathcal{I}L}|
 \ + \
 \gamma_{\mathcal{I}V} |\Sigma_{\mathcal{I}V}|
 \ + \
 \int_{\Sigma_{SL}} \gamma_{SL}(\mathbf{x} ) \,d\mathbf{x}
 \ + \
 \int_{\Sigma_{SV}} \gamma_{SV}(\mathbf{x} )
 \, \,d\mathbf{x}.
\]

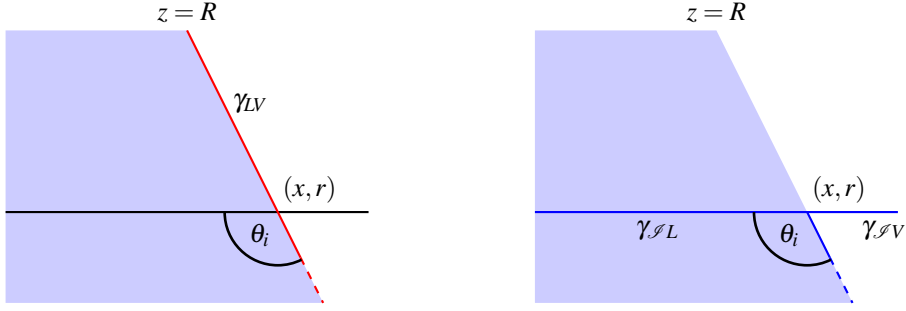
\begin{figure}[t]
\centering
\hspace{-10mm}
\begin{tikzpicture}
\begin{scope}[xshift=0cm, scale=0.6]
\begin{scope}
\fill[blue!20]
(-4,-2) -- (3,-2) -- (0,4)-- (-4,4) -- cycle; 
\end{scope}

% Imaginaey solid plane: y = 0
\draw[thick] (-4,0) -- (4,0);

% Interface
\draw[red, thick] (0,4) -- (2.5,-1);

\draw[red, thick, dashed] (2.5,-1) -- (3,-2);

% Imaginaey contact angle θ_i
\coordinate (O) at (2,0);
\coordinate (S) at (0,0); 
\coordinate (I) at (2.5,-1);   
\draw pic[draw=black,
line width=1.0pt,   
angle radius=20pt,
angle eccentricity=1.2]
{angle = S--O--I};

\node at (1.6, -1/2) {$\theta_i$};
\node at (2.7, 1/2) {$(x,r)$};
\node at (0.0, 4.4) {$z=R$};
\node at (1.4, 2.4) {$\gamma_{LV}$};
\end{scope}

\begin{scope}[xshift=7cm, scale=0.6]
\begin{scope}
\fill[blue!20]
(-4,-2) -- (3,-2) -- (0,4)-- (-4,4) -- cycle; 
\end{scope}

% Imaginaey solid plane: y = 0
\draw[blue, thick] (-4,0) -- (4,0);

% Interface
% \draw[blue, thick, dashed] (0,4) -- (2.0,0);

\draw[blue, thick] (2,0) -- (2.5,-1);
\draw[blue, thick, dashed] (2.5,-1) -- (3,-2);

% Imaginaey contact angle θ_i
\coordinate (O) at (2,0);
\coordinate (S) at (0,0); 
\coordinate (I) at (2.5,-1);   
\draw pic[draw=black,
line width=1.0pt,   
angle radius=20pt,
angle eccentricity=1.2]
{angle = S--O--I};

\node at (1.6, -1/2) {$\theta_i$};
\node at (2.7, 1/2) {$(x,r)$};
\node at (0.0, 4.4) {$z=R$};
\node at (3.7, -0.4) {$\gamma_{{\mathcal{I}}V}$};
\node at (-1.3, -0.4) {$\gamma_{{\mathcal{I}}L}$};
\end{scope}
\end{tikzpicture}
\caption{
{\bf (left)} 2D CMCF in a finite channel. The red line represents the interface associated with the global interfacial energy \(\mathcal{E}\).
{\bf (right)} 2D CMCF with imaginary solid \(\mathcal{I}\). The blue line represents the interface associated with the localized interfacial energy \(\bar{\mathcal{E}}\).
}
\label{fig:heuristic}
\end{figure}

A direct calculation shows that the first variations of \(\mathcal{E}\) and \(\bar{\mathcal{E}}\) with respect to \(x\) coincide. Formally,
\begin{equation*}
\begin{aligned}
    \mathcal{E}-\bar{\mathcal{E}}
    &= \gamma_{LV} |\Sigma_{LV}\cap {\{z\geq r\}}|
 \ - \
 \gamma_{\mathcal{I}L} |\Sigma_{\mathcal{I}L}|
 \ - \
 \gamma_{\mathcal{I}V} |\Sigma_{\mathcal{I}V}| \\
    \frac{d}{dx} \Bigl(\mathcal{E}-\bar{\mathcal{E}}\Bigr)&=\frac{d}{dx}\sqrt{R^2+x^2}\gamma_{LV}-\gamma_{\mathcal{I}L} +\gamma_{\mathcal{I}V}\\
    &=-\cos(\theta_i)\gamma_{LV}-\gamma_{\mathcal{I}L} +\gamma_{\mathcal{I}V} \\
    &= - (\gamma_{\mathcal{I}V}- \gamma_{\mathcal{I}L})  -\gamma_{\mathcal{I}L} +\gamma_{\mathcal{I}V} \\
    &=0.
\end{aligned}
\end{equation*}

Here we have written the length of \(\Sigma_{LV}\cap\{z\geq r\}\) as \(\sqrt{R^2+x^2}\), which is the leading-order expression since \(r\sim 1\ll R\).
This suggests that the localized interfacial energy \(\bar{\mathcal{E}}\) approximates the global interfacial energy \(\mathcal{E}\) at the level of first variation, in the sense that the linearizations of the two energies with respect to the perturbation parameter \(x\) coincide.

Motivated by this two-dimensional heuristic example, we expect that the localized energy \(\mathcal{E}^\tau\) provides an approximation of the global interfacial energy \(\mathcal{E}^{\tau_1,\tau_2}\) defined on \(\bar{\Omega}\) without the imaginary solid region:
\begin{equation*}
\begin{aligned}
&\mathcal{E}^{\tau_1,\tau_2} (\mathbf{1}_{L},\mathbf{1}_{V}; \bar{\Omega}) \ = \
\frac{\sqrt{\pi}\gamma_{LV}}{\sqrt{\tau_1}} \int_{\bar{\Omega}} \mathbf{1}_{L} \,G_{\tau_1} \ast \mathbf{1}_{V} \,d\mathbf{x}
\ + \
\frac{\sqrt{\pi}\gamma_{S L}}{\sqrt{\tau_2}} \int_{\bar{\Omega}} \mathbf{1}_{L} \,G_{\tau_2} \ast \mathbf{1}_{S} \,d\mathbf{x}
\ + \
\frac{\sqrt{\pi}\gamma_{S V}}{\sqrt{\tau_2}} \int_{\bar{\Omega}} \mathbf{1}_{V} \,G_{\tau_2} \ast \mathbf{1}_{S} \,d\mathbf{x}.
\end{aligned}
\end{equation*}
The corresponding linearized functional \({\mathcal{L}^{\tau_1,\tau_2}}\) is defined by

\begin{equation*}
 \begin{aligned}
&{\mathcal{L}^{\tau_1,\tau_2}}(v_1,v_2, u^k_1,u^k_2;
\bar{\Omega}) :=\sqrt{\pi}\int_{\bar{\Omega}} v_1\left(\frac{\gamma_{LV}}{\sqrt{\tau_1}} \,G_{\tau_1} \ast u^k_2 +\frac{\gamma_{S L}}{\sqrt{\tau_2}}\,G_{\tau_2}\ast \mathbf{1}_{S}\right) \, \,d\mathbf{x}
\, + \,
\sqrt{\pi}\int_{\bar{\Omega}} v_2 \left(\frac{\gamma_{LV}}{\sqrt{\tau_1}}\,G_{\tau_1} \ast u^k_1 +\frac{\gamma_{SV}}{\sqrt{\tau_2}} \,G_{\tau_2} \ast\mathbf{1}_{S}\right) \, \,d\mathbf{x}.
\end{aligned}
\end{equation*}

For small \(\tau_1 > 0\), it is clear that \(\mathcal{E}^{\tau_1,\tau_2}\) approximates \(\mathcal{E}\).
Therefore, \(\mathcal{E}^{\tau_1,\tau_2}\) can be viewed as an intermediate approximation between \(\mathcal{E}^\tau\) and \(\mathcal{E}\).
This motivates the following reference iteration, formulated entirely in terms of \(\mathcal{E}^{\tau_1,\tau_2}\).
Replacing \(\mathcal{L}\) by \(\mathcal{L}^{\tau_1,\tau_2}\) and \(\mathcal{E}\) by \(\mathcal{E}^{\tau_1,\tau_2}\), we obtain the following reformulation of the TSA iteration:

\begin{align}
(u_1^{2n+1}, u_2^{2n+1})
&= \arg\min_{(u_1,u_2)\in \bar{\mathcal{K}}_{\mathrm{near}}(u_1^{2n},u_2^{2n})} \mathcal{L}^{\tau_1,\tau_2}(u_1-u_1^{2n},u_2-u_2^{2n},u_1^{2n}, u_2^{2n}; \bar{\Omega})
\label{iter:step1_schwarz} \\
(u_1^{2n+2}, u_2^{2n+2})
&= \arg\min_{(u_1,u_2)\in \bar{\mathcal{K}}_{\mathrm{far}}(u_1^{2n+1},u_2^{2n+1})} \mathcal{E}^{\tau_1, \tau_2}( u_1,u_2; \bar{\Omega}).
\label{iter:step2_schwarz}
\end{align}

\begin{remark}
Note that the Neumann-like boundary condition in Step~\ref{iter:step1_TSA} is replaced by a Dirichlet boundary condition in \eqref{iter:step1_schwarz}.
The fast implementation of the MBO scheme relies on the convolution (FFT) with the heat kernel, which is natural under periodic or Neumann-like boundary conditions as in \eqref{prob:near_channel}.
Moreover, our diffusion step is realized through a multiphase convolution process~\eqref{equ:phi_in_MBO}.
Since this realization is not formulated through a boundary-value diffusion PDE, Dirichlet boundary conditions cannot be imposed directly.
\end{remark}

One of the important features of the Schwarz alternating method applied to minimization problems is its energy-decaying property.
The following theorem states that reference iterations \eqref{iter:step1_schwarz} and \eqref{iter:step2_schwarz} also satisfy an energy decaying property.

\begin{theorem}[Energy-decaying property of the reference iteration] \label{t:TSA_energy}
Let $(u_1^0,u_2^0) \in \bar{\mathcal{K}}$ with $\mathcal{E}^{\tau_1,\tau_2} \left(u_1^0,u_2^0; \bar{\Omega} \right) < \infty$.
The iterations \eqref{iter:step1_schwarz} and \eqref{iter:step2_schwarz} satisfy the energy-decreasing property:
\[
\mathcal{E}^{\tau_1,\tau_2} \left( u_1^{n+1}, u_2^{n+1}; \bar{\Omega} \right)
\ \leq \
\mathcal{E}^{\tau_1,\tau_2} \left(u_1^{n},u_2^{n}; \bar{\Omega}\right),
\quad \forall n \ge 0.
\]
\end{theorem}

\begin{proof}
By the expansion of the energy functional $\mathcal{E}^{\tau_1,\tau_2} (u_1,u_2; \bar\Omega)$, which is similar to \eqref{equ:expansion_interface_energy}, we show that the energy is non-increasing at each iteration in both regions.
In the near-region, we have
\begin{align*}
    \mathcal{E}^{\tau_1,\tau_2} (u_1^{2n+1},u_2^{2n+1}; \bar\Omega)-\mathcal{E}^{\tau_1,\tau_2} (u^{2n}_1,u^{2n}_2; \bar\Omega)
    &= {\mathcal{L}^{\tau_1,\tau_2}}(u_1^{2n+1}-u^{2n}_1,u_2^{2n+1}-u^{2n}_2, u^{2n}_1,u^{2n}_2; \bar{\Omega}) +H(u_1^{2n+1}-u^{2n}_1,u_2^{2n+1}-u^{2n}_2; \bar\Omega) \\
    &\leq H(u_1^{2n+1}-u^{2n}_1,u_2^{2n+1}-u^{2n}_2;\bar\Omega)
    \, \\&\leq \,  0,
\end{align*}
where \(H\) is as in \eqref{equ:expansion_interface_energy}, with \(\Omega\) replaced by \(\bar{\Omega}\), and \[{\mathcal{L}^{\tau_1,\tau_2}}(u_1^{2n+1}-u^{2n}_1,u_2^{2n+1}-u^{2n}_2, u^{2n}_1,u^{2n}_2; \bar\Omega)\leq \mathcal{L}^{\tau_1,\tau_2}(u_1^{2n}-u^{2n}_1,u_2^{2n}-u^{2n}_2, u^{2n}_1,u^{2n}_2; \bar\Omega) = 0\] comes from \eqref{iter:step1_schwarz}.
The proof follows the same idea as Theorem~\ref{thm:stablity_wetMBO}, but with a Dirichlet boundary condition.

In the far-region, the energy is minimized directly in \eqref{iter:step2_schwarz}, hence
\begin{align*}
    &\mathcal{E}^{\tau_1,\tau_2} (u^{2n+2}_1,u^{2n+2}_2; \bar\Omega)-\mathcal{E}^{\tau_1,\tau_2} (u_1^{2n+1},u_2^{2n+1};\bar\Omega) \leq 0.
\end{align*}
\end{proof}

This completes the reformulation of the TSA method and energy analysis. We now proceed to the implementation and present the numerical details.

\section{Implementation of the TSA method} \label{s:Imple}

In this section, we describe the numerical implementation of approximating the CAH interval via the channel method, where the CMCF~\eqref{prob:channel} is simulated by the TSA method of Algorithm~\ref{alg:TSA}.
To approximate the directional CAH of the rough surface $z = \psi(x,y) = \eps\sin\frac{x}{\eps}\sin\frac{y}{\eps}$
along various rational directions, we identify the maximal and minimal stationary solutions, starting from the lower and upper angle barrier configurations~\eqref{equ:angle_barriers}, respectively.
We begin with an overview of the TSA method.

\subsection{Overview of the TSA method} \label{sec:TSA_overview}

\noindent \emph{Step 1: Initialization.}
Let $C_{\mathrm{near}}$ and $C_{\mathrm{far}}$ denote the near- and far-regions.
Initialize the liquid, vapor, and solid regions $L^0, V^0, S \subset C_{\mathrm{near}}$, together with an imaginary solid region $\mathcal{I} \subset C_{\mathrm{far}} \setminus C_{\mathrm{near}}$ and initial imaginary angle $\theta_i^0$.
Set the time steps $(\tau_1, \tau_2,\tau_3)$, where
$\tau_2 = \left(\frac{\pi \cos\theta_Y}{\pi - 2 \theta_Y}\right)^2 \tau_1$ and $\tau_3 = \left(\frac{\pi \cos\theta_i^0}{\pi - 2 \theta_i^0}\right)^2 \tau_1$,
and a tolerance $\mathrm{tol}>0$.
Assume \(\Sigma_{LV}\) is represented in the far-region by a graph \(x = x(y,z)\) over the domain $\{(y,z) : y \in [-l_y, l_y],\ z \in [r, R]\}$, as in Theorem~\ref{thm:aprox_channel_min_upper_interface}.

\medskip

\noindent \emph{Step 2: Near-region update (MBO).}
At each iteration, the near-region $C_{\mathrm{near}}$ (together with $\mathcal{I}$) is centered at the liquid–vapor interface.
Then, for iteration $k$, we compute \(\phi\) defined in \eqref{equ:phi_in_MBO} and update the liquid and vapor regions by thresholding~\eqref{iter:threshold}.
This step can be interpreted as minimizing the linearized energy \(\mathcal{L}\) in Step~\ref{iter:step1_TSA}.

We emphasize that the Dirichlet boundary condition in \eqref{iter:step1_schwarz} is used at the analytical level to define a well-posed reference problem.
In contrast, the numerical scheme employs a Neumann-like boundary condition as in \eqref{prob:near_channel}, which preserves the structure of the MBO diffusion step and is more suitable for the convolution-based implementation.

Then the imaginary angle \(\theta_i^{k+1}\), together with the time step \(\tau_3\), is updated according to the upper trace of interface $\Sigma_{LV} \cap\partial C_{\mathrm{far}}$ to maintain the imposed upper boundary condition.

\medskip
\noindent \emph{Step 3: Far-region update (Minimal Surface).}
The liquid–vapor interface is reconstructed by solving a linearized minimal surface problem over $C_{\mathrm{far}}$, from the trace of updated near-region interface as boundary data.
Specifically, we compute \(x(y,z)\) in \eqref{equ:far_region_sol}.
The constructed interface hence approximates the solution of \eqref{iter:step2_schwarz}, which can be interpreted as a minimal surface connecting the trace of the lower interface to the upper boundary portion $L \cap \{z=R\} = \{(x,y,R)\colon x \le 0\}$.

With a slight abuse of notation, we update $L^{k+1}$ and $V^{k+1}$ only in $C_{\mathrm{near}} \cap C_{\mathrm{far}}$ using the approximated minimal surface.

\medskip
\noindent \emph{Step 4: Two-level stopping criterion.} 

\emph{Inner iteration (Algorithm~\ref{alg:TSA}).}
For a fixed time-step pair $(\tau_1,\tau_2)$, we repeat Steps~2 and~3 until
\begin{equation*}
    \|1_{L^{k+1}}-1_{L^k}\|_{L^1(C_{\mathrm{near}})}
< \mathrm{tol}.
\end{equation*}

This criterion identifies an approximate stationary state of the TSA
iteration for the current time steps. After convergence, we halve the time
steps and continue the iteration using the computed stationary state as the
initial configuration. 

\emph{Outer iteration.}
The overall procedure is terminated when the stationary states obtained from
two successive time-step pairs differ by less than $\mathrm{tol}$ in the
$L^1(C_{\mathrm{near}})$ norm.

\begin{remark}
We use a two-level stopping criterion following the MBO scheme of \cite{Wang_2019}.
The criterion first checks convergence of the TSA iteration for a fixed time step, and then checks the effect of time-step refinement on the computed configuration.
A time step that is too large may introduce excessive smoothing during the diffusion step and average out features of the rough surface near the contact line.
We therefore first use a sufficiently large time step to evolve the interface to an approximate stationary state, and then successively halve the time step to refine the contact-line configuration.
The refinement is stopped when further halving produces no change in the computed configuration on the fixed spatial grid.
\end{remark}

\vspace{5mm}

\subsection{Implementation details for Step 1: initialization}

Consider the rough solid surface
\[
S:=\{(x,y,z)\colon z\leq \psi(x,y)=\eps\sin\frac{x'}{\eps}\sin\frac{y'}{\eps}\}
\]
where $\eps=0.1$ and $\begin{bmatrix}x'\\y'\end{bmatrix}
=\begin{bmatrix}
    \cos\arg\mathbf{k} & -\sin\arg\mathbf{k} \\ \sin\arg\mathbf{k}&\cos\arg\mathbf{k}
\end{bmatrix}\begin{bmatrix}
    x\\y
\end{bmatrix}$.  Young's angle is set to $\theta_Y = \tfrac{\pi}{3}$, and we choose directions $\mathbf{k}$ with $\arg \mathbf{k} = \arctan(\tfrac{i}{6})$, $i=-6,-5,\dots,6$, to preserve the $\mathbb{Z}^2$-periodicity of the surface in the $xy$-plane. See Figure~\ref{fig:rotation_Z2}.

We implement the associated channel method in a channel of height $R=3.3$.
The far-region is located above $z=r=0.5$,
and the near-region has height $\frac{6}{5}r=0.6$.

The computational domain \(\Omega_{\mathrm{near}} =C_{\mathrm{near}}\cup S\cup {\mathcal{I}}\) is chosen as (with reflection in $x$-direction)
\[
x \in [-2l_x,\, 2l_x],
\qquad
y \in [-l_y,\, l_y],
\qquad
z \in [-l_z,\, l_z],
\]
where
\[
l_x = 4\pi\eps,
\qquad
l_y = \pi\eps \sqrt{p^2 + q^2},
\qquad
l_z = 6\eps,
\]
and $\tfrac{p}{q} = \tan(\arg\mathbf{k})$ in lowest terms, so that the domain size in the $y$-direction matches the rotated $\mathbb{Z}^2$-periodicity.

Each direction is uniformly discretized using $N=512$ grid points (with an additional $N$ grid points in the reflection), yielding mesh sizes
\[
\Delta x = \frac{2l_x}{N},
\qquad
\Delta y = \frac{2l_y}{N},
\qquad
\Delta z = \frac{2l_z}{N}.
\]
The discrete grid is then used for the convolution step in the MBO iteration via FFT, where the initial time step \(\tau_1\) is chosen as \(10^{-3}\).

The solid region is placed below \(z=-\frac{3}{5}r+\psi(x,y)\), while the imaginary solid region \(\mathcal{I}\) is placed above \(z=\frac{3}{5}r\).
As a result, the near-region is located at the center of the computational domain.
The interface portion \(\Sigma_{LV}\cap C_{\mathrm{near}}\) is positioned at the center of the left half of the domain.
A mirrored copy of this configuration is then constructed in the right half, allowing the use of the FFT without introducing boundary artifacts in the \(x\)-direction.
After each update, the computational domain is shifted so that the interface remains centered.

Starting from the lower and upper angle barrier configurations,
we compute the stationary states of the Schwarz alternating dynamics.
The resulting apparent contact angles are recorded as approximations of the receding and advancing angles, see Section~\ref{sec:two-level_stopping}.

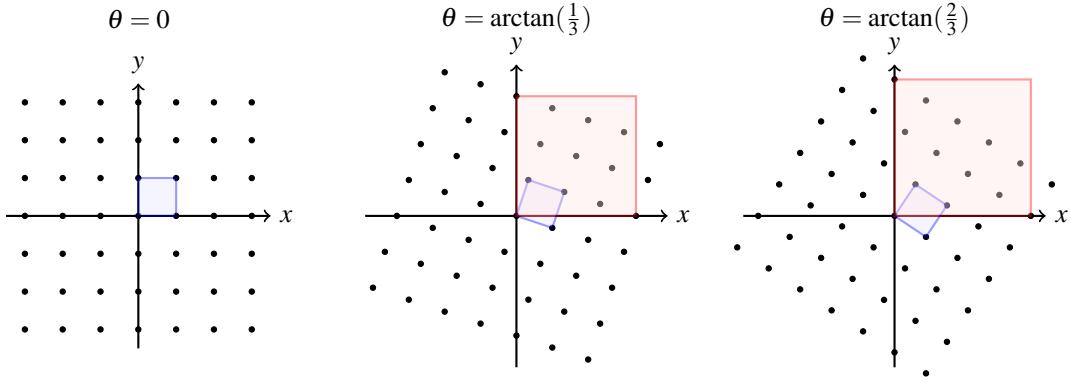
\begin{figure}[t]
\centering

\begin{tikzpicture}[scale=0.5]
\def\radius{0.07}
\def\maxgrid{3}

% FIRST: Z^2
\begin{scope}
  \node at (0,5.2) {$\theta = 0$};
  % Axes
  \draw[->, thick] (-\maxgrid-0.5,0) -- (\maxgrid+0.5,0) node[anchor=west] {$x$};
  \draw[->, thick] (0,-\maxgrid-0.5) -- (0,\maxgrid+0.5) node[anchor=south] {$y$};
  % Lattice
  \foreach \i in {-3,...,3} {
    \foreach \j in {-3,...,3} {
      \filldraw[black] (\i,\j) circle (\radius);
    }
  }
  \draw[blue, thick, fill=blue!10, opacity=0.4] (0,0) -- (1,0) -- (1,1) -- (0,1) -- cycle;
\end{scope}

% SECOND: rotation by arctan(1/3)
\begin{scope}[xshift=10cm]
  \node at (0,5.2) {$\theta = \arctan(\tfrac{1}{3})$};
  \pgfmathsetmacro{\theta}{atan(1/3)}
  \pgfmathsetmacro{\ct}{cos(\theta)}
  \pgfmathsetmacro{\st}{sin(\theta)}
  % Axes
  \draw[->, thick] (-4,0) -- (4,0) node[anchor=west] {$x$};
  \draw[->, thick] (0,-4) -- (0,4) node[anchor=south] {$y$};
  % Lattice
  \foreach \i in {-3,...,3} {
    \foreach \j in {-3,...,3} {
      \pgfmathsetmacro{\x}{\ct*\i + \st*\j}
      \pgfmathsetmacro{\y}{-\st*\i + \ct*\j}
      \filldraw[black] (\x,\y) circle (\radius);
      % Note that now [i,j] = M_\theta [x,y], circles are drawn only if [i,j] a interger pair
    }
  }
  % Define rotated basis vectors
  \pgfmathsetmacro{\vxone}{\ct}
  \pgfmathsetmacro{\vyone}{-\st}
  \pgfmathsetmacro{\vxtwo}{+\st}
  \pgfmathsetmacro{\vytwo}{\ct}

  % Fundamental cell
  \draw[blue, thick, fill=blue!10, opacity=0.4]
  (0,0) -- (\vxone,\vyone)
  -- ({\vxone+\vxtwo},{\vyone+\vytwo})
  -- (\vxtwo,\vytwo) -- cycle;
  \draw[red, thick, fill=red!10, opacity=0.4] (0,0) -- (10^0.5,0) -- (10^0.5,10^0.5) -- (0,10^0.5) -- cycle;
\end{scope}

% THIRD: rotation by arctan(2/3)
\begin{scope}[xshift=20cm]
  \node at (0,5.2) {$\theta = \arctan(\tfrac{2}{3})$};
  \pgfmathsetmacro{\theta}{atan(2/3)}
  \pgfmathsetmacro{\ct}{cos(\theta)}
  \pgfmathsetmacro{\st}{sin(\theta)}
  % Axes
  \draw[->, thick] (-4,0) -- (4,0) node[anchor=west] {$x$};
  \draw[->, thick] (0,-4) -- (0,4) node[anchor=south] {$y$};
  % Lattice
  \foreach \i in {-3,...,3} {
    \foreach \j in {-3,...,3} {
      \pgfmathsetmacro{\x}{\ct*\i + \st*\j}
      \pgfmathsetmacro{\y}{-\st*\i + \ct*\j}
      \filldraw[black] (\x,\y) circle (\radius);
    }
  }
  % Define rotated basis vectors
  \pgfmathsetmacro{\vxone}{\ct}
  \pgfmathsetmacro{\vyone}{-\st}
  \pgfmathsetmacro{\vxtwo}{\st}
  \pgfmathsetmacro{\vytwo}{\ct}

  % Fundamental cell
  \draw[blue, thick, fill=blue!10, opacity=0.4]
  (0,0) -- (\vxone,\vyone)
  -- ({\vxone+\vxtwo},{\vyone+\vytwo})
  -- (\vxtwo,\vytwo) -- cycle;
  \draw[red, thick, fill=red!10, opacity=0.4] (0,0) -- (13^0.5,0) -- (13^0.5,13^0.5) -- (0,13^0.5) -- cycle;
\end{scope}

\end{tikzpicture}
\vspace{1em}
\caption{Visualization of the $\mathbb{Z}^2$ lattice under rotation.
Left: standard lattice. Middle/right: lattices rotated by $\theta = \arctan(\tfrac{1}{3})$ and $\theta = \arctan(\tfrac{2}{3})$, respectively.
Highlighted fundamental cells (blue for the original lattice, red for the rotated lattices) illustrate the preservation of the $\mathbb{Z}^2$-periodicity.}
\label{fig:rotation_Z2}
\end{figure}

\subsection{Implementation details for Step 2: near-region update (MBO)} \label{sec:Imple_step2}
In the near-region \(C_{\mathrm{near}}\), the MBO scheme is implemented using MATLAB's built-in \texttt{fft}, which requires a periodic computational domain.
Periodicity in the $y$-direction is ensured by the $\mathbb{Z}^2$-periodicity of the surface, while in the $x$-direction it is imposed via reflection.
In the $z$-direction, the solid regions (including the imaginary solid) prevent boundary errors from affecting the liquid–vapor interface $\Sigma_{LV}$, which evolves only within $C_{\mathrm{near}}$.
For the convolution kernel, let $\tau_1$ denote the base time step.
The time steps \(\tau_2\) and \(\tau_3\) are given as in \eqref{equ:tau23}, which enforce the prescribed contact angles $\theta_Y$ and $\theta_i$ in the threshold dynamics.

The interface is initialized as a flat surface $x^0(y,z)$ with slope $\frac{dx}{dz} =\pm 3$ in the $xz$-plane, corresponding to the lower and upper angle barrier configurations~\eqref{equ:angle_barriers} in the channel method, as shown in Figure~\ref{fig:zoomin_channel}.
In practice, this is done by defining the initial liquid region as
\begin{equation} \label{equ:ICs}
    L^0 := \{(x,y,z) \in C_{\mathrm{near}}: x  \mp 3 (z-R)<0\}.
\end{equation}
The imaginary angle $\theta_i$ is updated after each iteration to approximate the upper boundary condition of the near-region.
It is computed as
\[
\theta_i = \frac{\pi}{2} - \arctan\left( m\right),
\]
where \(m=\frac{1}{(r-R)} \dashint_{-l_y}^{l_y} f(y)\,dy\) and $x=f(y)$ is the trace of the liquid–vapor interface on the lower boundary of $C_{\mathrm{far}}$,
and $r$ is the height of the lower boundary of $C_{\mathrm{far}}$.

Note that $L^k$, $V^k$, and $S$ are now defined only within the near-region $C_{\mathrm{near}}$.

\subsection{Implementation details for Step 3: far-region update (minimal surface)}
In the far-region, the interface is represented as a graph
\[
x = x(y,z), \quad (y,z) \in [-l_y,l_y] \times [r,R].
\]
The explicit formula for $x(y,z)$ is given in Theorem~\ref{thm:aprox_channel_min_upper_interface} where $A_n$ and $B_n$ are the Fourier coefficients of
\(
g(y)=f(\sqrt{1+m^2}\,y-l_y)-m(r-R).
\)

The truncated series can also be obtained efficiently via FFT.
Specifically, we evaluate \(g(y)\) on a uniform grid over $[-l_y,l_y]$ and compute its discrete Fourier transform.
The cosine and sine coefficients are extracted from the real and imaginary parts of the Fourier modes, and the corresponding factors are applied to  each Fourier mode to reconstruct the truncated series for $x(y,z)$. The truncation parameter \(N_{\mathrm{F}}\) is chosen to be the same as the discretization parameter \(N\). 

Note that this truncated series is then used to update only the portion of \(L^k\) and \(V^k\) over the transition region \(C_{\mathrm{near}}\cap C_{\mathrm{far}}\).

\subsection{Implementation details for Step 4: two-level stopping criterion} \label{sec:two-level_stopping}

For the inner iteration, we use the discrete $\ell^1$-norm of the discrete indicator function $1_{L^k}$ as the stopping criterion, up to a constant grid-volume factor.
The iteration is terminated when
\[
\|1_{L^{k-1}}-1_{L^{k}}\|_1 < \mathrm{tol},
\]
where $\|\cdot\|_1$ denotes the sum of the absolute values of the discrete function $1_{L^{k-1}} - 1_{L^{k}}$ over all grid points.
The tolerance $\mathrm{tol}$ is chosen as \(10^{-7} N^3\) so that it is proportional to the number of grid points.
After convergence, the time step $\tau_1$ is halved, with $\tau_2$ and $\tau_3$ then updated according to \eqref{equ:tau23}, and the iteration is continued.

For the outer iteration, we initialize $L^\star = L^0$.
After each inner-iteration convergence, the resulting stationary state $L^k$ is compared with the recorded state $L^\star$.
If
\[
\|1_{L^{\star}}-1_{L^{k}}\|_1 < \mathrm{tol},
\]
the overall procedure is terminated; otherwise, we set $L^\star = L^k$ and continue with halved time steps.

The apparent contact angles of the stationary states obtained from the lower and upper angle barrier configurations~\eqref{equ:ICs} are taken as approximations of the receding (\(\thetarec\)) and advancing (\(\thetaadv\)) angles.
Since \(\theta_i\) is measured on the opposite side of the interface, the apparent contact angle is recorded as \(\pi-\theta_i\).

\subsection{Implementation details for computing the CAH interval}

For each direction \(\mathbf{k}\) with \(\arg \mathbf{k}=\arctan(\tfrac{i}{6})\), \(i=-6,-5, \ldots ,6\), the receding and advancing angles are approximated by the apparent angles of the stationary states obtained from the initial flat surfaces~\eqref{equ:ICs}.

\section{Numerical experiments} \label{s:NumExp}

In this section, we present numerical experiments for computing the directional CAH of the rough surface \(z=\psi(x,y)=\eps\sin\frac{x}{\eps}\sin\frac{y}{\eps}\).
Hysteresis intervals are approximated for various rational contact-line directions, allowing us to characterize the directional dependence of the CAH. We then examine the corresponding macroscopic slope and visualize representative liquid–vapor interfaces and contact lines.

\subsection{Directional CAH over rational directions}

For selected contact-line normals \(\mathbf{k}\), the hysteresis intervals were approximated using the TSA method.
Figure~\ref{fig:CAH_wrt_k} shows the CAH interval as a function of the contact-line orientation \(\arg\mathbf{k}\), together with the computed endpoint values.
The computed data include only 13 sampled rational directions $\arg\mathbf{k}=\arctan(i/6)$, $i=-6,-5,\ldots,6$; the full angular range is obtained by extending the result according to the symmetry of the surface.

\begin{figure}[t]
\centering
\begin{minipage}[t]{0.37\textwidth}%
\vspace{0pt}
\centering
\includegraphics[width=\linewidth, trim={85 135 15 80}, clip]{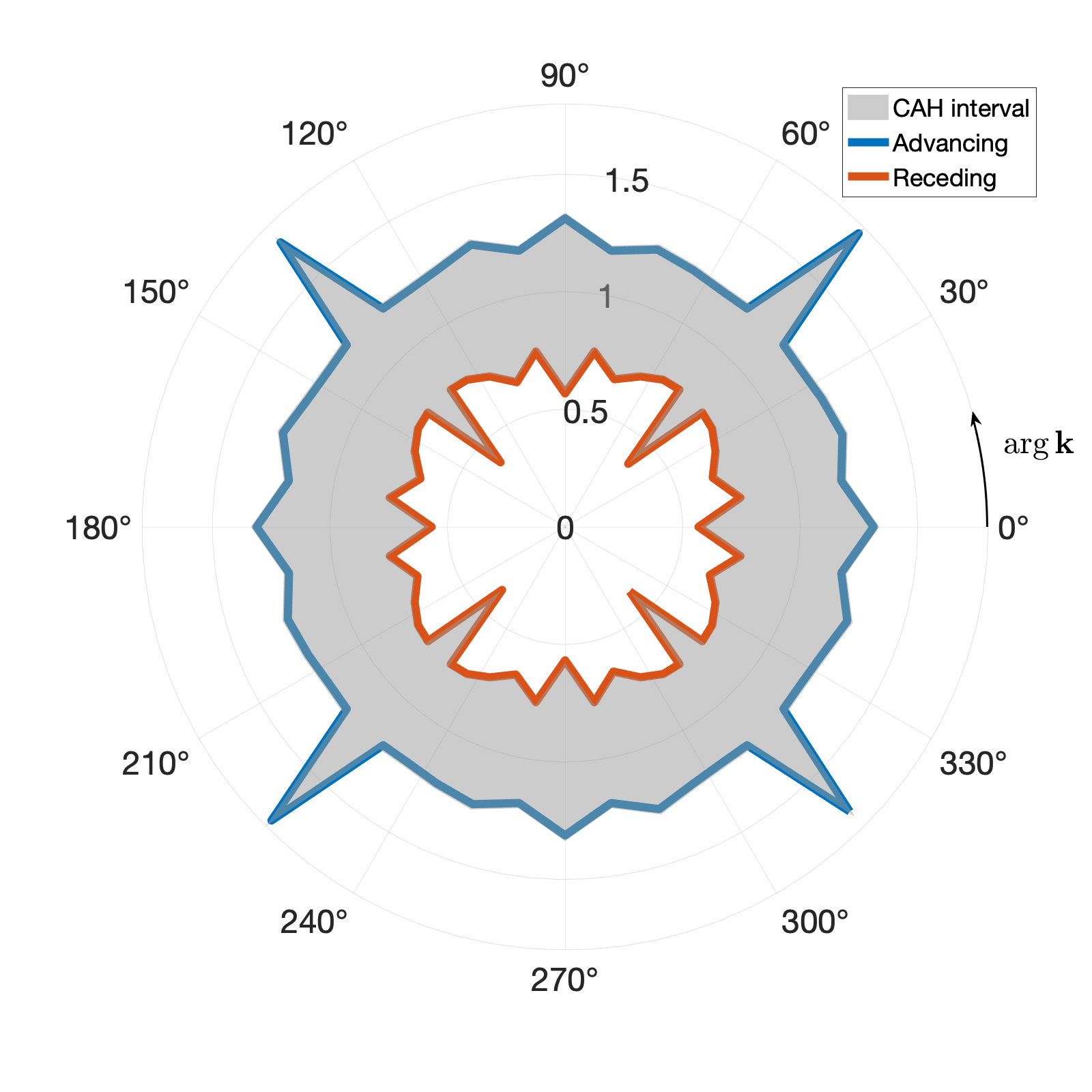}
\end{minipage}%
\hspace{1.5em}%
\begin{minipage}[t]{0.38\textwidth}%
\vspace{0pt}
\centering
\footnotesize
   \begin{tabular}{lccc}
        \toprule
        $\arg\mathbf{k}$ (rad)
        & $\thetarec$
        & $\thetaadv$
        & $\thetaadv-\thetarec$ \\
        \midrule
        $-\tfrac{\pi}{4}$
            & 0.3881 & 1.7121 & 1.3239 \\
        $-\arctan(\tfrac{5}{6})$
            & 0.7598 & 1.2089 & 0.4490 \\
        $-\arctan(\tfrac{2}{3})$
            & 0.7523 & 1.2051 & 0.4528 \\
        $-\arctan(\tfrac{1}{2})$
            & 0.7156 & 1.2202 & 0.5046 \\
        $-\arctan(\tfrac{1}{3})$
            & 0.6489 & 1.2653 & 0.6165 \\
        $-\arctan(\tfrac{1}{6})$
            & 0.7564 & 1.1908 & 0.4344 \\
        $0$
            & 0.5674 & 1.3122 & 0.7448 \\
        $\arctan(\tfrac{1}{6})$
            & 0.7564 & 1.1908 & 0.4344 \\
        $\arctan(\tfrac{1}{3})$
            & 0.6619 & 1.2431 & 0.5813 \\
        $\arctan(\tfrac{1}{2})$
            & 0.7156 & 1.2202 & 0.5046 \\
        $\arctan(\tfrac{2}{3})$
            & 0.7523 & 1.2051 & 0.4528 \\
        $\arctan(\tfrac{5}{6})$
            & 0.7598 & 1.2089 & 0.4490 \\
        $\tfrac{\pi}{4}$
            & 0.3793 & 1.7666 & 1.3873 \\
        \bottomrule
    \end{tabular}
\end{minipage}
\caption{Computed CAH interval for \(z= \eps \sin\frac{x}{\eps}\sin\frac{y}{\eps} \) with $\theta_Y=\tfrac{\pi}{3}$, for the 13 sampled rational directions $\arg\mathbf{k}=\arctan(\tfrac{i}{6})$, $i=-6,-5,\ldots,6$.
(\textbf{left}) The inner and outer curves denote the receding and advancing angles, respectively, and the shaded region denotes the computed interval; the radial coordinate is measured in radians.
(\textbf{right}) The corresponding endpoint values. The interval width $\thetaadv-\thetarec$ is computed from the unrounded endpoints. All angles are in radians.}
\label{fig:CAH_wrt_k}
\end{figure}

For every sampled rational orientation, the computed receding and advancing angles are distinct, yielding a nontrivial CAH interval.
Among the sampled directions, the interval is widest at the diagonal orientation \(\arg\mathbf{k}=\tfrac{\pi}{4}\).
Both endpoints change sharply between \(\arg\mathbf{k}=\arctan(\tfrac{5}{6})\) and \(\tfrac{\pi}{4}\), whereas their directional variation is smoother away from the diagonal.
This behavior is consistent with the symmetry of the rough surface and with the stronger pinning expected along its principal periodic directions.
In the following, we refer to this sharp directional change as the CAH discontinuity for convenience.
We emphasize that this terminology does not imply that a discontinuity has been established numerically, since the computations use a finite set of rational orientations.
In related exactly solvable models \cite{feldman2021limit,feldman2019free} the hysteresis interval does rigorously have discontinuous dependence on the normal direction at certain rational directions, so we do believe true discontinuity is likely.

The rough surface satisfies
$ \psi(x+\pi\eps,-y)=\psi(x,y), $
and hence the CAH interval is symmetric under
\(\arg\mathbf{k}\mapsto-\arg\mathbf{k}\). The computed endpoints reproduce this symmetry to four decimal places at \(\arg\mathbf{k}=\pm\arctan(\frac16)\), \(\pm\arctan(\frac 1 2)\), \(\pm\arctan(\frac23)\), and \(\pm\arctan (\frac 56)\). At \(\pm\arctan(\frac13)\), the receding and advancing endpoints differ by \(0.0130\) and \(0.0222\) radians, respectively, while at \(\pm\frac \pi 4\) they differ by \(0.0088\) and \(0.0545\) radians.

The pattern is not accidental: whether the finite-channel computation respects this symmetry at a given orientation is determined by the parity of the two integers that specify the direction, as we now show.
Let \(\alpha=\arg\mathbf{k}=\arctan(p/q)\), where \(p\) and \(q\) are relatively prime, and set \(D=p^2+q^2\). After rotating the propagation direction to the computational \(x\)-axis, the rough surface is
\[
\psi_\alpha(x,y)=\eps\sin\left(\frac{qx-py}{\eps\sqrt D}\right)\sin\left(\frac{px+qy}{\eps\sqrt D}\right).
\]
Along the phase-fixed line \(x=0\),
\[
\psi_\alpha(0,y)=-\eps\sin\left(\frac{py}{\eps\sqrt D}\right)\sin\left(\frac{qy}{\eps\sqrt D}\right),
\qquad
\psi_{-\alpha}(0,y)=\eps\sin\left(\frac{py}{\eps\sqrt D}\right)\sin\left(\frac{qy}{\eps\sqrt D}\right).
\]
Since \(l_y=\pi\eps \sqrt{p^2 + q^2}=\pi\eps\sqrt D\), the rotated surfaces satisfy
\[
\psi_{-\alpha}(x,-y+l_y)=(-1)^{p+q+1}\psi_\alpha(x,y).
\]
When \(p+q\) is odd, the transformation \((x,y)\mapsto(x,-y+l_y)\) leaves \(x=0\) fixed and maps the \(+\alpha\) finite-channel configuration to the \(-\alpha\) configuration without changing its phase relative to the upper boundary. This applies to \(\frac pq=\frac 16,\frac 12,\frac 23,\) and \(\frac 56\). For \(\frac pq=\frac 13\) and \(\frac pq=1\), both \(p\) and \(q\) are odd, so \(p+q\) is even and the same transformation changes \(\psi_\alpha\) to \(-\psi_\alpha\). Relating the two rotated roughness patterns then requires a nonzero translation in \(x\), which changes the phase fixed by the upper boundary \(x=0\). Thus, the finite-channel computations at \(\pm\arctan(\frac 13)\) and \(\pm\frac \pi4\) need not produce the same values, even though the CAH interval of the rough surface is symmetric.

Since the CAH interval of the rough surface is symmetric, the difference between the values computed at \(+\alpha\) and \(-\alpha\) measures the sensitivity of the finite-channel computation to the phase fixed by the upper boundary.
At \(\pm\arctan(\tfrac13)\) and \(\pm\tfrac{\pi}{4}\) this gives an error estimate of roughly \(0.01\)--\(0.05\) radians for the tabulated endpoints.
At the remaining sampled directions the transformation above forces the two computations to agree, so their agreement is automatic and provides no comparable estimate.

\subsection{Large-scale monotonicity of the slope \texorpdfstring{\(m\)}{m}}
\label{sec:numeric_slope}

The monotonicity of the macroscopic interface slope \(m\) is essential for the convergence of the CAH interval.
The energy interpretation discussed in Section~\ref{sec:energy_to_slope} provides a way to understand this monotonicity through the leading-order relation between the energy and the slope.
We now examine this behavior numerically.

At each iteration, we record the macroscopic slope \(m\) computed as described in Section~\ref{sec:Imple_step2}.
Figure~\ref{fig:slope_evolution} shows the evolution of \(m\) for two representative contact-line orientations, \(\arg\mathbf{k}=\arctan(\tfrac{5}{6})\) and \(\arg\mathbf{k}=\tfrac{\pi}{4}\).
In the large-scale plots, the advancing and receding processes both follow the monotone behavior predicted by the energy-slope relation.
The corresponding tail plots show small variations near the final stages of the evolution, while the overall large-scale trend remains unchanged.

\begin{figure}[t]
    \centering

    \begin{subfigure}{0.49\textwidth}
        \centering
        \includegraphics[width=.8\textwidth]{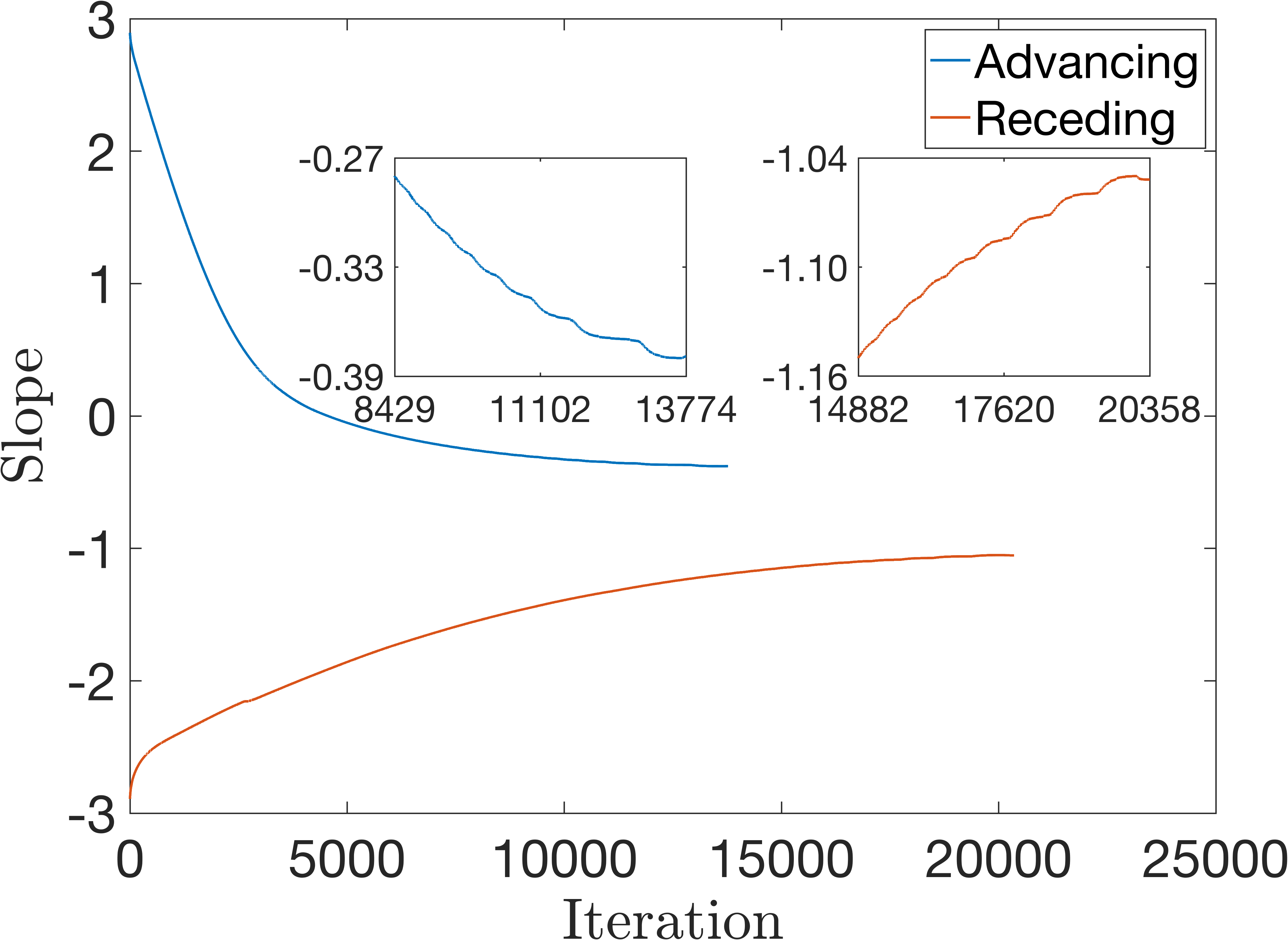}
        \caption{$\arg\mathbf{k}=\arctan(\tfrac{5}{6})$.}
    \end{subfigure}
    \hfill
    \begin{subfigure}{0.49\textwidth}
        \centering
        \includegraphics[width=.8\textwidth]{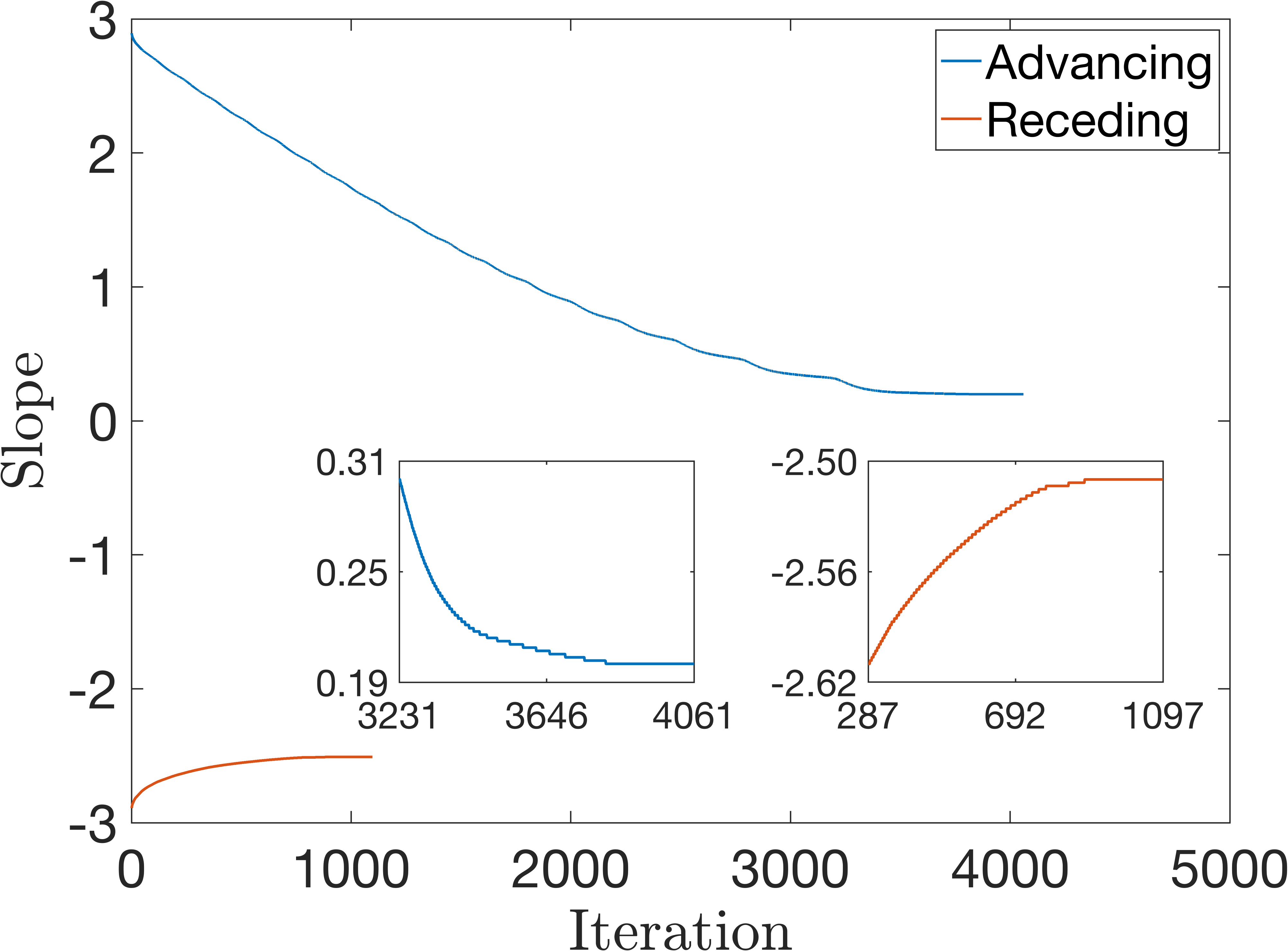}
        \caption{$\arg\mathbf{k}=\tfrac{\pi}{4}$.}
    \end{subfigure}

    \caption{
        Evolution of the macroscopic interface slope \(m\) for two representative contact-line orientations.
        In each panel, the large-scale plot shows the evolution of \(m\) for the advancing and receding processes, while the smaller panels show the corresponding tails near the final values.
        The horizontal axis denotes algorithmic iteration count; the time-step parameters are reduced according to the two-level stopping procedure.
    }
    \label{fig:slope_evolution}
\end{figure}

\subsection{Interface and contact-line structures near the CAH discontinuity} \label{sec:interface_structure}

We next examine the corresponding liquid–vapor interfaces and contact lines to understand the geometric structures associated with the directional CAH. The directions \(\arg\mathbf{k}=\arctan(\tfrac{5}{6})\) and \(\arg\mathbf{k}=\tfrac{\pi}{4}\) are selected to compare the interface configurations near the CAH discontinuity.

\begin{figure} [t]
    \centering
    \begin{subfigure}{0.4\textwidth}
        \centering
        \includegraphics[width=\linewidth]{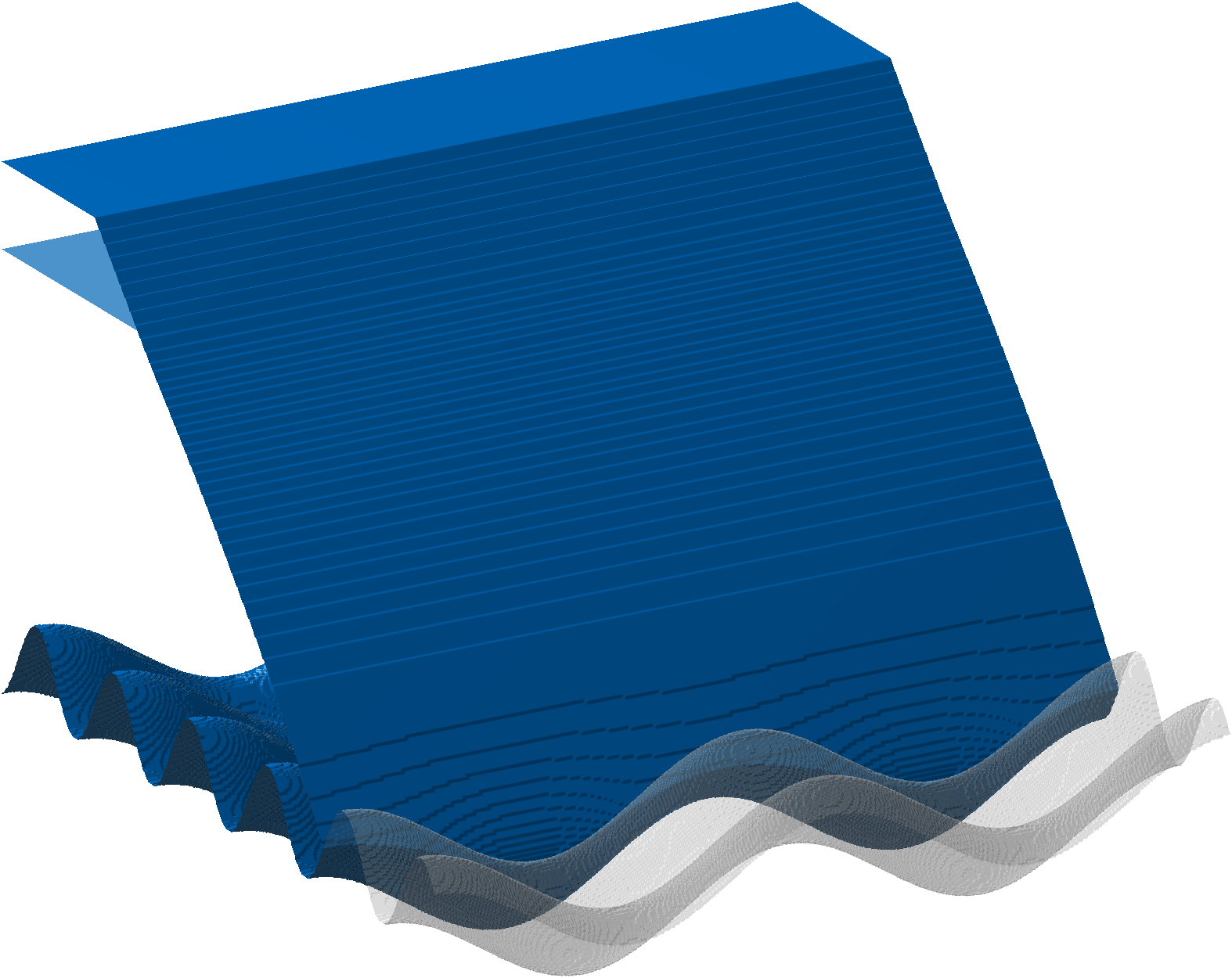}
        \caption{Receding angle, \(\arg \mathbf{k}=\tfrac{\pi}{4}\)}
    \end{subfigure}
        \hspace{1cm}
    \begin{subfigure}{0.4\textwidth}
        \centering
        \includegraphics[width=\linewidth]{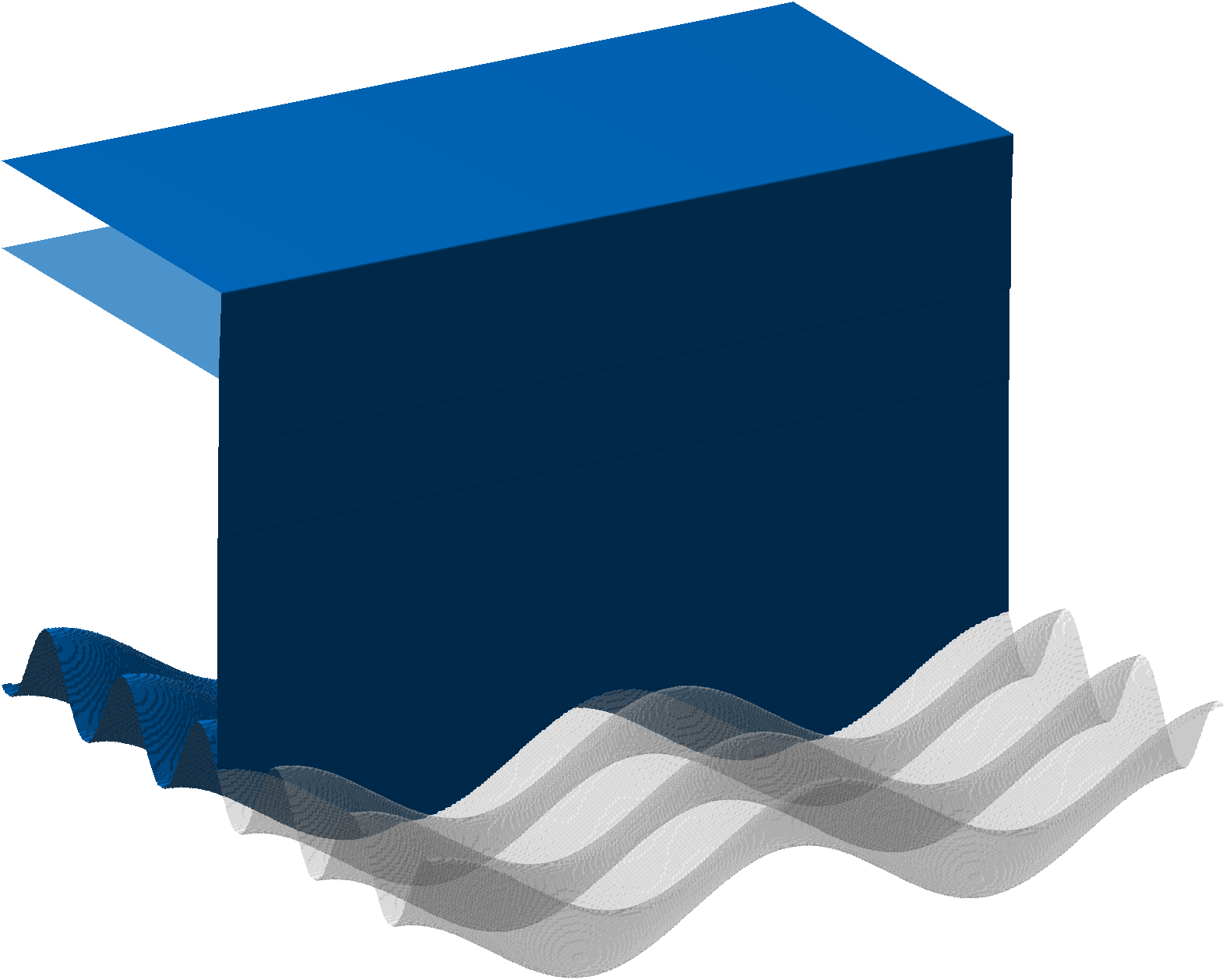}
        \caption{Advancing angle, \(\arg \mathbf{k}=\tfrac{\pi}{4}\)}
    \end{subfigure}\\
    \vspace{10mm}
\begin{subfigure}{0.4\textwidth}
        \centering
        \includegraphics[width=\linewidth, trim={170 230 120 200}, clip]{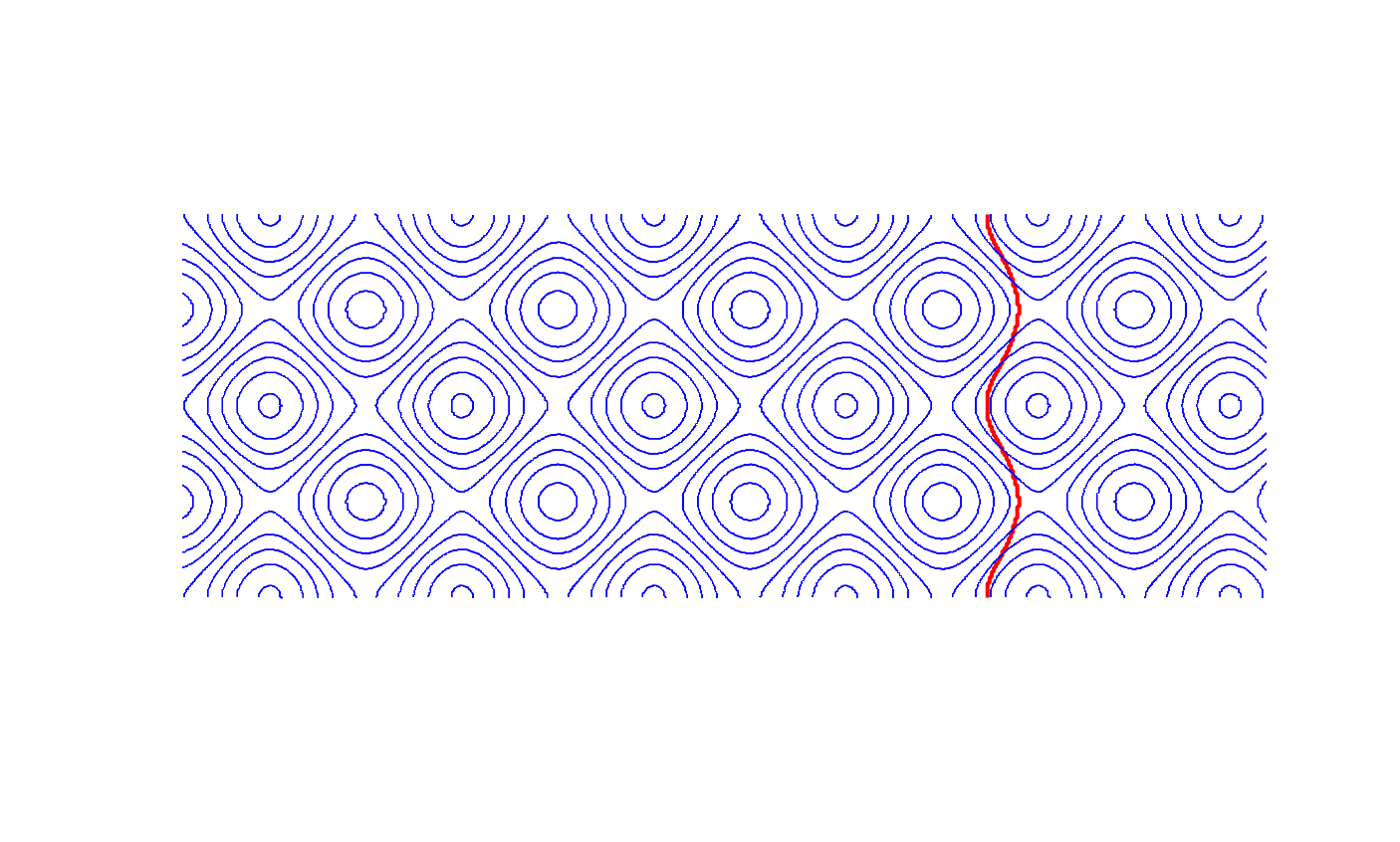}
        \caption{Receding contact line, \(\arg \mathbf{k}=\tfrac{\pi}{4}\)}
    \end{subfigure}
        \hspace{1cm}
    \begin{subfigure}{0.4\textwidth}
        \centering
        \includegraphics[width=\linewidth, trim={170 230 120 200}, clip]{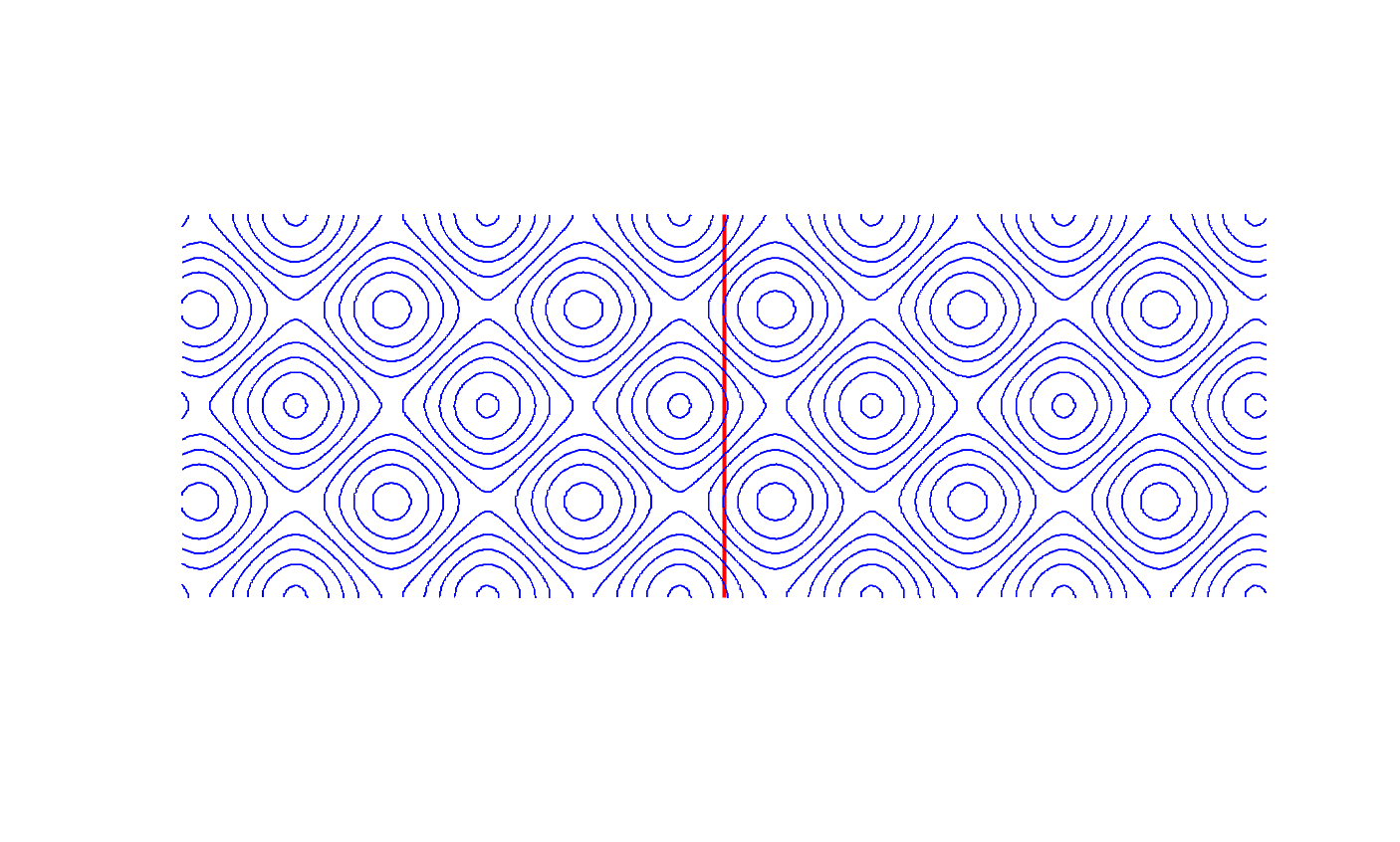}
        \caption{Advancing contact line, \(\arg \mathbf{k}=\tfrac{\pi}{4}\)}
    \end{subfigure}
    \caption{Receding and advancing liquid–vapor interfaces and contact lines for  \(\arg \mathbf{k}=\frac{\pi}{4}\).} \label{fig:interface_pi4}
\end{figure}

The side views of the 3D liquid–vapor interfaces are displayed over both the near- and far-regions; a translucent surface is used in the far-region for clarity.
For all tested directions, the liquid–vapor interfaces have a planar-like behavior as \(z\) increases, in both advancing and receding cases; see Figures~\ref{fig:3D_apparent}, \ref{fig:interface_pi4} and~\ref{fig:interface_5_6}.
This supports the graphical representation \(x=x(y,z)\), with deviations from a planar profile of order \(O(\eta)\), as discussed in Section~\ref{sec:condition_over_far}.

\begin{figure}[t!]
    \centering
    \begin{subfigure}{0.4\textwidth}
        \centering
        \includegraphics[width=\linewidth]{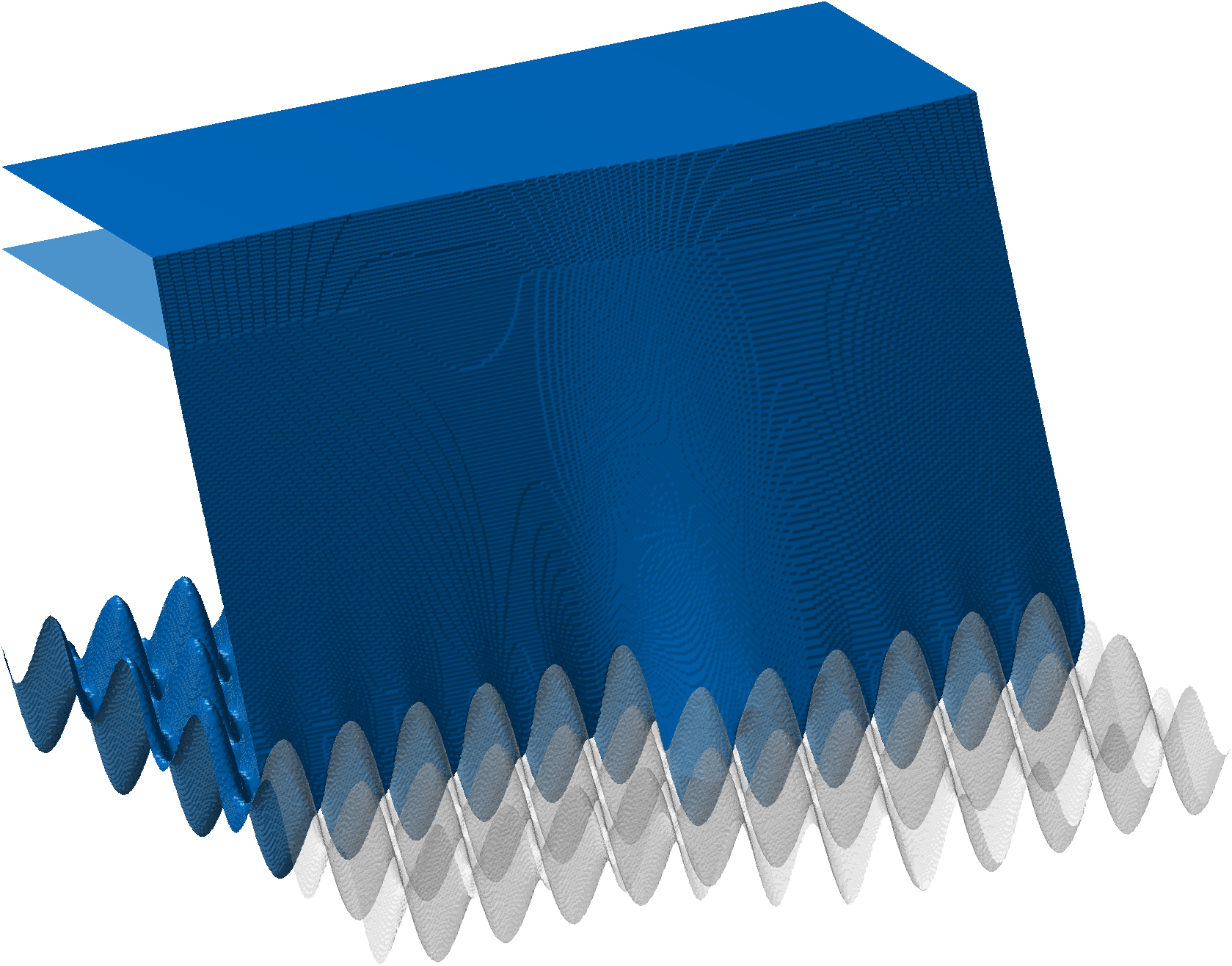}
        \caption{Receding angle, \(\arg \mathbf{k}=\arctan(\tfrac{5}{6})\)}
    \end{subfigure}
    \hspace{1cm}
    \begin{subfigure}{0.4\textwidth}
        \centering
        \includegraphics[width=\linewidth]{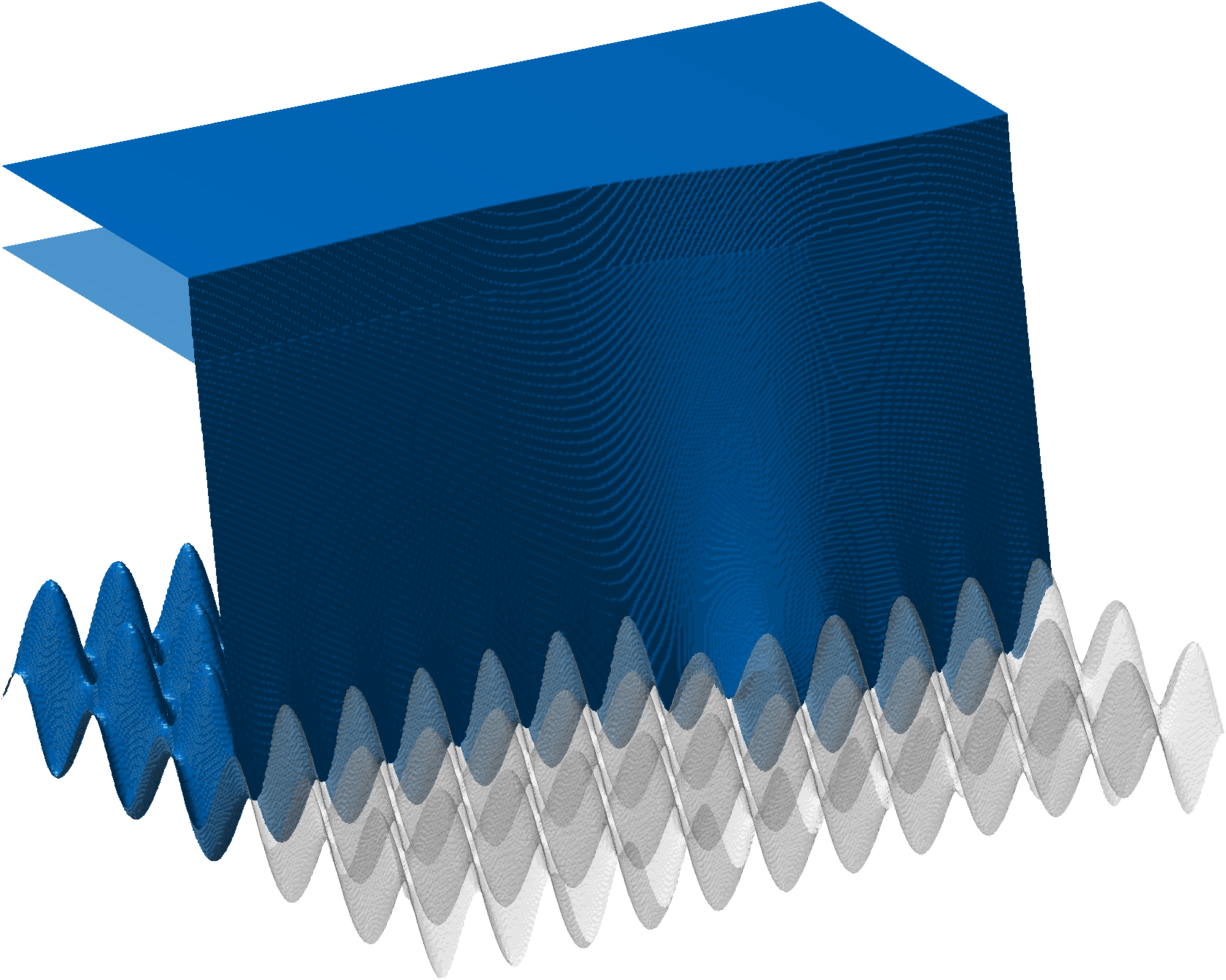}
        \caption{Advancing angle, \(\arg \mathbf{k}=\arctan(\tfrac{5}{6})\)}
    \end{subfigure}
    \caption{Receding and advancing liquid–vapor interfaces for  \(\arg \mathbf{k}=\arctan(\tfrac{5}{6})\).} \label{fig:interface_5_6}
\end{figure}

\begin{figure}[t!]
    \centering
    \begin{subfigure}{0.3\textwidth}
        \centering
        \includegraphics[width=\linewidth]{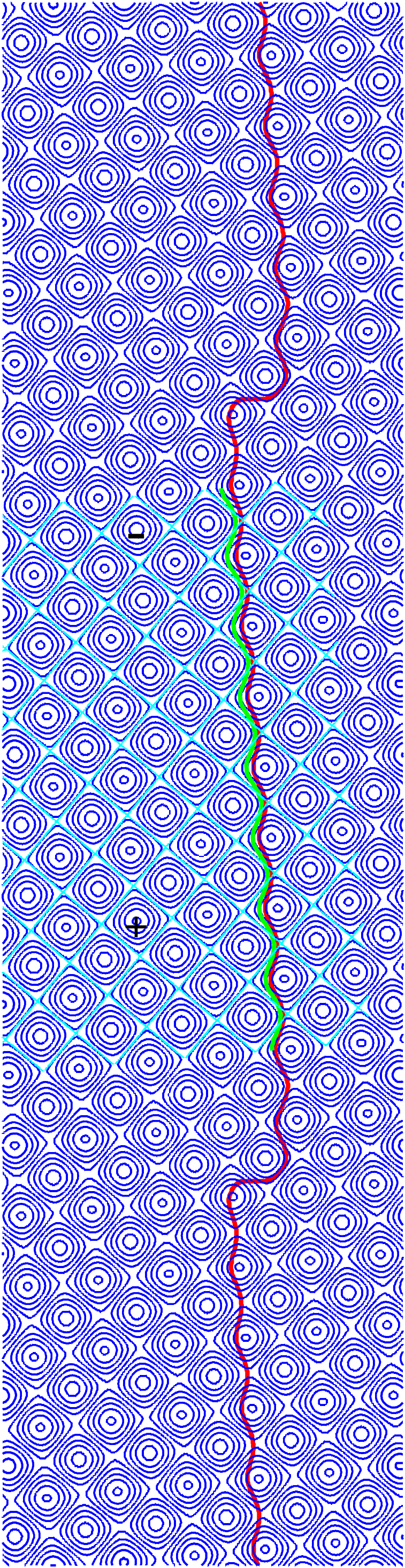}
        \caption{Receding contact lines.}
    \end{subfigure}
    \hspace{1cm}
    \begin{subfigure}{0.3\textwidth}
        \centering
        \includegraphics[width=\linewidth]{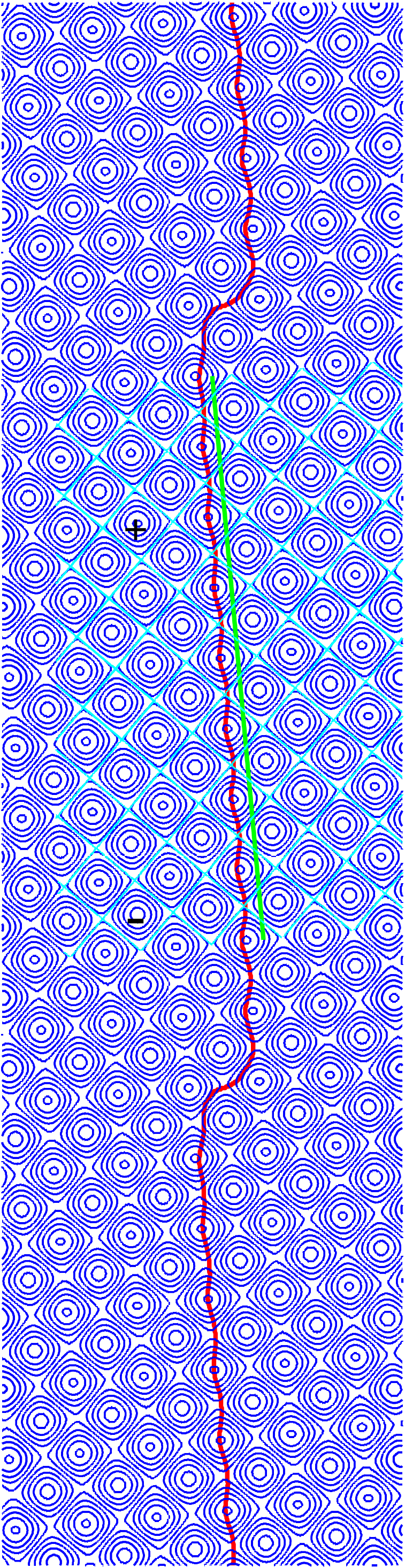}
        \caption{Advancing contact lines}
    \end{subfigure}
    \caption{A comparison of the receding and advancing contact lines for
    \(\arg \mathbf{k}=\arctan(\tfrac{5}{6})\)  (red) and \(\arg \mathbf{k}=\tfrac{\pi}{4}\) (green).} \label{fig:comparison_contact_line}
\end{figure}

The corresponding contact lines are also plotted together with the contour plot of the rough surface.
The local maxima and minima of the surface are marked by plus and minus signs, respectively.
To study the contact-line structure near the CAH discontinuity, we compare two nearby directions, \(\arg\mathbf{k}=\arctan(\tfrac{5}{6})\) and \(\arg\mathbf{k}=\tfrac{\pi}{4} \), in Figure~\ref{fig:comparison_contact_line}.
The contours of the solid surface are also plotted so that both contact lines are shown against the same surface.

Large portions of the two stationary contact lines have nearly the same pattern, while their differences are concentrated in several localized regions.
In particular, the contact line for \(\arg\mathbf{k}=\arctan(\tfrac{5}{6})\) follows the pattern for \(\arg\mathbf{k}=\tfrac{\pi}{4}\) over a large portion before switching to another segment of the same pattern.
This geometry is consistent with a possible stepwise transition, in which most of the contact line remains pinned while a localized segment shifts to a neighboring pattern.
Repeated localized shifts could produce a macroscopic advance resembling the motion of a typewriter print head; see Figure~\ref{fig:typewritter}.
We use this analogy as a geometric interpretation of the computed stationary configurations, rather than as a model of the physical transient dynamics.

\begin{figure}[t]
    \centering
    \begin{subfigure}{0.45\textwidth}
        \begin{tikzpicture}[x={(1cm,0cm)}, y={(0cm,0.1cm)}, z={(0cm,1cm)}, scale=2]

% Cylinder parameters
\def\radius{1}
\def\height{2}

% Draw cylinder surface
\draw[fill=gray!20] (-\radius,0,0) -- (\radius,0,0) -- (\radius,0,\height) -- (-\radius,0,\height) -- cycle;

\draw[fill=gray!10] (0,0,\height) ellipse (\radius cm and 0.1*\radius cm);

\draw[gray, thick, fill=gray!20] (0,0,0) ++(180:\radius cm and 0.1*\radius cm) arc (180:360:\radius cm and 0.1*\radius cm);

\draw[dashed, gray, thick, fill=gray!20] (0,0,0) ++(0:\radius cm and 0.1*\radius cm) arc (0:180:\radius cm and 0.1*\radius cm);

% Spiral around cylinder
\draw[red, thick, samples=200, domain=0:6.283*6, variable=\t, smooth]
    plot ({\radius*cos((\t+3.1415*3/2) r)}, {\radius*sin((\t+3.1415*3/2) r)}, {(\height/(6.283*6))*\t});

\draw[blue, thick] (0,-\radius,2/6*\height) -- (0,-\radius,3/6*\height);

\draw[blue, thick, samples=200, domain=6.283*2:6.283*3, variable=\t, smooth]
    plot ({\radius*cos((\t+3.1415*3/2) r)}, {\radius*sin((\t+3.1415*3/2) r)}, {(\height/(6.283*6))*\t});

\end{tikzpicture}

    \end{subfigure}
    \begin{subfigure}{0.45\textwidth}
        \begin{tikzpicture}[scale=2]

% Cylinder parameters
\def\radius{1}
\def\height{2}
\def\nturns{6}

% Width of rectangle
\def\width{pi*\radius}

% Draw cylinder side (flattened)
\draw[fill=gray!20] (0,0) rectangle (\width,\height);

% Draw spiral
\draw[red, thick] (0.5*\width,0) -- (\width,0.5*\height/\nturns);
\draw[red, thick] (0,5.5*\height/\nturns) -- (0.5*\width,\height) ;

\foreach \i in {0,...,4} {
    \draw[red, thick]
        (0,{\i*\height/\nturns+0.5*\height/\nturns}) -- (\width,{(\i+1)*\height/\nturns+0.5*\height/\nturns});
}

\draw[blue, thick] (0,2.5*\height/\nturns) -- (0.5*\width,3*\height/\nturns);
\draw[blue, thick] (0.5*\width,2*\height/\nturns) -- (0.5*\width,3*\height/\nturns);
\draw[blue, thick] (0.5*\width,2*\height/\nturns) -- (\width,2.5*\height/\nturns);

\end{tikzpicture}
    \end{subfigure}
    \caption{Illustration of the proposed typewriter-like rearrangement in three and two dimensions.
    In the left panel, the \(y\)-periodic \(xy\)-plane is represented as a cylinder by gluing the two boundaries in the \(y\)-direction; the right panel shows the corresponding unwrapped planar representation.
    The red helix, shown as parallel red lines in the unwrapped view, represents the contact-line pattern for the diagonal orientation \(\arg\mathbf{k}=\tfrac{\pi}{4}\).
    The blue curve represents the contact-line pattern for the nearby orientation \(\arg\mathbf{k}=\arctan(\frac{5}{6})\).
    The blue curve follows the diagonal pattern over long portions and then switches locally to a neighboring periodic copy, as indicated by the short connecting segment.
    This localized switching illustrates the proposed typewriter-like geometric transition between the two stationary contact-line configurations.}
    \label{fig:typewritter}
\end{figure}
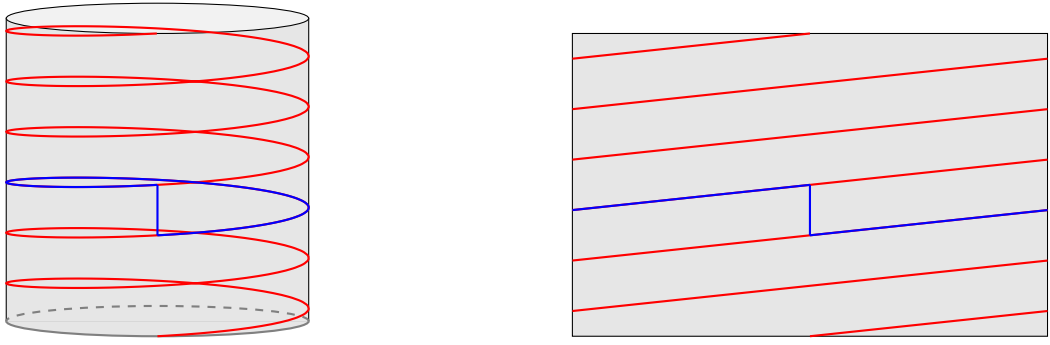

\subsection{Revisiting the anisotropy effect}

We now return to the anisotropy effect discussed in Section~\ref{s:TheoHystInt} and consider what the numerical observations above suggest about the corresponding wetted-region geometry.

First, the directional dependence of the CAH indicates that the pinning range depends on the contact-line orientation.
In particular, the CAH interval is wider near the diagonal directions, so the contact line can remain pinned over a larger range of contact angles in these directions.
In related settings with Dirichlet boundary conditions and without a volume constraint, a direct connection has been established between jump discontinuities in the hysteresis interval and contact-line faceting.
The corresponding connection is less clear in the volume-constrained setting considered here.
Nevertheless, our numerical results suggest that the strong directional dependence of the CAH can be associated with faceted stationary wetted regions.
For the present surface, the diagonal directions are distinguished by the symmetry of the surface, and the square-like stationary wetted region obtained with Surface Evolver has its sides aligned with these directions; see Figure~\ref{fig:droplet_sideview}.

Second, the localized structure of the contact lines described in Section~\ref{sec:interface_structure} constrains the geometry of the wetted set.
Together with the wider CAH interval near the diagonal directions, that behavior is consistent with a contact line that contains long portions aligned with the diagonal directions while changing through localized transitions.
This provides a geometric interpretation of the square-like wetted region shown in Figure~\ref{fig:droplet_sideview}.

\section{Discussion} \label{s:Disc}

We introduced a two-scale alternating method for approximating directional CAH on periodic rough surfaces.
The method couples an MBO diffusion-generated update near the contact line with a Fourier-based approximation of the minimal surface in the far region.
For an idealized reference iteration, we proved decay of an approximate interfacial energy.
This result, together with the heuristic energy-slope relation, provides an explanation for the large-scale monotone behavior of the macroscopic slope \(m\) observed in the numerical iteration.

For the representative surface \(z=\eps\sin\frac{x}{\eps}\sin\frac{y}{\eps}\), the computed CAH interval has a clear directional dependence and is widest along the diagonal direction \(\arg\mathbf{k}=\tfrac{\pi}{4}\); the remaining diagonal directions follow by the symmetry of the surface.
The corresponding far-region interfaces are planar with \(O(\eta)\) correction, consistent with the graph-based approximation used in the method.
Near the diagonal orientation, differences between stationary contact lines are spatially localized, giving a typewriter-like geometric interpretation of their transition.

Beyond characterizing a given surface, the computed interval constrains which macroscopic wetted regions that surface can support.
A stationary droplet must meet the solid at an apparent angle lying in \([\thetarec(\mathbf{k}),\thetaadv(\mathbf{k})]\) at every point of its contact line, so the directional interval is a necessary condition on the shape that contact line can take.
For the surface studied here the interval is widest at \(\arg\mathbf{k}=\tfrac{\pi}{4}\), and a square-like wetted set with sides along the diagonals is correspondingly admissible: the contact angles realized by the droplet computed with Surface Evolver in Figure~\ref{fig:droplet_sideview} lie inside the computed interval along the entire contact line.
The anisotropy of the hysteresis interval, computed once for a given micropattern, therefore predicts which droplet shapes that pattern can hold, which is the question that motivates the design of micropatterned surfaces.

These conclusions are subject to several limitations.
The computations use approximate far-field boundary conditions, a single doubly periodic surface, and a finite set of rational orientations.
The observed directional changes should therefore be interpreted as numerical evidence rather than as a general proof of discontinuity.
Extensions to additional roughness geometries, nonperiodic surfaces, and broader parameter ranges would help determine the robustness of the observed behavior.

We think there are several interesting extensions of this work.
The formulation of Section~\ref{sec:Intro} allows a spatially varying Young's angle \(\theta_Y(\mathbf{x})\), and the near-region MBO step accommodates it without modification, so the same two-scale decomposition should apply to chemically patterned surfaces, where the pinning mechanism is well understood but the anisotropy of the resulting interval is not.
It is a challenging extension to include the volume constraint to directly compute stationary droplets on rough surfaces.
On the theoretical side, the closed-interval conjecture of Section~\ref{s:TheoHystInt} remains open in the non-graphical setting considered here, as does the question of whether the sharp directional change observed at \(\tfrac{\pi}{4}\) is a genuine discontinuity; in related exactly solvable models both properties can be established \cite{feldman2021limit,feldman2019free}, which makes them plausible targets.
Finally, the alternating structure used here --- a finely resolved update near the contact line coupled to an explicitly solvable far-field problem --- is not specific to capillarity, and may be useful for other free boundary problems in which the boundary condition is generated at a scale far below that of the solution.

\section*{Acknowledgements}
Funding: B.O. and Z.L. were supported by the NSF DMS-2136198 and DMS-2513175. W.F. and Z.L. were supported by the NSF DMS-2407235.

\section*{Data availability}
The MATLAB implementation of the two-scale alternating method, together with the
computed contact angle hysteresis data reported in Section~\ref{s:NumExp}, is
openly available at \cite{TSAcode}.

\section*{Declaration of Generative AI and AI-assisted technologies in the writing process}
During the preparation of this work the authors used Claude (Anthropic) in order to edit and improve the language, presentation, and \LaTeX{} formatting of the manuscript. After using this tool, the authors reviewed and edited the content as needed and take full responsibility for the content of the publication. No generative AI tools were used in the development of the mathematical analysis, numerical methods, the implementation, or the generation and analysis of the reported results.

\bibliographystyle{elsarticle-num}
\bibliography{references}

\appendix

\section{Derivation of the MBO scheme for a capillary problem without volume constraint} \label{sec:Derivation_MBO}
Based on a variation of the MBO diffusion-generated method for a capillary problem on a rough surface \cite{Wang_2019}, we first provide a derivation for the choice of threshold in \eqref{iter:threshold}.

Suppose we solve problem~\eqref{prob:relaxed_wetting} using an iterative method.
In the $k$-th step, we have an approximated solution $(u^k_1,u^k_2)$.
The energy functional $\mathcal{E}^{\tau} (u_1,u_2; \Omega_{\mathrm{near}}) $ can be linearized near the point $(u^k_1,u^k_2)$ as in \eqref{equ:expansion_interface_energy}.
Then we minimize the linearized energy \(\mathcal{L}\) in \eqref{equ:expansion_interface_energy} and set $(u_1^{k+1},u_2^{k+1})$ as the solution of
\begin{align} \label{prob:linearized}
\min_{(u_1,u_2)\in \mathcal{K}} {\mathcal{L}}(u_1,u_2, u^k_1,u^k_2; \Omega_{\mathrm{near}}).
\end{align}
\begin{lemma} \label{lem:iter_sol}
Denote
$\phi=\frac{1}{\sqrt{\tau_1}} \,G_{\tau_1} \ast (u^k_2-u_1^k) -\frac{\cos{\theta_Y}}{\sqrt{\tau_2}} \,G_{\tau_2} \ast  \mathbf{1}_{S}-\frac{\cos{\theta_i}}{\sqrt{\tau_3}} \,G_{\tau_3} \ast  \mathbf{1}_{\mathcal{I}}$.
Let
$L^{k+1} = \{\mathbf{x}\in \Omega_{\mathrm{near}}\setminus (S\cup\mathcal{I})\colon  \phi<0\}$ and define $V^{k+1}=\Omega_{\mathrm{near}}\setminus (L^{k+1} \cup S\cup\mathcal{I})$.
Then $(u_1^{k+1},u_2^{k+1})= (\mathbf{1}_{L^{k+1}},\mathbf{1}_{V^{k+1}})$ is a solution to \eqref{prob:linearized}.
\end{lemma}

\begin{proof}
According to Lemma~\ref{lem:char_min}, it suffices to show that
$${\mathcal{L}} (u_1^{k+1},u_2^{k+1},u_1^k, u_2^k; \Omega_{\mathrm{near}})\leq {\mathcal{L}} (u_1, u_2,u_1^k, u_2^k; \Omega_{\mathrm{near}})$$
 for all $(u_1,u_2)\in \mathcal{B}$.
 By definition of $\mathcal{B}$, any $(u_1,u_2) \in \mathcal{B}$ satisfies $u_1=\mathbf{1}_{\hat{D}_1}$ and $u_2=\mathbf{1}_{\hat{D}_2}$ for some open sets $\hat{D}_1,\hat{D}_2 \subseteq {\Omega_{\mathrm{near}}}$ s.t. $\hat{D}_1\cap \hat{D}_2=\emptyset$ and $\hat{D}_1\cup \hat{D}_2=\Omega_{\mathrm{near}} \setminus (S\cup\mathcal{I})$.
 Let $A_1=\hat{D}_1 \setminus L^{k+1}=V^{k+1} \setminus \hat{D}_2$ and $A_2=\hat{D}_2 \setminus V^{k+1}=L^{k+1} \setminus \hat{D}_1$.
 By definition of $L^{k+1},V^{k+1}$ and the fact that $A_1\subset V^{k+1}$ and $A_2\subset L^{k+1}$,
 \[
 \phi(\mathbf{x})\geq 0, \quad\forall \mathbf{x}\in A_1
 \qquad \text{and} \qquad
 \phi(\mathbf{x})<0,\quad \forall \mathbf{x}\in A_2.
 \]
 Therefore, direct calculation yields
 \begin{align*}
 &{\mathcal{L}} (u_1^{k+1},u_2^{k+1},u_1^k, u_2^k; \Omega_{\mathrm{near}})- {\mathcal{L}} (u_1, u_2,u_1^k, u_2^k; \Omega_{\mathrm{near}})\\
 &=\sqrt{\pi}\int_{\Omega_{\mathrm{near}}} (u_1^{k+1}-u_1)\left(\frac{\gamma_{LV}}{\sqrt{\tau_1}} \,G_{\tau_1} \ast u^k_2 +\sum_{j=2,3}\frac{\gamma_{S_jL}}{\sqrt{\tau_j}}\,G_{\tau_j}\ast \mathbf{1}_{S_j}\right) + (u_2^{k+1}-u_2)\left(\frac{\gamma_{LV}}{\sqrt{\tau_1}}\,G_{\tau_1} \ast u^k_1 +\sum_{j=2,3}\frac{\gamma_{S_jV}}{\sqrt{\tau_j}} \,G_{\tau_j} \ast\mathbf{1}_{S_j}\right) \,d\mathbf{x} \\
 &=\sqrt{\pi}\int_{\Omega_{\mathrm{near}}} (u_1^{k+1}-u_1)\left(\frac{\gamma_{LV}}{\sqrt{\tau_1}} \,G_{\tau_1} \ast u^k_2 +\sum_{j=2,3}\frac{\gamma_{S_jL}}{\sqrt{\tau_j}}\,G_{\tau_j}\ast \mathbf{1}_{S_j}\right) - (u_1^{k+1}-u_1)\left(\frac{\gamma_{LV}}{\sqrt{\tau_1}}\,G_{\tau_1} \ast u^k_1 +\sum_{j=2,3}\frac{\gamma_{S_jV}}{\sqrt{\tau_j}} \,G_{\tau_j} \ast\mathbf{1}_{S_j}\right) \,d\mathbf{x} \\
 &=\sqrt{\pi}\gamma_{LV}\int_{\Omega_{\mathrm{near}}} (u_1^{k+1}-u_1)\phi \,d\mathbf{x} \\
 &=\sqrt{\pi}\gamma_{LV}\left(\int_{A_2} \phi \,d\mathbf{x}-\int_{A_1} \phi \,d\mathbf{x}\right)\\
 &\leq 0,
 \end{align*}
 where by definition of $\mathcal{K}$, $u^{k+1}_1+u^{k+1}_2=u_1+u_2=\mathbf{1}_{\Omega_{\mathrm{near}} \setminus (S\cup \mathcal{I})}$ hence $u^{k+1}_1-u_1=-(
u^{k+1}_2-u_2)$.
\end{proof}

Thus, we arrive at the specific choice of threshold $0$.
Next, we present the proof of Theorem~\ref{thm:stablity_wetMBO} to show that the MBO scheme for a capillary problem is stable in the sense that the energy $\mathcal{E}^{\tau}$ is always decreasing for all $\tau_1>0$.
We follow the idea of \cite[Theorem~2.1]{Wang_2019}.

\begin{proof}[Proof of Theorem~\ref{thm:stablity_wetMBO}] \label{proof:stability_analysis}
By Lemma~\ref{lem:iter_sol}, we have the  inequality,
$\mathcal{L}(u^{k}_1,u^{k}_2,u^{k}_1,u^{k}_2; \Omega_{\mathrm{near}})\geq \mathcal{L}(u^{k+1}_1,u^{k+1}_2,u^{k}_1,u^{k}_2; \Omega_{\mathrm{near}})$, where by definition of $\mathcal{L}$, the LHS and RHS could be further written as
\begin{align*}
  &\mathcal{E}^{\tau}(u^{k}_1,u^{k}_2; \Omega_{\mathrm{near}}) + \frac{\sqrt{\pi}\gamma_{LV}}{\sqrt{\tau_1}} \int_{\Omega_{\mathrm{near}}} u^{k}_1 \,G_{\tau_1} \ast u^{k}_2 \,d\mathbf{x}  \\
  &\geq \mathcal{E}^{\tau}(u^{k+1}_1,u^{k+1}_2; \Omega_{\mathrm{near}})+ \frac{\sqrt{\pi}\gamma_{LV}}{\sqrt{\tau_1}} \left[\int_{\Omega_{\mathrm{near}}} u^{k+1}_1 \,G_{\tau_1} \ast u^{k}_2 \,d\mathbf{x} \right.+\left.\int_{\Omega_{\mathrm{near}}} u^{k+1}_2 \,G_{\tau_1} \ast u^{k}_1 \,d\mathbf{x}-\int_{\Omega_{\mathrm{near}}} u^{k+1}_1 \,G_{\tau_1} \ast u^{k+1}_2 \,d\mathbf{x}\right].
\end{align*}
Rearrangement yields that
\begin{align*}
     \mathcal{E}^{\tau}(u^{k}_1,u^{k}_2; \Omega_{\mathrm{near}})
     \ - \ &
     \mathcal{E}^{\tau}(u^{k+1}_1,u^{k+1}_2; \Omega_{\mathrm{near}}) \\
    &\geq \frac{\sqrt{\pi}\gamma_{LV}}{\sqrt{\tau_1}} \left[\int_{\Omega_{\mathrm{near}}} u^{k+1}_1 \,G_{\tau_1} \ast u^{k}_2 \,d\mathbf{x} -\int_{\Omega_{\mathrm{near}}} u^{k}_1 \,G_{\tau_1} \ast u^{k}_2 \,d\mathbf{x}\right.
    +\left. \int_{\Omega_{\mathrm{near}}} u^{k+1}_2 \,G_{\tau_1} \ast u^{k}_1 \,d\mathbf{x}-\int_{\Omega_{\mathrm{near}}} u^{k+1}_1 \,G_{\tau_1} \ast u^{k+1}_2 \,d\mathbf{x}\right] \\ &=-\frac{\sqrt{\pi}\gamma_{LV}}{\sqrt{\tau_1}}\int_{\Omega_{\mathrm{near}}} (u^{k+1}_1-u^k_1)\,G_{\tau_1} \ast(u_2^{k+1}-u^k_2) \,d\mathbf{x}.
\end{align*}
By definition of $\mathcal{K}$,
$u^{k+1}_1+u^{k+1}_2=u^{k}_1+u^{k}_2  =\mathbf{1}_{\Omega_{\mathrm{near}} \setminus (S\cup \mathcal{I})}$,
and
\begin{align*}
    -\frac{\sqrt{\pi}\gamma_{LV}}{\sqrt{\tau_1}}\int_{\Omega_{\mathrm{near}}} (u^{k+1}_1-u^k_1)\,G_{\tau_1} \ast(u_2^{k+1}-u^k_2) \,d\mathbf{x}
    \ = \
    \frac{\sqrt{\pi}\gamma_{LV}}{\sqrt{\tau_1}}\int_{\Omega_{\mathrm{near}}} (u^{k+1}_1-u^k_1)\,G_{\tau_1} \ast(u_1^{k+1}-u^k_1) \,d\mathbf{x}
    \ \geq \  0.
\end{align*}
Therefore,
$ \mathcal{E}^{\tau}(u^{k}_1,u^{k}_2; \Omega_{\mathrm{near}})-\mathcal{E}^{\tau}(u^{k+1}_1,u^{k+1}_2; \Omega_{\mathrm{near}}) \geq 0$,
as desired.
\end{proof}

\section{Justification of the linearization far from the boundary}
\label{sec:Justification_linearization}

In this section, we provide justifications for the graphicality assumption of the liquid--vapor interface away from the contact line in Section~\ref{sec:condition_over_far}.

\begin{proposition}
\label{prop:regularity-of-planelikes}
Suppose that \(L\) is a surface-area minimizer in \(\mathbb{R}^3_+\) which is sandwiched between two parallel planes of slope \(m\), corresponding to a contact angle
\(\theta\in(0,\pi)\), i.e.,
\[
    \{x<mz-C_0\}
    \subset L
    \subset
    \{x<mz+C_0\}
    \qquad \hbox{in } z\geq 0,
\]
and suppose that \(L\) is \(T\)-periodic in the \(y\)-variable. Then, for \(z_\ast=C_2T\), \(L\) is a subgraph in the \((y,z)\) variables of the form
\[
    L\cap\{z\geq z_\ast\}
    =
    \{(x,y,z):x<X(y,z)\}
    \qquad \hbox{in } z\geq z_\ast.
\]
The function \(X\) is smooth, \(T\)-periodic in the \(y\)-variable, and solves the graphical minimal-surface equation
\[
    (1+X_y^2)X_{zz}
    -2X_yX_zX_{yz}
    +(1+X_z^2)X_{yy}
    =0
    \qquad \hbox{in } z\geq z_\ast.
\]
Moreover, there exists \(s\in[-C_0,C_0]\) such that
\[
    \left|X(y,z)-(mz+s)\right|
    \leq
    C_1e^{-c(z-z_\ast)/T}
    \qquad \hbox{in } z\geq z_\ast.
\]
\end{proposition}

Note that Proposition~\ref{prop:regularity-of-planelikes} represents the interface by the full graph \(X=X(y,z)\), which is denoted by \(x=x(y,z)\) in Section~\ref{sec:condition_over_far}. The function
\[
    v(y,z):=X(y,z)-mz
\]
used below denotes the bounded deviation of the interface from the reference plane. The plane \(mz\) differs from the plane \(m(z-R)\) used in Section~\ref{sec:condition_over_far} only by an additive constant.

The outline of the proof is as follows. First, we apply the De Giorgi \(\eps\)-regularity theorem for flat codimension-one minimal surfaces to show that \(L\) is graphical for
\(z\geq z_\ast\). 
A particular consequence is that \(X\) solves the graphical
minimal-surface equation.
The deviation \(v=X-mz\) then satisfies a linear divergence-form equation with \(T\)-periodic coefficients, so we can apply a general result on the exponential boundary layer of solutions of elliptic equations in a periodic half-space.

First, we recall De Giorgi's \(\eps\)-regularity theorem. A proof of this theorem can be found in
\cite[Theorem 5.1]{Savin_2009}.

\begin{theorem}[De Giorgi]
\label{t.dg}
If \(L\) is a perimeter minimizer in
\(B_1\subset\mathbb{R}^d\), \(d\geq2\), and
\[
    \{x_d\leq-\delta\}
    \subset L
    \subset
    \{x_d\leq\delta\}
    \qquad \hbox{in } B_1,
\]
with \(\delta\leq\delta_0(d)\), then \(L\) is a smooth subgraph in
\(B_{1/2}\), i.e.,
\[
    L\cap B_{1/2}
    =
    \{(x',x_d)\in B_{1/2}:x_d\leq w(x')\},
\]
where \(w\) is \(C^\infty\) and solves the graphical minimal-surface equation
\[
    -\nabla\cdot
    \left(
        \frac{\nabla w}{\sqrt{1+|\nabla w|^2}}
    \right)
    =0
    \qquad \hbox{in } B_{1/2}',
\]
and
\(
    |\nabla w|\leq C(d)\delta.
\)
\end{theorem}

Next, we describe the boundary layers of elliptic equations in a half-space of \(\mathbb{R}^d\). Consider the problem
\begin{equation}
\label{equ:half-space-layer}
    -\nabla\cdot\bigl(A(x)\nabla v\bigr)=0
    \qquad \hbox{in } \{x_d>0\},
    \qquad
    v(x)=g(x)
    \qquad \hbox{on } \{x_d=0\}.
\end{equation}
We assume that:
\begin{enumerate}
    \item
    \(A\) is a symmetric, uniformly elliptic matrix field; that is,
    there exists \(\lambda\in(0,1)\) such that
    \[
        \lambda|\xi|^2
        \leq
        \xi\cdot A(x)\xi
        \leq
        \lambda^{-1}|\xi|^2
    \]
    for all \(x\in\mathbb{R}^d_+\) and
    \(\xi\in\mathbb{R}^d\).

    \item
    Both \(A(x)\) and \(g(x)\) are $\mathbb{Z}^{d-1} \times \{0\}$ periodic;
    that is, for any \(m'\in\mathbb{Z}^{d-1}\) and
    \((x',x_d)\in\mathbb{R}^d_+\),
    \[
        A(x'+m',x_d)=A(x',x_d),
        \qquad
        g(x'+m')=g(x').
    \]

    \item
    \(A\) and \(g\) are smooth and bounded.
\end{enumerate}

Under these hypotheses, we obtain the following information on solutions of \eqref{equ:half-space-layer}.

\begin{lemma}
\label{l.exponential-bdry-layer}
Let \(A\) and \(g\) satisfy the above hypotheses. There exists a unique bounded solution \(v\) of
\eqref{equ:half-space-layer}, and there is \(C\geq1\), depending only on \(\lambda\) and \(d\), such that \(
s:=\lim_{x_d\to+\infty}v(x',x_d)\) exists and 
\[ 
|v(x)-s|  \leq Ce^{-C^{-1}x_d}.
\]
\end{lemma}

A proof of this result can be found in
\cite[Lemma 5.2]{Feldman_2019}. We now return to Proposition~\ref{prop:regularity-of-planelikes}.

\begin{proof}[Proof of Proposition~\ref{prop:regularity-of-planelikes}]
We combine the two results above as follows. In the following proof, the constants \(C>1>c>0\) may change from line to line and may depend on \(d\), \(m\), and \(C_0\), but not on the other parameters.

Fix \(p\in\partial L\) at height \(z_p\). Since \(B_{z_p/2}(p)\) is contained in the open half-space, \(L\) is a perimeter minimizer in \(B_{z_p/2}(p)\).
The flatness hypothesis confines
\(\partial L\cap B_{z_p/2}(p)\) to a slab of width at most \(2C_0\) centered on the plane
\(\{x=mz\}.\) Rescaling \(B_{z_p/2}(p)\) to the unit ball, the flatness is at most \(2C_0/z_p\), which is below the threshold \(\delta_0\) of
Theorem~\ref{t.dg} once \(z_p\geq C_2T\), with \(C_2\) sufficiently large depending on \(C_0/T\).

Theorem~\ref{t.dg} then makes \(L\cap B_{z_p/4}(p)\) a smooth subgraph over the plane \(\{x=mz\}\), with slope at most \(CC_0/z_p\).
Choosing \(C_2\) larger if necessary, we can also conclude that \(L\cap B_{z_p/4}(p)\) is a smooth subgraph in the \((y,z)\) variables. Since \(z_p\geq z_\ast=C_2T\) was arbitrary, and the local graph parametrizations agree on overlaps, we conclude that \(L\cap\{z\geq z_\ast\}\) is globally a smooth subgraph in the \((y,z)\) coordinates.

In other words,
\(
    L\cap\{z\geq z_\ast\}
    =
    \{(x,y,z):x<X(y,z)\},
\)
where \(X\) is smooth and \(T\)-periodic in \(y\). The trapping assumption gives
\(
    |X(y,z)-mz|\leq C_0.
\)
Moreover, \(X\) solves the graphical minimal-surface equation, whose divergence form is
\[
    -\nabla\cdot
    \left(
        \frac{\nabla X}{\sqrt{1+|\nabla X|^2}}
    \right)
    =0,
\]
where the gradient is taken with respect to \((y,z)\).

Define
\(
    v(y,z):=X(y,z)-mz.
\)
Then \(v\) is bounded and \(T\)-periodic in \(y\).
Let
\[
    F(q):=\frac{q}{\sqrt{1+|q|^2}},
    \qquad
    q_0:=(0,m).
\]
Since
\(
    \nabla X=q_0+\nabla v
\)
and \(F(q_0)\) is constant, we obtain
\[
\begin{aligned}
    0
    &=
    -\nabla\cdot
    \bigl(F(q_0+\nabla v)-F(q_0)\bigr) \\
    &=
    -\nabla\cdot\bigl(A(y,z)\nabla v\bigr),
\end{aligned}
\]
where
\[
    A(y,z)
    :=
    \int_0^1
    DF\bigl(q_0+t\nabla v(y,z)\bigr)\,dt.
\]
The matrix field \(A\) is smooth, uniformly elliptic, and
\(T\)-periodic in \(y\).

Thus, \(v\) is a bounded solution of a linear equation of the form \eqref{equ:half-space-layer} in the half-plane \(\{z>z_\ast\}\), with boundary data
\(
    g(y)=v(y,z_\ast).
\)
Applying a translated and rescaled version of Lemma~\ref{l.exponential-bdry-layer}, we obtain a constant \(s\) such that
\[
    |v(y,z)-s|
    \leq
    Ce^{-c(z-z_\ast)/T}.
\]
Consequently,
\(
    |X(y,z)-(mz+s)|
    \leq
    Ce^{-c(z-z_\ast)/T}.
\)
Finally, the bound \(|s|\leq C_0\) follows from
\[
    |v(y,z)|=|X(y,z)-mz|\leq C_0.
\]
\end{proof}

\end{document}